\documentclass[11pt,reqno]{article}

\usepackage{appendix, amsmath, amssymb, amsthm, color, fancyhdr, graphicx, latexsym, marvosym, mathtools, multirow, psfrag, yhmath}

\numberwithin{equation}{section}
\newtheorem{thm}{Theorem}[section]
\newtheorem{defn}[thm]{Definition}
\newtheorem{prop}[thm]{Proposition}
\newtheorem{lemma}[thm]{Lemma}
\newtheorem{remark}[thm]{Remark}

\newcommand{\N}{\mathbb{N}}
\newcommand{\Z}{\mathbb{Z}}

\newcommand{\R}{\mathbb{R}}
\newcommand{\T}{\mathbb{T}}
\newcommand{\C}{\mathbb{C}}

\newcommand{\X}{\mathcal{X}}
\newcommand{\Y}{\mathcal{Y}}

\newcommand{\FF}{\mathcal{F}}
\newcommand{\HH}{\mathcal{H}}

\DeclareMathOperator{\supp}{supp}

\date{}

\begin{document}

\title{Non-homogeneous Problems for Nonlinear Schr\"odinger Equations in a Strip Domain}

\author{Yu Ran$^1$ and Shu-Ming Sun$^2$ \\
{\small $^1$Department of Mathematics} \\
{\small China Jiliang University} \\
{\small 258 Xueyuan Rd., Hangzhou 310018, China} \\
{\small Email: ryan\_ranyu@live.com} \\
{\small $^2$Department of Mathematics} \\
{\small Virginia Polytechnic Institute and State University} \\
{\small Blacksburg, Virginia 24061, USA} \\
{\small Email: sun@math.vt.edu}
}

\baselineskip=15pt

\maketitle

\begin{abstract}
    This paper studies the initial-boundary-value problem (IBVP) of a nonlinear Schr\"odinger equation posed on a strip domain $\R\times[0,1]$ with non-homogeneous Dirichlet boundary conditions. For any $s\ge0$, if the initial data $\varphi(x,y)$ is in Sobolev space $H^s(\R\times[0,1])$ and the boundary data $h(x,t)$ is in
    $$
    {\cal H}^s (\R ) = \left \{ h (x, t) \in L^2 ( \R^2 ) \  \big | \  ( 1 + |\lambda | + |\xi|)^{\frac12} ( 1+ |\lambda | + |\xi |^2 )^{\frac{s}{2}}\hat h ( \lambda, \xi  ) \in L^2 (\R^2 ) \right \}
    $$
    where $\hat h $ is the Fourier transform of $h$ with respect to $t$ and $ x$, the local well-posedness of the IBVP in $C([0,T]; H^s(\R \times [0,1]))$ is proved. The global well-posedness is also obtained for $s = 1$.
    The basic idea used here relies on the derivation of an integral operator for the non-homogeneous boundary data and the proof of the series version of Strichartz's estimates for this operator. After the problem is transformed to finding a fixed point of an integral operator, the contraction mapping argument then yields a fixed point using the Strichartz's estimates for initial and boundary operators. The global well-posedness is proved using {\it a-priori} estimates of the solutions.

\end{abstract}

\allowdisplaybreaks

\section{Introduction}\label{NLSE_IB_2d_stp_Intro_sec}

\setcounter{equation}{0}

    In this paper, we focus on the initial-boundary-value problem (IBVP) of a nonlinear Schr\"odinger (NLS) equation on a strip domain $\R \times [0,1]$,
    \begin{equation}\label{1.1} 
        \left\{
        \begin{array}{ll}
            i u_t + u_{xx} + u_{yy} + \lambda |u|^{p-2}u = 0 \, \text{, } & (x,y,t) \in \R \times (0,1) \times \R\, ,  \\
            u(x,y,0) = \varphi(x,y) \, \text{, } & (x,y) \in \R \times [0,1] \,   ,\\
            u(x,0,t)=h_1(x,t) \, \text{, } u(x,1,t)=h_2(x,t) \, \text{, } & (x,t) \in \R \times \R \, ,
        \end{array}
        \right.
    \end{equation}
    where $\lambda\in\R$ and $p\ge3$ is a constant (although our attention will focus on the case of $p\ge3$, many of the results in this paper are valid for $p>2$).

    During last 40 years, the NLS equations have been used as model equations to many physical applications and become an essential part in the field of physics, mechanics and mathematics whose solutions describe the wave propagation spreading out in space as they evolve in time.
    Here, we mainly consider the mathematical perspective of the IBVP (\ref{1.1}) and concentrate on the \emph{well-posedness} of (\ref{1.1}) in the Sobolev space $H^s(\R\times(0,1))$.

The mathematical study of NLS equations was first accomplished for a pure initial value problem (IVP), i.e., for $(x,y)\in\R^2$
\begin{equation} \label{1.2}
        iu_t +  u_{xx} + u_{yy} + \lambda |u|^{p-2}u = 0, \ u(x,y,0) = \varphi (x,y), \qquad (x, y)  \in \R^2, \ t \in \R \, ,
\end{equation}
where $(x, y)$ can be replaced by $\vec x\in \R^n$ or $\vec x \in \T^n$ with $\T^n$ an $n$-dimensional torus. Numerous new ideas, methods and mathematical tools for studying the IVPs of the NLS equations have  been developed and the most of literatures on this subject have been concerned with the basic well-posedness question, i.e., the existence, uniqueness, and continuous dependence of solutions with respect to the initial data in the corresponding Sobolev spaces. In particular, significant progress has been made recently for the well-posedness of the problem with low regularity of solutions, which was pioneered by Bourgain \cite{Bourgain_Exp_NLSE, Bourgain_FrRc_NLSE} for his study on the NLS equations in periodic domains. He
developed a method with harmonic analysis analogous to Strichartz's estimates for studying this problem and obtained its global well-posedness \cite{Bourgain_FrRc_NLSE}. The research for the IVPs in other domains with some periodic (in spatial
variables) structures can be found in
\cite{Bourgain_NLSE_Periodic_IM, Ion_NLS_GLB_RT3, Ion_NLS_T3, Takaoka_NLSE_2d_Periodic} and the references therein. In particular, Takaoka and Tzvetkov \cite{Takaoka_NLSE_2d_Periodic} studied a two-dimensional IBVP for the solution with $(x, y) \in \R \times \T$ and proved that the equation
$$
i \partial_t u + u_{xx}+ u_{yy} + |u|^{p-2}u = 0\, ,\qquad u(x, y, 0) = \varphi(x, y)
$$
is globally well-posed for $2<p<4$ on $\R\times\T$ with $\varphi\in L^2(\R\times \T)$ and is globally well-posed for $p=4$ with $\|\varphi\|_{L^2(\R\times\T)}$ sufficiently small.
A very small sample of other excellent papers on the IVPs of NLS equations can be found in \cite{Bourgain_NLSE_Glb, Cazenave_NLSE, Caz_CD_frac, Caz_Wei_NLSE_CP_crt_Hs, Fang_NLSE_CP_Hs, Ginibre_CP_G, Kato_NLSE, Kato_NLSE_lec, Kato_NLSE_UC_Hs, TsutsumiY_NLSE_L2} and the references therein, where the methods of nonlinear functional
analysis and harmonic analysis have been used and Strichartz's estimate (see \cite{Stri_Frr_wave}) plays an important role in the study. The book by Cazenave \cite{Cazenave_NLSE} is a terrific reference into the literatures on this subject.

In contrast to the IVPs of the NLS equations, a less extent of progress to the study of the IBVP (\ref{1.1})
with non-homogeneous boundary conditions was made in a number of literatures (e.g., see \cite{Aud_Disp, Aud_se_nbvp, Gallouet_NLSE, Bu_NLSE_IB_Cubic_PDE, Bu_NLSE_IB_WP, Bu_NLSE_IB_Rbn, Bu_NLSE_Cubic_Gen, Bu_NLSE_IB_WP_DC_BU, Bu_NLSE_IB_IH_Diri, Holmer_NLSE_1d, Kamvissis_IST_HL, Ran_NLSE_2d_uhp, Bu_NLSE_IB_IH, TsutsumiM_NLSE_IB_H0_2d, TsutsumiM_NLSE_IB_GLB_Nd, TsutsumiY_NLSE_H2} and the references therein), using the method of nonlinear functional analysis (harmonic analysis). The well-posedness of one-dimensional IBVP (\ref{1.1}) over a finite interval $[0,L]$ with solutions in $C([0,T]; H^s[0,L])$ for $s\ge0$ has been addressed in \cite{Sun_NLSE_1d} using boundary integral operator method. It showed that
\begin{equation}\label{1.3} 
    i u_t + u_{xx} + \lambda |u|^{p-2}u = 0 \, , u(x,0)=\varphi(x) \, , \, u(0,t)=h_1(t) \, , \, u(L,t)=h_2(t)
\end{equation}
is globally well-posed when $\varphi\in H^s$ and $h_1$ and $h_2$ are both in $H^{\frac{s+1}{2}}$ with $3\le p \le \frac{6-4s}{1-2s}$ if $0\le s < 1/2$ or $3\le p < \infty$ if $s>1/2$.

In this paper, we study the IBVP (\ref{1.1}) for its local and global well-posedness in $H^s (\R \times [0,1]) $ for $s \ge 0$ with appropriate initial and boundary conditions by applying the similar strategy and method in \cite{Sun_NLSE_1d, Ran_NLSE_2d_uhp} and try to level the results up to those for the IBVP (\ref{1.3}). In order to have the solution of (\ref{1.1}) in the space $C([0,T]; H^s (\R\times[0,1]))$ with $s\ge0$, the initial value $\varphi(x,y)$ is chosen in $H^s(\R \times [0,1])$, but the choice of the function spaces for the boundary data $h(x,t)$ needs some discussion. If we let $\widehat{w}$ stand for the Fourier transform of $w(x, t) $ with respect to both $t$ and $x$, it has been shown from the initial value problem (\ref{1.2}) in \cite{Ran_NLSE_2d_uhp} that the optimal space for the boundary data $h(x, t)$ is
$$
    ( 1 + |\lambda - \xi^2|)^{1/4}  (1+ |\lambda| +|\xi|^2)^{\frac{s}{2}} \hat h ( \xi , \lambda ) \in L^2 (\R^2 )\, ,
$$
where $\hat h $ is the Fourier transform of $h$ with respect to $t$ and $ x$.
Here, we may use a slightly more restrictive space
$$( 1 + |\lambda|+ |\xi|^2)^{\frac{2s+1}{4}} \hat h ( \xi , \lambda ) \in L^2 (\R^2 )\, ,$$
for $h_1$ and $h_2$ in (\ref{1.1}) if they are extended to $\R^2$. However, as discussed in \cite{Sun_NLSE_1d}, for the NLS equations posed in a finite domain, one can show that it is necessary to impose more regularity conditions on $h (x , t)$ with respect to $t$ in order for the solutions to be in $C([0,T]; H^s)$. Based upon the function space used in \cite{Sun_NLSE_1d}, we let
    \begin{align*}
        {\cal H}^s (\R ) = & \left \{ w (x, t) \in L^2 ( \R^2 ) \  | \  ( 1 + |\lambda | + |\xi|)^{\frac12} ( 1+ |\lambda | + |\xi |^2 )^{\frac{s}{2}}\hat w ( \lambda, \xi  ) \in L^2 (\R^2 )\right \}\\
        & \mbox{with } \  \| w\|_{{\cal H}^s (\R)} = \left \| ( 1 + |\lambda | + |\xi|)^{\frac12} ( 1+ |\lambda | + |\xi |^2 )^{\frac{s}{2}}\hat w ( \lambda, \xi  ) \right \|_{ L^2 (\R^2 )} \, ,
    \end{align*}
    and assume that $h_1 (x, t)$ and $h_2 (x,t ) $ belong to the space
    \begin{align*}
        {\cal H}^s (0, T) :=  & \left \{ h (x, t ) \in  L^2 ( \R \times [0, T] ) \  |  \  \mbox{$\exists $ $w\in {\cal H }^s (\R )$ as an extenion of $h$ in $\R^2$  } \right \}  \,  \nonumber \\
        & \mbox{with } \  \| h\|_{{\cal H}^s (0, T)} = \inf \left \{ \| w \|_{{\cal H}^s (\R)} \ | \ w \mbox{ is an extnsion of $h$  in } \R^2\right \}   \, ,\label{NLSE_IB_2d_stp_bdyspT}
    \end{align*}
    so that we can study the solutions of (\ref{1.1}) in $ C([0, T];H^s(\R\times[0,1]))$. Note that we need $1/4$ more derivative on $t$. It will be shown that for $s=0$, the half-derivative on $t$ for $h_1$ and $h_2$ is optimal.

    To study the solutions in fractional Sobolev spaces, the following definition of the well-posedness for the IBVP (\ref{1.1}) is natural.
    \begin{defn}[Well-posedness] \label{WP_defn}
        For any given $s \in \R$ and $T > 0$, the IBVP (\ref{1.1}) is locally well-posed in $H^s (\R \times [0,1])$ if for any constant $r>0$ there is a $T^*\in(0,T]$ such that for $\varphi \in H^s (\R \times [0,1])$ and $h_j \in {\cal H}^s (0,T)$, $j=1$, $2$, satisfying
        $$\|\varphi\|_{H^s (\R \times [0,1])} + \sum_{j=1,2} \|h_j\|_{{\cal H}^s (0,T)} \le r $$
        and some compatibility conditions, the IBVP (\ref{1.1}) has a unique solution in $C([0,T^*]; H^s (\R \times [0,1]))$, which continuously depends upon $(\varphi, h)$ in the corresponding spaces. If $T^*$ can be chosen independently of $r$, then the IBVP (\ref{1.1}) is globally well-posed.
    \end{defn}

    For the low regularity (small $s$), we need clarify the meaning of solutions of (\ref{1.1}) with the initial and boundary conditions:
    \begin{defn} \label{gen_soln}
        For $s<2$ and $T>0$, let $\varphi \in H^s (\R \times [0,1])$ and $h_j \in{\cal H}^s (0,T)$, $j=1$, $2$. Then $u(x,y,t)$ is called a mild solution of the IBVP (\ref{1.1}) if there is a sequence
        $$u_n \in C ([0,T]; H^2 (\R \times [0,1])) \cap C^1 ([0, T]; L^2 (\R \times [0,1]))\, ,\quad n = 1,2, \dots , $$
        satisfying the following properties:
        \begin{enumerate}
            \item $u_n$ is a solution of (\ref{1.1}) in $L^ 2 (\R \times [0,1])$ for $0 \le t \le T$;
            \item $u_n \rightarrow u $ in $C ([0,T]; H^s (\R \times [0,1]))$ as $n \rightarrow \infty$;
            \item $u_n(x,y,0) \rightarrow \varphi(x,y)$ in $H^s (\R \times [0,1])$ and $u_n(x,0,t) \rightarrow h_1(x,t)$ in ${\cal H}^s (0,T)$, $u_n(x,1,t) \rightarrow h_2(x,t)$ in ${\cal H}^s (0,T)$,  as $n \rightarrow \infty$.
        \end{enumerate}
    \end{defn}

    To study the well-posedness of the IBVP (\ref{1.1}), we denote $W^{s,r} (\R \times [0,1])$ as the classical Sobolev space in $L^r$-norm. Thus, $W^{s,2} = H^s$. Moreover, the following concept is required for later use:
    \begin{defn}
        $X^{\sigma,s}$ denotes a Bourgain space over $\R \times \T\times \R$ by
        \begin{equation*}\label{Bourgain_RT_defn} 
            X^{\sigma,s} = \left\{u \left| \, \|u\|_{X^{\sigma,s}} = \left\|e^{it\Delta}u \right\|_{H^{\sigma}_t(H^s_{x, y})} < \infty \right. \right\} \, ,
        \end{equation*}
        where $\T$ stands for the torus $\R/\Z$ and $e^{-it\Delta}$ is the Schr\"odinger operator defined by
        \begin{equation*}\label{NLSE_IB_2d_stp_loc_lin_op}
            e^{-it\Delta}\phi = \int_{-\infty}^{\infty} \sum_{n\in\Z} e^{-i\left(\xi^2+n^2\right)t+i(x\xi+yn)} \widehat{\phi}(\xi,n) \, d\xi \, .
        \end{equation*}
        Here, the Bourgain norm can also be written by
        \begin{equation*}\label{Bourgain_RT_norm}
            \|u\|_{X^{\sigma,s}} = \left\{\int_{-\infty}^{\infty} \int_{-\infty}^{\infty} \sum_{n\in\Z} \left(1+|\xi|+|n|\right)^{2s} \left(1+\left|\lambda+\xi^2+n^2\right|\right)^{2\sigma} \left|\widehat{u}(\xi,n,\lambda)\right|^2 \, d\xi \, d\lambda \right\}^{\frac{1}{2}}\, .
        \end{equation*}
    \end{defn}
    The Bourgain space restricted on a finite time interval $[0, T]$ is defined by
    \begin{align*}
        X^{\sigma,s}(0, T) :=  & \left\{ u(x,t) \in  L^2(\R \times \T \times [0,T] ) \  |  \  \mbox{$\exists \, \tilde{u}\in X^{\sigma,s}$ as an extenion of $u$ in $\R \times \T \times \R$} \right\}  \,  \nonumber \\
        & \mbox{with } \  \|u\|_{X^{\sigma,s}(0, T)} = \inf \left\{ \|\tilde{u}\|_{X^{\sigma,s}} \ | \ \tilde{u} \mbox{ is an extnsion of $u$  in } \R \times \T \times \R \right\} \, ,
    \end{align*}
    and sometimes we may use notation $X^{\sigma,s}$ for $X^{\sigma,s}(0, T)$ when no confusion arises.

The main result in this paper can be stated as follows.
\begin{thm}[Main Theorem] \label{thm_main} For given $s\ge0$, $T>0$ and $\mu>0$, assume $\varphi \in H^s (\R \times [0,1])$ and $h_j \in {\cal H}^s(0,T)$, $j=1$, $2$  satisfying
        $$\| \varphi\|_{H^s (\R \times [0,1])} + \sum_{j=1,2} \|h_j\|_{{\cal H}^s (0,T)} \le \mu \, $$
        and some natural compatibility conditions.
        \begin{enumerate}
        \item If $0 \le s < \frac{1}{2}$ with $3 \le p \le 4$ or $\frac{1}{2} \le s <1$ with $3 \le p \le \frac{3-2s}{1-s}$ or $s=1$ with $3 \le p < \infty$, (\ref{1.1}) is conditionally locally well-posed in $C\left(\R; H^s(\R \times (0,1)) \right)$ with
            \begin{equation}\label{1.4} 
                u\in L^r_t \left([0,T]; \, W^{s,r}_{xy}(\R \times [0,1])\right) \, \text{for } r\in[2,4]\, .
            \end{equation}
        \item If $s > 1$ and $3 \le p < \infty$ satisfying that
         \begin{equation}
         \mbox{$p$ is even, or if $p$ is not even, either $s \le p-1$  for $s\in\Z$ or $[s]\le p-2$ for $s\notin\Z$, }\label{1.4.1}
         \end{equation}
         (\ref{1.1}) is unconditionally locally well-posed. Here $[s]$ is the largest integer less than $s$.
        \item If $0\le s\le 1$ is given, the condition (\ref{1.4}) can be removed, and therefore the  well-posedness is unconditional.
        \item (\ref{1.1}) is globally well-posed in $H^1(\R \times [0,1])$ for $\varphi\in H^1(\R \times [0,1])$ and $$h_j \in H^1_{loc} \left(\R; \, L^2(\R)\right) \cap L^2_{loc} \left(\R; \, H^2(\R)\right), \, j=1, \, 2 \, , $$ if either $p\ge3$ for $\lambda<0$ or $p=3$ for $\lambda>0$.
        \end{enumerate}
    \end{thm}

The following remarks give some expansion about the statement of the theorem.

\begin{remark} As mentioned in the first paragraph, the results stated in (i) and (ii) of Theorem 1.4 are also valid for $ 2 < p < 3$.
\end{remark}

\begin{remark}\label{rem1.6}
The compatibility conditions stated in the theorem deserve more discussions. When $ t= 0 $ and $y=0,1$, by the trace theorems, the initial condition and boundary conditions imply that the compatibility conditions
$\varphi (x, 0) = h_1 ( x, 0 )$ and $\varphi(x, 1) = h_2 ( x, 0) $ must be satisfied if $ 1/2 < s < 5/2$ (note that in this case, $h_i(x, 0), i =1,2$ are defined if $h_i (x, t)\in {\cal H}^s(0, T) ,i=1,2 $). For $5/2 < s < 9/2$, more compatibility conditions, which are derived from the equation, have to be satisfied, i.e.,
$$
\left ( \varphi_{xx} + \varphi_{yy} + |\varphi |^{p-2} \varphi \right )\big |_{y =j} = -i h_{j+1, t} (x, 0) \qquad \mbox{for}\quad j= 0,1\, .
$$
Even more conditions have to be imposed if $s > 9/2$. Detailed discussions on general compatibility conditions for KdV equations or parabolic equations can be found in \cite{bsz-1, LM1972-2}.
\end{remark}

    Here, we note that for $0 \le s \le 1$, the local well-posedness result presented in part $(i)$ of the theorem is conditional since (\ref{1.4}) is required to ensure the uniqueness. However, we can remove the condition and the corresponding well-posedness is called unconditional. By \cite{Sun_CUC_Evo}, \emph{Theorem~\ref{thm_main}} shows that the solution obtained is a mild solution defined in Definition \ref{gen_soln}.

    To consider the local well-posedness for (\ref{1.1}), we use the methods introduced in \cite{Sun_NLSE_1d} for studying the IBVPs of one-dimensional NLS equations posed on a finite interval and in \cite{Bourgain_FrRc_NLSE, Takaoka_NLSE_2d_Periodic} for the study of the NLS equations over a {mixed} region with torus. First, the IBVP (\ref{1.1})  is decomposed into three parts (see \cite{Ran_NLSE_2d_uhp}): one in $\R \times [0,1]$ with initial condition and linear Schr\"odinger equation, one in $\R \times [0,1]$ with homogeneous initial condition and non-homogeneous linear equation, and the last one in $\R \times [0,1]$ with non-homogeneous boundary data, homogeneous initial condition and homogenous linear equation. The key step is to study the following linear non-homogeneous boundary value problem:
    \begin{equation} \label{1.5}
        \left\{
        \begin{array}{ll}
            i u_t + u_{xx} + u_{yy} = 0\, , & (x,y,t) \in \R \times (0,1) \times (0,T)\, , \\ [.08in]
            u(x,y,0) = 0 \, , & \\ [.08in]
            u(x,0,t) = h_1(x,t)\, , \quad u(x,1,t) = h_2(x,t)\, . &
        \end{array}
        \right. \,
    \end{equation}
    We apply the Fourier series to the linear equation with respect to $y$ to explicitly formulate the solution in terms of the boundary data $h_j$, $j=1$, $2$, called the boundary integral operator, $$u(x,y,t):= W_{b}[h_1,h_2](x,y,t) \, .$$
    For the Strichartz's estimates of the operator, we apply the work in \cite{Bourgain_FrRc_NLSE} and \cite{Takaoka_NLSE_2d_Periodic} and show that for any given $s\ge0$, $T>0$ and $h_1, h_2 \in {\cal H}^s(0,T)$, the IBVP (\ref{1.5})  admits a unique solution $u\in C([0,T]; H^s(\R \times [0,1])$ and
    \begin{equation*} \label{NLSE_IB_2d_stp_loc_lin_bdy_est_Intro}
        \|u\|_{L^r ((0,T); W^{s,r}(\R \times [0,1]))} \le C\left ( \| h_1 \|_{{\cal H}^s (0,T)} + \| h_2 \|_{{\cal H}^s (0,T)}\right )
    \end{equation*}
    for any $r\in[2,4]$. The basic plot of the proof for the remaining arguments, especially local well-posedness, is to derive an equivalent integral equation for the NLS equation by semi-group theory and perform Banach fixed point argument to obtain the existence and uniqueness of the solution to (\ref{1.1}). The continuous dependence then follows from the property of the contraction. As a part of the argument, the global well-posedness in $H^1(\R\times[0,1])$ is also achieved by using the energy estimates and corresponding {\it a-priori} estimates from some conserved quantities.

    Some notations are adopted in the context. For two real-valued terms $A$ and $B$, write: ($a$) $A \eqsim B$ if there is a positive constant $c$ so that $A = c B$; ($b$) $A \sim B$ if there exist two independent positive numbers $c_1$ and $c_2$ so that $c_1 A \le B \le c_2 A$; ($c$) $A \lesssim B$ (or $A \gtrsim B$) if there is a positive constant $c$ such that $A \le c B$ (or $A \ge c B$).

    The paper is organized as follows. Section \ref{sec_2} gives the formulation of the problem and the representations of solution operators. Section \ref{sec_3} deals with various estimates for the solution operators. The local well-posedness of the IBVP is proved in Section \ref{sec_4}. The global well-posedness of (\ref{1.1}) is provided in Section \ref{sec_5}.

\section{Formulation of the problem and representations of solutions}\label{sec_2}

    In this section, we apply Fourier series and Fourier transforms for the solution of the IBVP and give an integral representation of the solution for this problem.

    Write (\ref{1.1}) as
    \begin{equation}\label{2.1}
        \left\{
          \begin{array}{lrl}
            i u_t + u_{xx} + u_{yy} + f = 0\, , \quad & \qquad (x,y,t) & \in \R \times (0,1) \times (0,T) \, , \\
            u(x,y,0) = \varphi(x,y)\, ,  \quad & \qquad (x,y) & \in \R \times (0,1) \, , \\
            u(x,0,t) = h_1(x,t)\, ,\quad
            u(x,1,t) = h_2(x,t) \, , \quad & \qquad (x,t) & \in \R \times (0,T)\, ,  \\
          \end{array}
        \right.
    \end{equation}
    where $f(x,y,t) = \lambda |u(x,y,t)|^{p-2} u(x,y,t)$ for $p\ge3$ and $(x,y,t) \in \R \times (0,1) \times (0,T)$.
    The solution formula of \eqref{2.1} is derived as follows.
It is known that for any $g (y) \in L^2(0,1)$, the eigenfunctions $\sin (n\pi y) , n = 1, 2 ,\dots $ of the following Sturm-Liouville problem form a basis in $L^2 ( 0,1)$,
$$
- v_{yy} = \tilde \lambda_0 v \quad \mbox{for}\ y \in (0,1)\, ,\quad v ( 0 ) = v(1) = 0\, .
$$
which implies that $ g ( y ) = (1/2) \sum_{n=1}^\infty g_n\sin ( n \pi y ) $ with $g_n = \int^1_0 g (y ) \sin (n \pi y ) $ as the Fourier coefficient of $g$. Now, we multiply the equation in \eqref{2.1} by $\sin (n\pi y) $ and integrate the resulting equation from $0$ to $1$ together with integration by parts twice with respect to $y$, which yield the following equation for $U ( x, t, n ) = \int^1_0 u ( x, y, t) \sin ( n \pi y ) dy $ (note that $u ( x, y, t) = (1/2) \sum_{n=1}^\infty U(x,t,n) \sin ( n \pi y ) $),
\begin{align*}
& i U_t + U_{xx} - (n\pi )^2 U =  \big ( (-1)^n  h_2 - h_1 \big ) n\pi - \int^1_0 f ( x, y, t ) \sin ( n\pi y ) dy \, ,\\
& U (x, 0 , n) =  \int^1_0 \varphi ( x , y ) \sin (n\pi y ) dy \, .
\end{align*}
If $\hat U = \FF_x [U ](\xi , t , n) $ is the Fourier transform of $U$ with respect to $x$, then $\hat U$ satisfies
\begin{align*}
& i \hat U_t - ( \xi^2 + (n\pi )^2 ) \hat U = \FF_x \Big [ \big ( (-1)^n  h_2 - h_1 \big ) n\pi - \int^1_0 f ( x, y, t ) \sin ( n\pi y ) dy\Big ]  \, ,\\
 & \hat U (\xi, 0 , n) =  \int^1_0 \FF_x [ \varphi ( x , y ) ] \sin (n\pi y ) dy \, .
\end{align*}
Note that the above equation is a first-order nonhomogeneous ordinary differential equation of $\hat U$ with respect to $t$. After solving the equation for $\hat U$ and performing inverse Fourier transform with respect to $\xi$, it is
straightforward to derive the following integral equation equivalent to (\ref{2.1}),
    \begin{equation}\label{2.5}
       u(x,y,t) = W_0(t)\varphi (x,y) + W_b[h_1,h_2] (x,y,t) + \Phi_{0, f } ( x, y, t) \, \, ,
    \end{equation}
where $W_0 (t)$, $ W_b (x, y, t) $ and $ \Phi_{0, f } ( x, y, t)$  are given in the following propositions.
\begin{prop}\label{prop_2.2} 
    \begin{align} \nonumber
            W_0(t) \varphi(x,y) & = \int_{-\infty}^{\infty} e^{i\pi x\xi}\left[\sum_{n=1}^{\infty} C_n(\xi) \cdot e^{-i\pi^2\left(\xi^2+n^2\right)t} \cdot \sin(n\pi y)\right] \, d\xi \\
            &= \dfrac{1}{2i} \int_{-\infty}^{\infty} e^{i\pi x\xi}\left[\sum_{n\in\Z} \widetilde{C}_n(\xi) \cdot e^{-i\pi^2\left(\xi^2+n^2\right)t+in\pi y}\right] \, d\xi \label{2.6}
    \end{align}
        where $C_n(\xi) = 2\int_0^1 \widehat{\varphi^x}(\xi,y)\sin(n\pi y) \, dy$ and
        \begin{equation*}
            \widetilde{C}_n =
            \left\{
            \begin{array}{r}
                 C_n \, , \hspace{1in} n>0 \, , \\
                0 \, , \hfill n=0 \, , \\
                 -C_n \, , \hfill n<0\, .
            \end{array}
            \right.
        \end{equation*}
        Here, $\widehat{\varphi^x}$ denotes the Fourier transform of function $\varphi$ with respect to $x$.
    \end{prop}

It can be shown that $W_0(t)\varphi(x,y)$ is an odd function in $y$ if $\varphi$ is oddly extended to $[-1, 0]$. Also, it is interesting to note that $W_0 (t) \varphi$ solves \eqref{2.1} with $f = h_1 = h_2=0$ and can considered as a $C_0$-semigroup generated by $i\Delta$ (see \cite{Takaoka_NLSE_2d_Periodic}). However, the derivation of \eqref{2.5} is independent of any information or conditions used in semigroup theory.

\begin{prop}\label{prop_2.4} 
        \begin{equation*}\label{NLSE_IB_2d_stp_loc_nonlin_w0}
            \Phi_{0,f}(x,y,t) = i \left(\int_0^t W_0(t-\tau)f(\tau) \, d\tau \right) (x,y)\, .
        \end{equation*}
    \end{prop}
\begin{prop}\label{prop_2.5} 
        \begin{align}
             W_b[h_1,h_2](x,y,t)
            =& \pi \int_{-\infty}^{\infty} \sum_n \left[n \int_0^t e^{-i\pi^2\left(\xi^2+n^2\right)(t-\tau) + i\pi(\xi x + n y)} \left(\widehat{h_1^x} - (-1)^n \widehat{h_2^x} \right)(\xi, \tau) \, d\tau\right] \, d\xi \nonumber \\
            &\eqsim W_{b_1}[h_1](x,y,t) + W_{b_2}[h_2](x,y,t)\, ,  \label{2.7}
        \end{align}
        where
        \begin{align}
            W_{b_1}[h](x,y,t) &= \int_{-\infty}^{\infty} \sum_n \left[n \int_0^t e^{-i\pi^2\left(\xi^2+n^2\right)(t-\tau) + i\pi(\xi x + n y)} \widehat{h^x}(\xi, \tau) \, d\tau\right] \, d\xi \, , \label{2.8} \\
            W_{b_2}[h](x,y,t) &= \int_{-\infty}^{\infty} \sum_n \left[n \int_0^t e^{-i\pi^2\left(\xi^2+n^2\right)(t-\tau) + i\pi(\xi x + n y)} (-1)^{n+1} \widehat{h^x}(\xi, \tau) \, d\tau\right] \, d\xi \, . \label{2.9}
        \end{align}
    \end{prop}

        One may notice the equivalency between $|W_{b_1}[h]|$ in (\ref{2.8}) and $|W_{b_2}[h]|$ in (\ref{2.9}), as far as the estimates are concerned. Moreover, $|W_b[h_1,h_2] | \le |W_{b_1}[h_1]| + |W_{b_1}[h_2]|$, which means that we only need derive estimates for $W_{b_1}[h]$. The following form is also adopted for $W_{b_1}[h]$ alternatively,
        \begin{equation*}\label{NLSE_2d_stp_loc_lin_wb_exp}
            W_{b_1}[h](x,y,t) \eqsim \int_{-\infty}^{\infty} e^{i\pi x\xi} \sum_{n=1}^{\infty} \left[n \int_0^t e^{-i\pi^2\left(\xi^2+n^2\right)(t-\tau)} \left(\widehat{h^x} (\xi , \tau) \right) \, d\tau \sin(n\pi y)\right] \, d\xi\, .
        \end{equation*}
        If no confusion arises, we may use $W_b[h]$ instead of $W_{b_1}[h]$.

From the above discussion, we only need to study the solutions of \eqref{2.5}, which is derived from (\ref{2.1}).
The equivalency lemma for solutions of (\ref{2.1}) and those of (\ref{2.5}) is also valid as stated in \emph{Lemma 4.2.8} \cite{Cazenave_NLSE} if $h_1 = h_2 =0$.
    \begin{remark}\label{remark_2.1} 
        If $h_1 $ and $h_2$ are not both equal zero, then {Lemma 4.2.8} in \cite{Cazenave_NLSE} cannot be applied directly and some conditions on $h_1$ and $h_2$ must be added. If the initial and boundary data are smooth enough with compatibility conditions and the solutions of (\ref{2.5}) are in $C([0, T]; H^2 (\mathbb{R}\times [0,1])) \cap
        C^1([0, T]; L^2 (\mathbb{R}\times [0,1]))$, then it is straightforward to check that such solutions of (\ref{2.5}) are strong solutions of
        (\ref{2.1}) by reversing the derivation of (\ref{2.5}) from (\ref{2.1}). For general initial and boundary data, we will only consider the solutions of (\ref{2.5}), which is consistent with the mild solutions of (\ref{2.1}) defined in Definition \ref{gen_soln}.
    \end{remark}

    \begin{remark}\label{rmk_2.7} 

    The derivation of (\ref{2.5}) from (\ref{2.1}) shows that the procedure makes sense if $\varphi \in L^2 ( \mathbb{R}\times (0,1)), h_1, h_2 \in  L^2_{loc, t}  (\mathbb{R} \times  \mathbb{R})$ and $ f \in L^2_{loc, t} ( \mathbb{R}\times (0,1) \times \mathbb{R} )$ so that the Fourier series or Fourier transforms can be applied. There are no compatibility conditions at $t = 0 $ and $ y= 0, 1$ or other conditions on $\varphi, h_1, h_2$ and $f$ involved. Therefore, if a solution of (\ref{2.5}) can be found, then such a solution must be a solution of (\ref{2.1}) in the distributional sense, or the mild solution of (\ref{2.1}) defined in Definition \ref{gen_soln}. If $\varphi $ and $h_1, h_2$ have more regularities, say, in $H^1$, a necessary condition for (\ref{2.1}) to have a solution is that the compatibility conditions
    $\varphi ( x, 0 ) = h_1 ( x, 0)$ and $\varphi ( x, 1 ) = h_2 ( x, 0)$ in $L^2 (\mathbb{R})$ must be satisfied so that the convergent sequences for the initial and boundary data used in Definition \ref{gen_soln} can also satisfy those compatibility conditions. Here, notice that (\ref{2.5}) has only three independent integral operators and there are no compatibility conditions on $\varphi$ and $h_1, h_2$ needed for (\ref{2.5}) to have a solution. Therefore, (\ref{2.5}) is more general than (\ref{2.1}) in terms of the choice of $\varphi $ and $h_1, h_2$. The solution of (\ref{2.5}) can still be a solution of (\ref{2.1}) in the sense of distribution,
    if $\varphi$ and $h_1, h_2 $ have more regularity, but do not satisfy the compatibility conditions. If the compatibility conditions are satisfied, then by the definition of mild solution in Definition \ref{gen_soln}, the regularity of the solution of (\ref{2.5}) implies the regularity of the solution of (\ref{2.1}). It is noted that more compatibility conditions are required if more regularity of $\varphi$ and $h_1, h_2$ is imposed (see Remark \ref{rem1.6}).  More detailed discussions of such compatibility conditions for the KdV equations or parabolic equations in domains with boundaries can be found in \cite{bsz-1,LM1972-2}.

    \end{remark}

\section{Estimates of solution operators}\label{sec_3}

    In this section, the estimates for solution operators in \emph{Section~\ref{sec_2}} are derived.

    First, show that $W_0$ maps from $H^s (\R\times [0,1])$ to $$L^r_t\left([0,T]; \, W^{s,r}_{xy}(\R\times [0,1])\right) \bigcap L^{\infty}_t \left([0,T]; \, H^s_{xy}(\R\times [0,1])\right)$$ for $r\in[2,4]$.

    \begin{prop}\label{prop_3.1} 
        Let $r\in [2,4]$ and $s\ge 0$. Then for some $\sigma>\frac{1}{2}$,
        \begin{align}
            \left\|W_0(t) \varphi\right\|_{L^r_t \left([0,T]; \, W^{s,r}_{xy}(\R \times [0,1])\right)} & \lesssim \left(T^{\frac{1}{2}}+T^{\frac{1}{2}-\sigma}\right) \left\|\varphi\right\|_{H^s_{xy}(\R \times [0,1])} \, , \label{3.1} \\ 
            \left\|W_0(t)\varphi\right\|_{L^{\infty}_t \left([0,T]; \, H^s_{xy}(\R \times [0,1])\right)} & \lesssim \left\|\varphi\right\|_{H^s_{xy}(\R \times [0,1])} \, . \label{3.2} 
        \end{align}
    \end{prop}

    \begin{proof}
        First, we let $s = 0$. By using an odd extension of  initial condition to $y \in [-1, 0]$ and then a periodic extension to $y \in \T$, the estimate of $\left\|W_0(\cdot)\varphi\right\|_{L^r_{xyt} (\R \times [0,1] \times [0,T])}$ is provided for $r\in [2,4]$ in \emph{Section 2} of \cite{Takaoka_NLSE_2d_Periodic} (here, since $s=0$, the extension is always possible). By restricting the estimate on the strip domain $\R \times [0,1]$, we have
        \begin{equation}\label{3.3} 
            \left\|W_0(t)\varphi\right\|_{L^r_{xyt} (\R \times [0,1] \times [0,T])} \lesssim \left(T^{\frac{1}{2}}+T^{\frac{1}{2}-\sigma}\right) \left\|\varphi\right\|_{L^2_{xy}(\R \times [0,1])} \, .
        \end{equation}
        However, the representation of $W_0(t)\varphi$ in (\ref{2.6}) shows that if $\alpha = ( \alpha_1, \alpha_2 ) $ is a nonnegative two-dimensional multi-index with $|\alpha|=\alpha_1 + \alpha_2 = s$ with $\alpha_1, \alpha_2$ integers, then
        \begin{align*}
        D^{\alpha} W_0(t)\varphi = & \dfrac{1}{2i} D^{\alpha} \int_{-\infty}^{\infty} e^{i\pi x\xi}\left[\sum_{n\in\Z} \widetilde{C}_n(\xi) \cdot e^{-i\pi^2\left(\xi^2+n^2\right)t+in\pi y}\right] \, d\xi \\
         \simeq &\int^\infty_{-\infty} \sum_{n\in\Z}  e^{-i\pi^2\left(\xi^2+n^2\right)t+i\pi (  x\xi + n y)} n^{\alpha_2}\widehat { \partial ^{\alpha_1}_x \phi} (\xi , n ) \, d\xi
        \end{align*}
        where $\phi $ is the odd extension of $\varphi $ to $[-1, 1]$ for the variable $y$ (note that we only use the Fourier coefficients of $\phi$ for $y\in [0,1]$). Thus, from the proof of (\ref{3.3}), we have the estimate for the norm on $L^r_t \left([0,T], W^{s,r}_{xy}(\R \times [0,1])\right)$, which yields (\ref{3.1}). Here, we note that, as usual, the Sobolev norm $H^s (0,1)$ is equivalent to the norm  $\left ( \sum_{n\in \Z} (1+n^2 )^s |c_n|^2\right )^{1/2} $ with $c_n$ as the Fourier coefficients of the function in $(0,1)$ and the interpolation theorems are applied if $s$ is not a positive integer.
        Moreover, if  $t\in [0,T]$ is fixed and $\phi $ is an odd extension of $\varphi$, (\ref{2.6}) implies
        \begin{align*}
            \|W_0(t)\varphi\|^2&_{L^2_{xy}(\R \times [0,1])} = \dfrac{1}{2} \|W_0(t)\phi\|^2_{L^2_{xy}(\R \times [-1,1])} \\
            &= \dfrac{1}{2} \int_{-\infty}^{\infty} \int_{-1}^1 \left|\int_{-\infty}^{\infty} \sum_{n\in\Z} e^{-i\pi^2\left(\xi^2+n^2\right)t+i\pi(x\xi+ny)} \cdot \widehat{\phi}(\xi,n) \, d\xi \right|^2 \, dy \, dx \\
            &\eqsim \int_{-\infty}^{\infty} \sum_{n\in\Z} \left|e^{-i\pi^2\left(\xi^2+n^2\right)t} \cdot \widehat{\phi}(\xi,n) \right|^2 \, d\xi \\
            &= \int_{-\infty}^{\infty} \sum_{n\in\Z} \left|\widehat{\phi}(\xi,n) \right|^2 \, d\xi = \left\|\phi\right\|^2_{L^2_{xy}(\R \times [-1,1])} \eqsim \left\|\varphi\right\|^2_{L^2_{xy}(\R \times [0,1])}\, .
        \end{align*}
        By a similar argument for its derivatives, the estimate (\ref{3.2}) is valid for any $s\ge0$.
    \end{proof}

    Note that in the derivation of (\ref{3.3}) given in \cite{Takaoka_NLSE_2d_Periodic}, the following result is proved.
    \begin{lemma}
        For $r\in [2,4]$ and any $\sigma>1/2$, let $f=f(x,y,t)$. Then
        \begin{equation}\label{3.4} 
            \|f\|_{L^r_{xyt} (\R \times [-1,1] \times [0,T]) \bigcap L^{\infty}_t\left([0,T]; \, L^2_{xy}(\R \times [-1,1])\right)} \lesssim \|f\|_{X^{\sigma,0} \left(\R \times [-1,1] \times [0,T]\right)}
        \end{equation}
        and from Definition~\ref{Bourgain_RT_defn} of the Bourgain space,
        \begin{equation}\label{3.5} 
            \left\|W_0 f\right\|_{L^r_{xyt} (\R \times [-1,1] \times [0,T])} \lesssim \left\|W_0 f\right\|_{X^{\sigma,0} \left(\R \times [-1,1] \times [0,T]\right)} = \|f\|_{H^{\sigma}_t \left([0,T]; \, L^2_{xy}(\R \times [-1,1]) \right)} \, .
        \end{equation}
    \end{lemma}
We note that \eqref{3.4} can be proved similarly as that in \cite{Takaoka_NLSE_2d_Periodic} if $f$ is extended to $\R$ periodically for the $y$-variable, where no boundary conditions are involved. Moreover, the formula of $W_0 (t) \varphi$ in \eqref{2.6} is same as the one derived for the semi-group $U(t) \varphi$ in \cite{Takaoka_NLSE_2d_Periodic} for periodic case, which implies that the estimates of $U(t) \varphi$ in \cite{Takaoka_NLSE_2d_Periodic} can be applied to obtain the estimates of $W_0 (t) \varphi$ here. One may also easily verify that, if $f$ is an odd function in $y$, then (\ref{3.4}) and (\ref{3.5}) can be equivalently rewritten over domain $(\R \times [0,1] \times [0,T])$. In the following, we obtain the estimates for $\Phi_{0,f} $.
    \begin{prop}\label{prop_3.3} 
        For $r \in [2,4]$ and $\sigma \in (1/2,1]$, there is a $q \in [4\sigma/(1+\sigma),2]$ such that
        \begin{align}
            & \left\|\Phi_{0,f}\right\|_{L^r \left(\R \times [0,1] \times [0,T]\right) \bigcap L^{\infty}_t \left([0,T]; \, L^2_{xy}(\R \times [0,1])\right)} \lesssim T^{1+\sigma -\frac{4\sigma}{q}} \|f\|_{L^q \left(\R \times [0,1] \times [0,T]\right)} \, , \label{3.6} \\ 
            & \left\|\Phi_{0,f}\right\|_{L^r \left([0,T]; \, W^{s,r}(\R \times [0,1])\right) \bigcap L^{\infty}_t \left([0,T]; \, H^s_{xy}(\R \times [0,1])\right)} \lesssim T^{1+\sigma -\frac{4\sigma}{q}} \|f\|_{L^q \left([0,T]; \, W^{s,q}(\R \times [0,1])\right)} \, , \label{3.7} \\[.08in] 
            & \left\|\Phi_{0,f}\right\|_{L^4 \left([0,T]; \, W^{s,4}(\R \times [0,1])\right) \bigcap L^{\infty}_t \left([0,T]; \, H^s_{xy}(\R \times [0,1])\right)} \lesssim C_T \|f\|_{L^{4/3} \left([0,T]; \, W^{s,4/3}(\R \times [0,1])\right)} \, , \label{3.8}
        \end{align}
        where $C_T$ only depends upon $T$.
    \end{prop}

    \begin{proof}
        Choose $\sigma \in (\frac{1}{2},1]$ as that in (\ref{3.4}) and $q \in [\frac{4\sigma}{(1+\sigma)},2]$. Using duality on (\ref{3.4}) with $r=4$, we are able to see
        \begin{equation*}
            \|f\|_{X^{-\sigma,0} \left(\R \times [0,1] \times [0,T]\right)} \lesssim \|f\|_{L^{\frac{4}{3}} \left(\R \times [0,1] \times [0,T]\right)}
        \end{equation*}
        as well as
        \begin{equation*}
            \|f\|_{X^{0,0} \left(\R \times [0,1] \times [0,T]\right)} \lesssim \|f\|_{L^2 \left(\R \times [0,1] \times [0,T]\right)}\, .
        \end{equation*}
        By interpolation, we can obtain
        \begin{equation}\label{3.9} 
            \|f\|_{X^{\sigma',0} \left(\R \times [0,1] \times [0,T]\right)} \lesssim \|f\|_{L^q \left(\R \times [0,1] \times [0,T]\right)}
        \end{equation}
        with $\frac{1}{q} = \frac{a}{4/3} + \frac{1-a}{2}$ and $\sigma' = a(-\sigma)$ for $0\le a \le1$. Therefore
        \begin{equation*}\label{NLSE_IB_2d_stp_loc_nonlin_w0_Bourgain_sigma_interp_pair} 
            \sigma' = \frac{2\sigma(q-2)}{q}
        \end{equation*}
        where $\frac{4}{3}<\frac{4\sigma}{1+\sigma} \le q \le 2$ for $\sigma>\frac{1}{2}$. It is easy to check that $\sigma'\le0$, $$-\sigma' = 4\sigma\left(\dfrac{1}{q}-\dfrac{1}{2}\right) \le 4\sigma\left(\dfrac{1+\sigma}{4\sigma}-\dfrac{1}{2}\right) = 1-\sigma < \dfrac{1}{2} \, , \quad \mbox{and}\quad 1-(\sigma-\sigma') = 1+\sigma-\dfrac{4\sigma}{q} \ge 0 \, .$$
        Moreover, we may choose $-\frac{1}{2} < \sigma' \le 0 \le \sigma \le \sigma'+1$ in order to use Lemma 3.2 in \cite{Ginibre_Periodic_semilin}. Thus, from this lemma in \cite{Ginibre_Periodic_semilin},  (\ref{3.4}) with $s=0$, and (\ref{3.9}), we have
        \begin{align*}
            \left\|\Phi_{0,f}\right\|&_{L^r \left(\R \times [0,1] \times [0,T]\right) \bigcap L^{\infty}_t \left([0,T]; \, L^2_{xy}(\R \times [0,1])\right)} \\
            & = \left\|\int_0^t W_0(t-\tau)f(\tau) \, d\tau\right\|_{L^r \left(\R \times [0,1] \times [0,T]\right) \bigcap L^{\infty}_t \left([0,T]; \, L^2_{xy}(\R \times [0,1])\right)} \\
            & \lesssim \left\|\int_0^t W_0(t-\tau)f(\tau) \, d\tau\right\|_{X^{\sigma,0} \left(\R \times [0,1] \times [0,T]\right)} \\
            & \lesssim T^{1-(\sigma-\sigma')} \|f\|_{X^{\sigma',0} \left(\R \times [0,1] \times [0,T]\right)} \lesssim T^{1+\sigma-\frac{4\sigma}{q}} \|f\|_{L^q \left(\R \times [0,1] \times [0,T]\right)}\, .
        \end{align*}
        For $s>0$ and $|\alpha|=\alpha_1 + \alpha_2 = s$, by using a similar idea for $D^\alpha W_0 (t) \varphi$, we can derive the following estimates using (\ref{3.4}) again, with the same $\sigma$ and $\sigma'$ defined above:
        \begin{align*}
            \left\|D^{\alpha}_{xy} \Phi_{0,f}\right\|&_{L^r \left(\R \times [0,1] \times [0,T]\right) \bigcap L^{\infty}_t \left([0,T]; \, L^2_{xy}(\R \times [0,1])\right)} \\
            &\simeq \left\|\int_0^t W_0(t-\tau)  \tilde D^{\alpha_2}_y {\left(D^{\alpha_1}_{x}f(\tau)\right)} \, d\tau\right\|_{L^r \left(\R \times [0,1] \times [0,T]\right) \bigcap L^{\infty}_t \left([0,T]; \, L^2_{xy}(\R \times [0,1])\right)} \\
            & \lesssim \left\|\int_0^t W_0(t-\tau) \tilde D^{\alpha_2}_y { \left(D^{\alpha_1}_{x}f(\tau)\right) }\, d\tau\right\|_{X^{\sigma,0} \left(\R \times [0,1] \times [0,T]\right)} \\
            & \lesssim T^{1-(\sigma-\sigma')} \| \tilde D^{\alpha_2}_y { \left(D^{\alpha_1}_{x}f \right ) } \|_{X^{\sigma',0} \left(\R \times [0,1] \times [0,T]\right)} \le T^{1+\sigma-\frac{4\sigma}{q}} \| f\|_{L^q \left([0,T], W^{s, q} (\R \times [0,1]) \right)}\, ,
        \end{align*}
        where $\tilde D_y^{\alpha_2}  f(y) $ denotes the function obtained from $f (y)$ with its Fourier coefficients as $n^{\alpha_2} c_n, n \in \Z$ and $c_n, n\in \Z$ as the Fourier coefficients of $f$ with odd extension to $[-1, 1]$.
Hence, both (\ref{3.6}) and (\ref{3.7}) are proved.

To prove (\ref{3.8}), for $s = 0$, Lemma 4.1 of \cite{Takaoka_NLSE_2d_Periodic} and the above estimates for $ q= 3/4$ yield
\begin{align*}
& \left\|\Phi_{0,f}\right\|_{L^4 \left ([0,T]\times \R \times [0,1]\right ) \bigcap L^{\infty}_t \left([0,T]; \, L^2_{xy}(\R \times [0,1])\right)} \lesssim C_T \|f\|_{L^{4/3} \left([0,T]\times \R \times [0,1]\right)} \, .
\end{align*}
For an integer $s>0$ and $|\alpha|=\alpha_1 + \alpha_2 = s$ with nonnegative integers $\alpha_1$ and $\alpha_2$,
\begin{align*}
            \left\|D^{\alpha}_{xy} \Phi_{0,f}\right\|&_{L^4 \left([0,T]\times \R \times [0,1]\right) \bigcap L^{\infty}_t \left([0,T]; \, L^2_{xy}(\R \times [0,1])\right)} \\
            &\simeq \left\|\int_0^t W_0(t-\tau)  \tilde D^{\alpha_2}_y {\big (D^{\alpha_1}_{x}f(\tau)\big )} \, d\tau\right\|_{L^4 \left([0,T]\times \R \times [0,1]\right) \bigcap L^{\infty}_t \left([0,T]; \, L^2_{xy}(\R \times [0,1])\right)} \\
            & \lesssim C_T \|\tilde D^{\alpha_2}_y {\left(D^{\alpha_1}_{x}f(\tau)\right)} \|_{L^{4/3} \left([0,T]\times \R \times [0,1]\right)} \\
            & \lesssim C_T  \| f\|_{L^{4/3} \left([0,T], W^{s, {4/3}} (\R \times [0,1]) \right)}\, .
        \end{align*}
Then, a classical interpolation theorem gives the inequality for a non-integer $ s > 0$ and (\ref{3.8}) is proved. Here, we note that $W_0 (t) $ defined in Proposition \ref{prop_2.2} can be considered as an integral operator and does not require the condition that $\varphi$ or $f$ is zero at the boundary. In the proof of (\ref{3.8}) for a positive integer $ s$, the derivative is directly taken to the operator $W_0 (t- \tau ) f $, which, by the definition of $W_0 (t) $ in Proposition \ref{prop_2.2}, can be transferred to $f$. Therefore, no boundary conditions are required for $f$ in the proof of (\ref{3.8}).
\end{proof}

    Now, we turn our attention to the operator $W_b[h_1,h_2](x,y,t)$ (or $W_b[h]$).
    \begin{prop}\label{prop_3.4} 
        For $r\in [2,4]$ and $ s \geq 0$, if $h \in \HH^s(0,T)$, then
        \begin{align}
            \left\|W_b[h_1,h_2]\right\|&_{L^4_{xyt} \left(\R \times [0,1] \times [0,T]\right) \bigcap L^{\infty}_t \left([0,T]; \, L^2_{xy}(\R \times [0,1])\right)} \lesssim \sum_{j=1,2} \|h_j\|_{\HH^0(0,T)}\, , \label{3.10} \\
            \left\|W_b[h_1,h_2]\right\|&_{L^r_t \left([0,T]; \, W^{s,r}_{xy}(\R \times [0,1]) \right) \bigcap L^{\infty}_t \left([0,T]; \, H^s_{xy}(\R \times [0,1])\right)} \lesssim \sum_{j=1,2} \|h_j\|_{\HH^s(0,T)} \, . \label{3.11} 
        \end{align}
    \end{prop}

    \begin{proof}
        In this proof, we use $W_b$ for $W_{b_1}$ unless it is indicated otherwise.
        We first let
        \begin{equation}\label{3.12}
            \widetilde{h}(\xi,\lambda) = \int_0^T e^{i\pi^2\lambda t} \widehat{h^x}(\xi,t) \, dt \eqsim \widehat{h}(\xi,-\lambda) \, .
        \end{equation}
        Here, without loss of generality, we assume that the support of $ h$ with respect to $t$ is inside of $[0, T]$ (otherwise, we could just multiply $h$ by a smooth cut-off function). Thus, $\widehat{h^x}(\xi,t) \eqsim \int_{-\infty}^{\infty} e^{-i\pi^2\lambda t} \widetilde{h}(\xi,\lambda) \, d\lambda$.
        Substitute (\ref{3.12}) into (\ref{2.8}) and obtain
        \begin{align*}
            W_b [h] &(x,y,t) \\
            &\eqsim \int_{-\infty}^{\infty} \sum_n \left[e^{-i\pi^2\left(\xi^2+n^2\right)t+i\pi(x\xi+ny)} \cdot n \int_0^t e^{i\pi^2\left(\xi^2+n^2\right)\tau} \left(\int_{-\infty}^{\infty} e^{-i\pi^2\lambda\tau} \widetilde{h}(\xi,\lambda) \, d\lambda\right) \, d\tau\right] \, d\xi \\
            &= \int_{-\infty}^{\infty} \sum_n \left[e^{-i\pi^2\left(\xi^2+n^2\right)t+i\pi(x\xi+ny)} \cdot n \int_{-\infty}^{\infty} \left(\int_0^t e^{i\pi^2\left(\xi^2+n^2-\lambda\right)\tau} \, d\tau \right) \widetilde{h}(\xi,\lambda) \, d\lambda\right] \, d\xi \\
            &= \int_{-\infty}^{\infty} \sum_n \left[e^{-i\pi^2\left(\xi^2+n^2\right)t+i\pi(x\xi+ny)} \cdot n \int_{-\infty}^{\infty} \dfrac{e^{i\pi^2\left(\xi^2+n^2-\lambda\right)t}-1}{i\pi^2\left(\xi^2+n^2-\lambda\right)} \widetilde{h}(\xi,\lambda) \, d\lambda\right] \, d\xi \\
            &= \dfrac{1}{i\pi^2} \left(I^+ + I^-\right)\, ,
        \end{align*}
        where
        \begin{align*}
            I^+ &= \int_{-\infty}^{\infty} \sum_n \left(e^{-i\pi^2\left(\xi^2+n^2\right)t+i\pi(x\xi+ny)} \cdot n \int_0^{\infty} \dfrac{e^{i\pi^2\left(\xi^2+n^2-\lambda\right)t}-1}{\xi^2+n^2-\lambda} \cdot \widetilde{h}(\xi,\lambda) \, d\lambda\right) \, d\xi\, , \\
            I^- &= \int_{-\infty}^{\infty} \sum_n \left(e^{-i\pi^2\left(\xi^2+n^2\right)t+i\pi(x\xi+ny)} \cdot n \int_0^{\infty} \dfrac{e^{i\pi^2(\xi^2+n^2+\lambda)t}-1}{\xi^2+n^2+\lambda} \cdot \widetilde{h}(\xi,-\lambda) \, d\lambda\right) \, d\xi\, .
        \end{align*}
        We split $I^+$ and $I^-$ as follows:
        \begin{align*}
            & I^+ = I^+_1 + I^+_2 +I^+_3
        \end{align*}
        with
        \begin{align} 
            & I^+_1 = \int_{-\infty}^{\infty} \sum_n \left(e^{-i\pi^2\left(\xi^2+n^2\right)t+i\pi(x\xi+ny)}  n \int_0^{\infty} \dfrac{e^{i\pi^2\left(\xi^2+n^2-\lambda\right)t}-1}{\xi^2+n^2-\lambda} \cdot \widetilde{h}(\xi,\lambda) \psi\left(\xi^2+n^2-\lambda\right) \, d\lambda\right) \, d\xi  \, ,\label{3.13} \\
            & I^+_2 = \int_{-\infty}^{\infty} \sum_n \left(e^{i\pi(x\xi+ny)}  n \int_0^{\infty} \dfrac{e^{-i\pi^2\lambda t}}{\xi^2+n^2-\lambda} \cdot \widetilde{h}(\xi,\lambda) \left(1-\psi\left(\xi^2+n^2-\lambda\right)\right) \, d\lambda\right) \, d\xi\, , \label{3.14} \\
            & I^+_3 = \int_{-\infty}^{\infty} \sum_n \left(e^{-i\pi^2\left(\xi^2+n^2\right)t+i\pi(x\xi+ny)}  n \int_0^{\infty} \dfrac{-1}{\xi^2+n^2-\lambda} \cdot \widetilde{h}(\xi,\lambda) \left(1-\psi\left(\xi^2+n^2-\lambda\right)\right) \, d\lambda\right) \, d\xi\, , \label{3.15}
        \end{align}
        where a cut-off function $\psi(x)$ is defined by $\psi\in \mathcal{D}(\R )$ and
        \begin{equation*}\label{cutoff}
            \psi(x) =
            \left\{
            \begin{array}{ll}
                1 \text{, } & |x|\le1 \, ,\\
                0 \text{, } & |x|\ge2\, .
            \end{array}
            \right.
        \end{equation*}

        Similarly, rewrite
        \begin{align*}
            I^- &\eqsim \int_{-\infty}^{\infty} \sum_n \left(e^{-i\pi^2\left(\xi^2+n^2\right)t+i\pi(x\xi+ny)} \cdot n \int_0^{\infty} \dfrac{e^{i\pi^2(\xi^2+n^2+\lambda)t}-1}{\xi^2+n^2+\lambda} \widetilde{h}(\xi,-\lambda) \, d\lambda\right) \, d\xi \\
            & = \int_{-\infty}^{\infty} \sum_n \left(e^{i\pi(x\xi+ny)} \cdot n \int_0^{\infty} \dfrac{e^{i\pi^2\lambda t}-e^{-i\pi^2\left(\xi^2+n^2\right)t}}{\xi^2+n^2+\lambda} \widetilde{h}(\xi,-\lambda) \, d\lambda\right) \, d\xi \\
            & = I^-_1 - I^-_2
        \end{align*}
        with
        \begin{align}
            I^-_1 &= \int_{-\infty}^{\infty} \sum_n \left(e^{i\pi(x\xi+ny)} \cdot n \int_0^{\infty} \dfrac{e^{i\pi^2\lambda t}}{\xi^2+n^2+\lambda} \widetilde{h}(\xi,-\lambda) \, d\lambda\right) \, d\xi \, , \label{3.16} \\
            I^-_2 &= \int_{-\infty}^{\infty} \sum_n \left(e^{i\pi(x\xi+ny)} \cdot n \int_0^{\infty} \dfrac{e^{-i\pi^2\left(\xi^2+n^2\right)t}}{\xi^2+n^2+\lambda} \widetilde{h}(\xi,-\lambda) \, d\lambda\right) \, d\xi \, . \label{3.17}
        \end{align}

        We first study $I^+_1$. Rewrite $e^{i\pi^2\left(\xi^2+n^2- \lambda \right)t}$ in (\ref{3.13}) as its Taylor expansion and obtain that
        $
            \displaystyle I^+_1 = \sum_{k=0}^{\infty} I^+_{1,k}\, ,
        $
        where
        \begin{align*}
            I^+_{1,k} &= \dfrac {\left (i\pi^2 \cdot t\right)^{k+1}}{(k+1)!} \cdot \bigg [\int_{-\infty}^{\infty} \sum_n e^{-i\pi^2\left(\xi^2+n^2\right)t+i\pi(x\xi+ny)} \\
            & \qquad \times \left(n \int_0^{\infty} \left(\xi^2+n^2-\lambda\right)^k \cdot \widetilde{h}(\xi,\lambda) \psi\left(\xi^2+n^2-\lambda\right)\right) \, d\lambda \, d\xi \bigg ] \, .
        \end{align*}
        Note that $\psi\left(\xi^2+n^2-\lambda\right)\not\equiv 0$ for $-2\le\xi^2+n^2-\lambda\le2$, i.e. $-2+\lambda \le n^2+\xi^2 \le 2+\lambda$. For each $I^+_{1,k}$, apply (\ref{3.2}) and (\ref{3.3}) to obtain
        \begin{align*}
            & \|I^+_{1,k}\|_{L^4_{xyt} \left(\R \times [0,1] \times [0,T]\right) \bigcap L^{\infty}_t \left([0,T]; \, L^2_{xy}(\R \times [0,1])\right)} \\
            &\le \dfrac{\left(\pi^2 \cdot T\right)^{k+1}}{(k+1)!} \bigg \|\int_{-\infty}^{\infty} \sum_n e^{-i\pi^2\left(\xi^2+n^2\right)t+i\pi(x\xi+ny)}  \\
            & \qquad \times \left(n \int_0^{\infty} \left(\xi^2+n^2-\lambda\right)^k \cdot \widetilde{h}(\xi,\lambda) \psi\left(\xi^2+n^2-\lambda\right)\right) \, d\lambda \, d\xi \bigg \|_{L^4 \bigcap L^{\infty}_t \left(L^2_{xy}\right)} \\
            &\lesssim \dfrac{\left(\pi^2 \cdot T\right)^{k+1}}{(k+1)!} \left\|n \cdot \int_0^{\infty}\left(\xi^2+n^2-\lambda\right)^k \cdot \widetilde{h}(\xi,\lambda) \cdot \psi\left(\xi^2+n^2-\lambda\right) \, d\lambda \right\|_{L^2_{\xi n}} \\
            &= \dfrac{\left(\pi^2 \cdot T\right)^{k+1}}{(k+1)!} \left\{\int_{-\infty}^{\infty} \sum_n \left|n \cdot \int_0^{\infty}\left(\xi^2+n^2-\lambda\right)^k \cdot \widetilde{h}(\xi,\lambda) \cdot \psi\left(\xi^2+n^2-\lambda\right) \, d\lambda \right|^2 \, d\xi \right\}^{\frac{1}{2}} \\
            &\le \dfrac{\left(\pi^2 \cdot T\right)^{k+1}}{(k+1)!} \left\{\int_{-\infty}^{\infty} \sum_n \left(\int_0^{\infty} |n| \cdot \left|\xi^2+n^2-\lambda\right|^k \cdot \left|\widetilde{h}(\xi,\lambda)\right| \cdot \left|\psi\left(\xi^2+n^2-\lambda\right)\right| \, d\lambda \right)^2 \, d\xi \right\}^{\frac{1}{2}} \\
            &\lesssim \dfrac{2^k \left(\pi^2 \cdot T\right)^{k+1}}{(k+1)!} \bigg \{\int_{-\infty}^{\infty} \sum_n \left(\int_0^{\infty} \left(1+\xi^2+|\lambda|\right) \cdot \left|\widetilde{h}(\xi,\lambda) \right|^2 \left|\psi\left(\xi^2+n^2-\lambda\right)\right| \, d\lambda \right)  \\
            & \qquad \times \left(\int_0^{\infty} \left|\psi\left(\xi^2+n^2-\lambda\right)\right| \, d\lambda \right) \, d\xi \bigg \}^{\frac{1}{2}} \\
            &\le \dfrac{2^k \left(\pi^2 \cdot T\right)^{k+1}}{(k+1)!} \left[\int_{-\infty}^{\infty} \sum_n \left(\int_{\xi^2+n^2-2}^{\xi^2+n^2+2} \left(1+\xi^2+|\lambda|\right) \cdot \left|\widetilde{h}(\xi,\lambda) \right|^2 \, d\lambda\right) \cdot \left(\int_{\xi^2+n^2-2}^{\xi^2+n^2+2} 1 \, d\lambda\right) \, d\xi \right]^{\frac{1}{2}} \\
            &\lesssim \dfrac{\left(2 \pi^2 \cdot T\right)^{k+1}}{(k+1)!} \left[\int_{-\infty}^{\infty} \left(\sum_n \int_{\xi^2+n^2-2}^{\xi^2+n^2+2} \left(1+|\lambda|\right) \cdot \left|\widetilde{h}(\xi,\lambda) \right|^2 \, d\lambda\right) \, d\xi \right]^{\frac{1}{2}}\, .
        \end{align*}
        Since $(n+1)^2-2 > n^2+2$ for $n\ge1$, the sequence of intervals $$\left\{\left[\xi^2+n^2-2,\xi^2+n^2+2\right]\right\}_{n\in\Z}$$ is disjoint. Thus, we find that
        \begin{align*}
             \|I^+_{1,k}&\|_{L^4_{xyt} \left(\R \times [0,1] \times [0,T]\right) \bigcap L^{\infty}_t \left([0,T]; \, L^2_{xy}(\R \times [0,1])\right)} \\
            &\lesssim \dfrac{\left(2 \pi^2 \cdot T\right)^{k+1}}{(k+1)!} \left[\int_{-\infty}^{\infty} \int_0^{\infty} \left(1+|\lambda|\right) \left|\widetilde{h}(\xi,\lambda) \right|^2 \, d\lambda \, d\xi \right]^{\frac{1}{2}} = \dfrac{\left(2 \pi^2 \cdot T\right)^{k+1}}{(k+1)!} \|h\|_{H^{\frac{1}{2}}_t \left([0,T]; \, L^2_x(\R) \right)}\, .
        \end{align*}
        Hence,
        \begin{align}
            \|I^+_1\|&_{L^4_{xyt} \left(\R \times [0,1] \times [0,T]\right) \bigcap L^{\infty}_t \left([0,T]; \, L^2_{xy}(\R \times [0,1])\right)} \le \sum_{k=0}^{\infty} \|I^+_{1,k}\|_{L^4_{xyt} \left(\R \times [0,1] \times [0,T]\right) \bigcap L^{\infty}_t \left([0,T]; \, L^2_{xy}(\R \times [0,1])\right)} \nonumber \\
            &\lesssim \sum_{k=0}^{\infty} \dfrac{\left(2 \pi^2 \cdot T\right)^{k+1}}{(k+1)!} \|h\|_{H^{\frac{1}{2}}_t \left([0,T]; \, L^2_x(\R) \right)} \le \|h\|_{H^{\frac{1}{2}}_t \left([0,T]; \, L^2_x(\R) \right)} \, . \label{3.18}
        \end{align}
        Next, we consider the term $I^+_2$ in (\ref{3.14}). First, for a fixed $t\in(0,T]$,
        \begin{align*}
            \|I^+_2(t)\|_{L^2_{xy}}^2 &\eqsim \int_{-\infty}^{\infty} \sum_{n=1}^{\infty} \left|\left(\int_0^{\xi^2}+\int_{\xi^2}^{\infty}\right) n \cdot \dfrac{e^{-i\pi^2\lambda t}}{\xi^2+n^2-\lambda} \widetilde{h}(\xi,\lambda) \left(1-\psi\left(\xi^2+n^2-\lambda\right)\right) \, d\lambda\right|^2 \, d\xi \\
            &\le  S_1 + S_2
        \end{align*}
        where
        \begin{align}
            S_1 &= \int_{-\infty}^{\infty} \sum_{n=1}^{\infty} \left|\int_0^{\xi^2} n \cdot \dfrac{e^{-i\pi^2\lambda t}}{\xi^2+n^2-\lambda} \widetilde{h}(\xi,\lambda) \left(1-\psi\left(\xi^2+n^2-\lambda\right)\right) \, d\lambda\right|^2 \, d\xi \, , \label{3.19} \\ 
            S_2 &= \int_{-\infty}^{\infty} \sum_{n=1}^{\infty} \left|\int_{\xi^2}^{\infty} n \cdot \dfrac{e^{-i\pi^2\lambda t}}{\xi^2+n^2-\lambda} \widetilde{h}(\xi,\lambda) \left(1-\psi\left(\xi^2+n^2-\lambda\right)\right) \, d\lambda\right|^2 \, d\xi \, . \label{3.20} 
        \end{align}
        In (\ref{3.19}), substitute $\xi^2-\mu$ for $\lambda$, and use Holder's inequality to obtain
        \begin{align*}
            S_1 &= \int_{-\infty}^{\infty} \sum_{n=1}^{\infty} \left|\int_0^{\xi^2} n \cdot \dfrac{e^{-i\pi^2\left(\xi^2-\mu\right) t}}{n^2+\mu} \widetilde{h}\left(\xi,\xi^2-\mu\right) \left(1-\psi(n^2+\mu)\right) \, d\mu\right|^2 \, d\xi \\
            &\le \int_{-\infty}^{\infty} \sum_{n=1}^{\infty} \left(\int_0^{\xi^2} \left(1+\xi^2-\mu\right) \left|\widetilde{h}\left(\xi,\xi^2-\mu\right)\right|^2 \, d\mu\right) \cdot \left(\int_0^{\xi^2} \dfrac{n^2 \cdot \left(1-\psi(n^2+\mu)\right)}{(n^2+\mu)^2\left(1+\xi^2-\mu\right)} \, d\mu\right) \, d\xi \\
            &= \int_{-\infty}^{\infty} \left(\int_0^{\xi^2} \left(1+\xi^2-\mu\right) \left|\widetilde{h}(\xi,\xi^2-\mu)\right|^2 \, d\mu\right) \cdot \left(\sum_{n=1}^{\infty} \int_0^{\xi^2} \dfrac{n^2 \cdot \left(1-\psi(n^2+\mu)\right)}{(n^2+\mu)^2\left(1+\xi^2-\mu\right)} \, d\mu\right) \, d\xi \\
            &\le \int_{-\infty}^{\infty} \left(\int_0^{\xi^2} \left(1+\xi^2-\mu\right) \left|\widetilde{h}(\xi,\xi^2-\mu)\right|^2 \, d\mu\right) \cdot \left(\int_0^{\xi^2} \sum_{n=1}^{\infty} \dfrac{n^2}{(n^2+\mu)^2\left(1+\xi^2-\mu\right)} \, d\mu\right) \, d\xi\, .
        \end{align*}
        Let us first study the second integral factor. Since $({n^2+\mu})^{-1}$ is strictly decreasing in $n$, then $\sum_{n=1}^{\infty} ({n^2+\mu})^{-1} \le \int_0^{\infty} ({\eta^2+\mu})^{-1} {d\eta}$. Thus,
        \begin{align*}
            \int_0^{\xi^2} \sum_{n=1}^{\infty} & \dfrac{n^2}{(n^2+\mu)^2\left(1+\xi^2-\mu\right)} \, d\mu \le \int_0^{\xi^2} \sum_{n=1}^{\infty} \dfrac{1}{(n^2+\mu)\left(1+\xi^2-\mu\right)} \, d\mu \\
            &\le \int_0^{\xi^2} \dfrac{1}{1+\xi^2-\mu} \left(\int_0^{\infty} \dfrac{d\eta}{\eta^2+\mu} \right) \, d\mu \overset{\eta=t\sqrt{\mu}}{\le} \int_0^{\xi^2} \dfrac{1}{1+\xi^2-\mu} \left(\dfrac{1}{\sqrt{\mu}} \int_0^{\infty} \dfrac{1}{t^2+1} \, dt \right) \, d\mu \\
            &\lesssim \int_0^{\xi^2} \dfrac{1}{\sqrt{\mu} \left(1+\xi^2-\mu\right)} \, d\mu \overset{\sigma=\sqrt{\mu}}{\lesssim} \int_0^{|\xi|} \dfrac{1}{1+\xi^2-\sigma^2} \, d\sigma \\
            &\eqsim \dfrac{1}{\sqrt{1+\xi^2}} \int_0^{|\xi|} \left(\dfrac{1}{\sqrt{1+\xi^2}-\sigma}+\dfrac{1}{\sqrt{1+\xi^2}+\sigma}\right) \, d\sigma \\
            &\eqsim \dfrac{1}{\sqrt{1+\xi^2}} \left.\ln\left(\dfrac{\sqrt{1+\xi^2}+\sigma}{\sqrt{1+\xi^2}-\sigma} \right) \right|_0^{|\xi|} = \dfrac{1}{\sqrt{1+\xi^2}} \ln\left(\dfrac{\sqrt{1+\xi^2}+|\xi|}{\sqrt{1+\xi^2}-|\xi|} \right) \\
            &= \dfrac{2}{\sqrt{1+\xi^2}} \ln\left(\sqrt{1+\xi^2}+|\xi| \right) \lesssim \dfrac{1}{\sqrt{1+\xi^2}} \left(\sqrt{1+\xi^2}+|\xi| \right) \le 2\, ,
        \end{align*}
        which yields
        \begin{align*}
            S_1 &\lesssim \int_{-\infty}^{\infty} \int_0^{\xi^2} \left(1+\xi^2-\mu\right) \left|\widetilde{h}(\xi,\xi^2-\mu)\right|^2 \, d\mu \, d\xi \\
            &= \int_{-\infty}^{\infty} \int_0^{\xi^2} \left(1+\lambda\right) \left|\widetilde{h}(\xi,\lambda)\right|^2 \, d\lambda \, d\xi \lesssim \|h\|_{H_t^{\frac{1}{2}}\left([0,T]; \, L^2_x\right)}^2\, .
        \end{align*}
        Similarly, for (\ref{3.20}), we use $\nu^2+\xi^2$ to substitute for $\lambda$. As a result,
        \begin{align*}
            S_2 &= \int_{-\infty}^{\infty} \sum_{n=1}^{\infty} \left|\int_0^{\infty} \dfrac{2n \cdot e^{-i\pi^2(\nu^2+\xi^2) t}}{n^2-\nu^2} \cdot \nu \widetilde{h}\left(\xi,\nu^2+\xi^2\right) \left(1-\psi(n^2-\nu^2)\right) \, d\nu\right|^2 \, d\xi \\
            &\lesssim S_{2,1} + S_{2,2} + S_{2,3}
        \end{align*}
        where
        \begin{align}
            S_{2,1} &= \int_{-\infty}^{\infty} \sum_{n=1}^{\infty} \left|\int_0^{n-1} \dfrac{n \cdot e^{-i\pi^2(\nu^2+\xi^2) t}}{n^2-\nu^2} \cdot \nu \widetilde{h}\left(\xi,\nu^2+\xi^2\right) \left(1-\psi(n^2-\nu^2)\right) \, d\nu\right|^2 \, d\xi \, , \label{3.21} \\
            S_{2,2} &= \int_{-\infty}^{\infty} \sum_{n=1}^{\infty} \left|\int_{n-1}^{n+1} \dfrac{n \cdot e^{-i\pi^2(\nu^2+\xi^2) t}}{n^2-\nu^2} \cdot \nu \widetilde{h}\left(\xi,\nu^2+\xi^2\right) \left(1-\psi(n^2-\nu^2)\right) \, d\nu\right|^2 \, d\xi \, , \label{3.22} \\
            S_{2,3} &= \int_{-\infty}^{\infty} \sum_{n=1}^{\infty} \left|\int_{n+1}^{\infty} \dfrac{n \cdot e^{-i\pi^2(\nu^2+\xi^2) t}}{n^2-\nu^2} \cdot \nu \widetilde{h}\left(\xi,\nu^2+\xi^2\right) \left(1-\psi(n^2-\nu^2)\right) \, d\nu\right|^2 \, d\xi \, . \label{3.23}
        \end{align}
        For (\ref{3.21}), it is found that
        \begin{align*}
            S_{2,1} &= \int_{-\infty}^{\infty} \sum_{n=1}^{\infty} \left|\int_0^{n-1} \dfrac{2n \cdot e^{-i\pi^2(\nu^2+\xi^2) t}}{n^2-\nu^2} \cdot \nu \widetilde{h}\left(\xi,\nu^2+\xi^2\right) \left(1-\psi(n^2-\nu^2)\right) \, d\nu\right|^2 \, d\xi \\
            &\eqsim \int_{-\infty}^{\infty} \sum_{n=1}^{\infty} \left|\int_0^{n-1} \left[\dfrac{1}{n-\nu}+\dfrac{1}{n+\nu}\right] e^{-i\pi^2(\nu^2+\xi^2) t} \cdot \nu \widetilde{h}\left(\xi,\nu^2+\xi^2\right) \left(1-\psi(n^2-\nu^2)\right) \, d\nu\right|^2 \, d\xi \\
            &\lesssim \int_{-\infty}^{\infty} \sum_{n=1}^{\infty} \left|\int_0^{n-1} \dfrac{e^{-i\pi^2(\nu^2+\xi^2) t}}{n-\nu} \cdot \nu \widetilde{h}\left(\xi,\nu^2+\xi^2\right) \left(1-\psi(n^2-\nu^2)\right) \, d\nu\right|^2 \, d\xi \\
            &\lesssim \int_{-\infty}^{\infty} \sum_{n=1}^{\infty} \left(\int_0^{n-1} \left|\widetilde{h}\left(\xi,\nu^2+\xi^2\right)\right|^2 \dfrac{\nu^{2\alpha+2}}{(n-\nu)^{2-2\beta}} \, d\nu \right) \left(\int_0^{n-1} \dfrac{d\nu}{\nu^{2\alpha}(n-\nu)^{2\beta}} \right) \, d\xi\, .
        \end{align*}
        Consider the second integral factor in the summation and choose $\alpha$ and $\beta$ such that $0<\alpha$, $\beta<\frac{1}{2}$ and $1-2\alpha-2\beta<0$. Then,
        \begin{align*}
            \int_0^{n-1} & \dfrac{d\nu}{\nu^{2\alpha}(n-\nu)^{2\beta}} = \int_0^{n/2} \dfrac{d\nu}{\nu^{2\alpha}(n-\nu)^{2\beta}} + \int_{n/2}^{n-1} \dfrac{d\nu}{\nu^{2\alpha}(n-\nu)^{2\beta}} \\
            &\le \left(\dfrac{2}{n}\right)^{2\beta} \int_0^{n/2} \dfrac{d\nu}{\nu^{2\alpha}} + \left(\dfrac{2}{n}\right)^{2\alpha} \int_0^{n/2} \dfrac{d\nu}{(n-\nu)^{2\beta}} \\
            &= \left(\dfrac{2}{n}\right)^{2\beta} \left.\dfrac{\nu^{1-2\alpha}}{1-2\alpha}\right|_0^{n/2} + \left(\dfrac{2}{n}\right)^{2\alpha} \left.\dfrac{(n-\nu)^{1-2\beta}}{1-2\beta}\right|_{n/2}^{n-1} \\
            &= \left(\dfrac{n}{2}\right)^{1-2\alpha-2\beta} \left(\dfrac{1}{1-2\alpha}+\dfrac{1}{1-2\beta}\right) + \left(\dfrac{2}{n}\right)^{2\alpha} \dfrac{1}{2\beta-1} \le C\, .
        \end{align*}
        Hence, with $0<\alpha$, $\beta<\frac{1}{2}$ and $1-2\alpha-2\beta<0$,
        \begin{align*}
            S_{2,1} &\lesssim \int_{-\infty}^{\infty} \sum_{n=1}^{\infty} \int_0^{n-1} \left|\widetilde{h}\left(\xi,\nu^2+\xi^2\right)\right|^2 \dfrac{\nu^{2\alpha+2}}{(n-\nu)^{2-2\beta}} \, d\nu \, d\xi \\
            &= \int_{-\infty}^{\infty} \int_0^{\infty} \left|\widetilde{h}\left(\xi,\nu^2+\xi^2\right)\right|^2 \nu^{2\alpha+2} \left(\sum_{n=[\nu+2]}^{\infty} \dfrac{1}{(n-\nu)^{2-2\beta}}\right) \, d\nu \, d\xi \\
            &\lesssim \int_{-\infty}^{\infty} \int_0^{\infty} \left|\widetilde{h}\left(\xi,\nu^2+\xi^2\right)\right|^2 \nu^{2\alpha+2} \left(1 + \sum_{n=[\nu+3]}^{\infty} \dfrac{1}{(n-\nu)^{2-2\beta}}\right) \, d\nu \, d\xi \\
            &\lesssim \int_{-\infty}^{\infty} \int_0^{\infty} \left|\widetilde{h}\left(\xi,\nu^2+\xi^2\right)\right|^2 \nu^{2\alpha+2} \left(1 + \int_{\nu+1}^{\infty} \dfrac{d\eta}{(\eta-\nu)^{2-2\beta}}\right) \, d\nu \, d\xi \\
            &= \int_{-\infty}^{\infty} \int_0^{\infty} \left|\widetilde{h}\left(\xi,\nu^2+\xi^2\right)\right|^2 \nu^{2\alpha+2} \left(1 + \left.\dfrac{(\eta-\nu)^{2\beta-1}}{2\beta-1}\right|_{\nu+1}^{\infty}\right) \, d\nu \, d\xi \\
            &= \int_{-\infty}^{\infty} \int_0^{\infty} \left|\widetilde{h}\left(\xi,\nu^2+\xi^2\right)\right|^2 \nu^{2\alpha+2} \, d\nu \, d\xi \lesssim \int_{-\infty}^{\infty} \int_{\xi^2}^{\infty} \left(\lambda-\xi^2\right)^{\frac{1}{2}+\alpha} \left|\widetilde{h}(\xi,\lambda)\right|^2 \, d\lambda \, d\xi \\
            &\le \|h\|_{H_t^{\frac{1}{2}}\left([0,T]; \, L^2_x\right)}^2 \, .
        \end{align*}
        The symbol $[\cdot]$ represents the largest integer which is smaller or equal to the number inside. It is clear that the term in $S_{2,1}$ with $n=1$ is zero and the estimate of the term in $S_{2,2}$ with $n=1$ is given by the steps above. Thus, for (\ref{3.22}) we only need to consider the terms in $S_{2,2}$ with $n\ge2$. Also, note that $\nu \ge n-1$ implies ${\nu}^{-1} \le ({n-1})^{-1}$, and $1-\psi\neq0$ implies that $|n^2 - \nu^ 2 | \leq 1$ leading to  $\nu\le\sqrt{n^2-1}$ or $\nu\ge\sqrt{n^2+1}$. Then,
        \begin{align*}
            S_{2,2} &= \int_{-\infty}^{\infty} \sum_{n=2}^{\infty} \left|\int_{n-1}^{n+1} \dfrac{2n \cdot e^{-i\pi^2(\nu^2+\xi^2) t}}{n^2-\nu^2} \cdot \nu \widetilde{h}\left(\xi,\nu^2+\xi^2\right) \left(1-\psi(n^2-\nu^2)\right) \, d\nu\right|^2 \, d\xi \\
            &= \int_{-\infty}^{\infty} \sum_{n=2}^{\infty} \left|\int_{n-1}^{n+1} \left[\dfrac{1}{n-\nu}+\dfrac{1}{n+\nu}\right] e^{-i\pi^2(\nu^2+\xi^2) t} \cdot \nu \widetilde{h}\left(\xi,\nu^2+\xi^2\right) \left(1-\psi(n^2-\nu^2)\right) \, d\nu\right|^2 \, d\xi \\
            &\lesssim \int_{-\infty}^{\infty} \sum_{n=2}^{\infty} \left|\left(\int_{n-1}^{\sqrt{n^2-1}}+\int_{\sqrt{n^2+1}}^{n+1}\right) \dfrac{e^{-i\pi^2(\nu^2+\xi^2) t}}{n-\nu} \cdot \nu \widetilde{h}\left(\xi,\nu^2+\xi^2\right) \left(1-\psi(n^2-\nu^2)\right) \, d\nu\right|^2 \, d\xi \\
            &\lesssim \int_{-\infty}^{\infty} \sum_{n=2}^{\infty} \left(\int_{n-1}^{\sqrt{n^2-1}} \left|\widetilde{h}\left(\xi,\nu^2+\xi^2\right)\right|^2 \nu^3 \, d\nu \right) \left(\int_{n-1}^{\sqrt{n^2-1}} \dfrac{d\nu}{\nu(n-\nu)^2} \right) \, d\xi \\
            & \qquad + \int_{-\infty}^{\infty} \sum_{n=2}^{\infty} \left(\int_{\sqrt{n^2+1}}^{n+1} \left|\widetilde{h}\left(\xi,\nu^2+\xi^2\right)\right|^2 \nu^3 \, d\nu \right) \left(\int_{\sqrt{n^2+1}}^{n+1} \dfrac{d\nu}{\nu(n-\nu)^2} \right) \, d\xi \\
            &\le \int_{-\infty}^{\infty} \sum_{n=2}^{\infty} \left(\int_{n-1}^{\sqrt{n^2-1}} \left|\widetilde{h}\left(\xi,\nu^2+\xi^2\right)\right|^2 \nu^3 \, d\nu \right) \left.\dfrac{1}{(n-1)(n-\nu)} \right|_{n-1}^{\sqrt{n^2-1}} \, d\xi \\
            & \qquad + \int_{-\infty}^{\infty} \sum_{n=2}^{\infty} \left(\int_{\sqrt{n^2+1}}^{n+1} \left|\widetilde{h}\left(\xi,\nu^2+\xi^2\right)\right|^2 \nu^3 \, d\nu \right) \left.\dfrac{1}{(n-1)(n-\nu)} \right|_{\sqrt{n^2+1}}^{n+1} \, d\xi \\
            &= \int_{-\infty}^{\infty} \sum_{n=2}^{\infty} \left(\int_{n-1}^{\sqrt{n^2-1}} \left|\widetilde{h}\left(\xi,\nu^2+\xi^2\right)\right|^2 \nu^3 \, d\nu \right) \left(\dfrac{n+\sqrt{n^2-1}}{n-1}-1 \right) \, d\xi \\
            & \qquad + \int_{-\infty}^{\infty} \sum_{n=2}^{\infty} \left(\int_{\sqrt{n^2+1}}^{n+1} \left|\widetilde{h}\left(\xi,\nu^2+\xi^2\right)\right|^2 \nu^3 \, d\nu \right) \left(\dfrac{\sqrt{n^2+1}+n}{n-1}-1 \right) \, d\xi \\
            &\lesssim \int_{-\infty}^{\infty} \sum_{n=2}^{\infty} \left(\int_{n-1}^{\sqrt{n^2-1}}+\int_{\sqrt{n^2+1}}^{n+1}\right) \left|\widetilde{h}\left(\xi,\nu^2+\xi^2\right)\right|^2 \nu^3 \, d\nu \, d\xi \\
            &\lesssim \int_{-\infty}^{\infty} \int_0^{\infty} \left|\widetilde{h}\left(\xi,\nu^2+\xi^2\right)\right|^2 \nu^2 \, d\nu^2 \, d\xi \le \int_{-\infty}^{\infty} \int_{\xi^2}^{\infty} \left|\widetilde{h}(\xi,\lambda)\right|^2 (\lambda-\xi^2) \, d\lambda \, d\xi \le \|h\|_{H_t^{\frac{1}{2}}\left([0,T]; \, L^2_x\right)}^2\, .
        \end{align*}
        Finally, the only part left to show is the estimate for $S_{2,3}$  in (\ref{3.23}).
        \begin{align*}
            S_{2,3} &= \int_{-\infty}^{\infty} \sum_{n=1}^{\infty} \left|\int_{n+1}^{\infty} \dfrac{2n \cdot e^{-i\pi^2(\nu^2+\xi^2) t}}{n^2-\nu^2} \cdot \nu \widetilde{h}\left(\xi,\nu^2+\xi^2\right) \left(1-\psi(n^2-\nu^2)\right) \, d\nu\right|^2 \, d\xi \\
            &= \int_{-\infty}^{\infty} \sum_{n=1}^{\infty} \left|\int_{n+1}^{\infty} \left[\dfrac{1}{\nu-n}-\dfrac{1}{\nu+n}\right] e^{-i\pi^2(\nu^2+\xi^2) t} \cdot \nu \widetilde{h}\left(\xi,\nu^2+\xi^2\right) \left(1-\psi(n^2-\nu^2)\right) \, d\nu\right|^2 \, d\xi \\
            &\lesssim \int_{-\infty}^{\infty} \sum_{n=1}^{\infty} \left(\int_{n+1}^{\infty} \left|\widetilde{h}\left(\xi,\nu^2+\xi^2\right)\right|^2 \dfrac{\nu^{2\alpha+2}}{(\nu-n)^{2-2\beta}} \, d\nu \right) \left(\int_{n+1}^{\infty} \dfrac{d\nu}{\nu^{2\alpha}(\nu-n)^{2\beta}} \right) \, d\xi\, .
        \end{align*}
        Similarly, we need to work on the second factor inside the summation. Let $0<\alpha$, $\beta<\frac{1}{2}$ and $1-2\alpha-2\beta<0$. Also, $\nu>2n$ guarantees that $2(\nu-n)>\nu$. Thus,
        \begin{align*}
            \int_{n+1}^{\infty} &\dfrac{d\nu}{\nu^{2\alpha}(\nu-n)^{2\beta}} = \int_{n+1}^{2n} \dfrac{d\nu}{\nu^{2\alpha}(\nu-n)^{2\beta}} + \int_{2n}^{\infty} \dfrac{d\nu}{\nu^{2\alpha}(\nu-n)^{2\beta}} \\
            &\le \dfrac{1}{(2n)^{2\alpha}} \int_{n+1}^{2n} \dfrac{d\nu}{(\nu-n)^{2\beta}} + 2^{2\beta} \int_{2n}^{\infty} \dfrac{d\nu}{\nu^{2\alpha+2\beta}} \\
            &\le \dfrac{1}{n^{2\alpha}} \left.\dfrac{(\nu-n)^{1-2\beta}}{1-2\beta}\right|_{n+1}^{2n} + 2^{2\beta} \left.\dfrac{\nu^{1-2\alpha-2\beta}}{2\alpha+2\beta-1}\right|_{\infty}^{2n} \lesssim n^{1-2\alpha-2\beta} \le C\, .
        \end{align*}
        Therefore, (note that $\nu\ge2$),
        \begin{align*}
            S_{2,3} &\lesssim \int_{-\infty}^{\infty} \sum_{n=1}^{\infty} \int_{n+1}^{\infty} \left|\widetilde{h}\left(\xi,\nu^2+\xi^2\right)\right|^2 \dfrac{\nu^{2\alpha+2}}{(\nu-n)^{2-2\beta}} \, d\nu \, d\xi \\
            &= \int_{-\infty}^{\infty} \int_2^{\infty} \left|\widetilde{h}\left(\xi,\nu^2+\xi^2\right)\right|^2 \nu^{2\alpha+2} \left(\sum_{n=1}^{[\nu-1]} \dfrac{1}{(\nu-n)^{2-2\beta}}\right) \, d\nu \, d\xi \\
            &\le \int_{-\infty}^{\infty} \int_2^{\infty} \left|\widetilde{h}\left(\xi,\nu^2+\xi^2\right)\right|^2 \nu^{2\alpha+2} \left(1 + \sum_{n=1}^{[\nu-2]} \dfrac{1}{(\nu-n)^{2-2\beta}}\right) \, d\nu \, d\xi \\
            &\le \int_{-\infty}^{\infty} \int_2^{\infty} \left|\widetilde{h}\left(\xi,\nu^2+\xi^2\right)\right|^2 \nu^{2\alpha+2} \left(1 + \int_1^{\nu-1} \dfrac{d\eta}{(\nu-\eta)^{2-2\beta}}\right) \, d\nu \, d\xi \\
            &= \int_{-\infty}^{\infty} \int_2^{\infty} \left|\widetilde{h}\left(\xi,\nu^2+\xi^2\right)\right|^2 \nu^{2\alpha+2} \left(1 + \left.\dfrac{1}{(1-2\beta)(\nu-\eta)^{1-2\beta}}\right|_{\nu-1}^1 \right) \, d\nu \, d\xi \\
            &\eqsim \int_{-\infty}^{\infty} \int_2^{\infty} \left|\widetilde{h}\left(\xi,\nu^2+\xi^2\right)\right|^2 \nu^{2\alpha+1} \, d\nu^2 \, d\xi = \int_{-\infty}^{\infty} \int_0^{\infty} \left|\widetilde{h}(\xi,\lambda)\right|^2 \left(\lambda-\xi^2\right)^{\alpha+\frac{1}{2}} \, d\lambda \, d\xi \\
            &\le \int_{-\infty}^{\infty} \int_0^{\infty} \lambda \left|\widetilde{h}(\xi,\lambda)\right|^2 \, d\lambda \, d\xi \le \|h\|^2_{H_t^{\frac{1}{2}}\left([0,T]; \, L^2_x\right)} \, .
        \end{align*}
        By adding the estimates for $S_{2,1}$ to $S_{2,3}$, we obtain that $S_2 \lesssim \|h\|^2_{H_t^{\frac{1}{2}}\left([0,T]; \, L^2_x\right)}$.
        Hence,
        \begin{equation}\label{3.24}
            \sup_{0\le t \le T} \|I^+_2(t)\|_{L^2_{xy}} \lesssim \|h\|_{H^{\frac{1}{2}}_t \left([0,T]; \, L^2_x\right)}\, \, .
        \end{equation}

        To deal with the estimates of $I_2^+$ in terms of $L^4$-norm, we split $I^+_2$ as follows,
        \begin{align*}
            I^+_2 &=  I^+_{2,1} + I^+_{2,2} + I^+_{2,3}
        \end{align*}
        where
        \begin{align}
            I^+_{2,1} &= \int_{-\infty}^{\infty} \sum_n \left(e^{i\pi(x\xi+ny)} \cdot n \int_{\xi^2+n^2}^{\infty} \dfrac{e^{-i\pi^2\lambda t}}{\xi^2+n^2-\lambda} \widetilde{h}(\xi,\lambda) \left(1-\psi\left(\xi^2+n^2-\lambda\right)\right) \, d\lambda\right) \, d\xi\, , \label{3.25} \\
            I^+_{2,2} &= \int_{-\infty}^{\infty} \sum_n \left(e^{i\pi(x\xi+ny)} \cdot n \int_{\xi^2+\frac{n^2}{2}}^{\xi^2+n^2} \dfrac{e^{-i\pi^2\lambda t}}{\xi^2+n^2-\lambda} \widetilde{h}(\xi,\lambda) \left(1-\psi\left(\xi^2+n^2-\lambda\right)\right) \, d\lambda\right) \, d\xi \, , \label{3.26} \\
            I^+_{2,3} &= \int_{-\infty}^{\infty} \sum_n \left(e^{i\pi(x\xi+ny)} \cdot n \int_0^{\xi^2+\frac{n^2}{2}} \dfrac{e^{-i\pi^2\lambda t}}{\xi^2+n^2-\lambda} \widetilde{h}(\xi,\lambda) \left(1-\psi\left(\xi^2+n^2-\lambda\right)\right) \, d\lambda\right) \, d\xi\, . \label{3.27}
        \end{align}
        From (\ref{3.25}), we have
        \begin{align*}
            I^+_{2,1} = & \int_{-\infty}^{\infty} \sum_n \left(e^{i\pi(x\xi+ny)} \cdot n \int_{\xi^2+n^2}^{\infty} \dfrac{e^{-i\pi^2\lambda t}}{\xi^2+n^2-\lambda} \widetilde{h}(\xi,\lambda) \left(1-\psi\left(\xi^2+n^2-\lambda\right)\right) \, d\lambda\right) \, d\xi \\
            = & \int_{-\infty}^{\infty} \sum_n \bigg (e^{-i\pi^2\left(\xi^2+n^2\right)t+i\pi(x\xi+ny)}  \\
            & \qquad \qquad \times \int_{\xi^2+n^2}^{\infty} \dfrac{n e^{i\pi^2\left(\xi^2+n^2-\lambda\right) t}}{\left(\xi^2+n^2-\lambda\right)} \widetilde{h}(\xi,\lambda) \left(1-\psi\left(\xi^2+n^2-\lambda\right)\right) \, d\lambda\bigg ) \, d\xi\, .
        \end{align*}
        Then, choose $\mu=\xi^2+n^2-\lambda$ and $s=\sqrt{\lambda-\xi^2}$. By (\ref{3.5}), it is obtained that for $\frac{1}{2} < \sigma < 1$
        \begin{align*}
            \|I^+_{2,1}\|&_{L^4(\R \times [0,1] \times [0,T])}^2 \\
            &\lesssim \left\|\int_{\xi^2+n^2}^{\infty} \dfrac{n e^{i\pi^2\left(\xi^2+n^2-\lambda\right) t}}{\xi^2+n^2-\lambda} \widetilde{h}(\xi,\lambda) \left(1-\psi\left(\xi^2+n^2-\lambda\right)\right) \, d\lambda\right\|^2_{H_t^{\sigma}\left([0,T]; \, L^2_{\xi n}\right)} \\
            &= \left\|\int_{-\infty}^{\infty} \dfrac{n e^{i\pi^2\left(\xi^2+n^2-\lambda\right) t}}{\xi^2+n^2-\lambda} \chi_{\left[\xi^2+n^2,\infty\right)}(\lambda) \cdot \widetilde{h}(\xi,\lambda) \left(1-\psi\left(\xi^2+n^2-\lambda\right)\right) \, d\lambda\right\|^2_{H_t^{\sigma}\left([0,T]; \, L^2_{\xi n}\right)} \\
            &= \left\|\int_{-\infty}^{\infty} \dfrac{n e^{i\pi^2\mu t}}{\mu} \chi_{(-\infty, 0]}(\mu) \cdot \widetilde{h}\left(\xi,\xi^2+n^2-\mu\right) \left(1-\psi(\mu)\right) \, d\mu\right\|^2_{H_t^{\sigma}\left([0,T]; \, L^2_{\xi n}\right)}\, .
        \end{align*}
        Since it can be easily shown that the term only involving with the cut-off function $\psi$ is bounded by
        $\|h \|_{H_t^{1/2}([0,T]; \, L^2_x)}$ up to a constant based upon the study of (\ref{3.18}) for $I^+_1$, we may focus on the part without the cut-off function.
        \begin{align*}
            & \left\|\int_{-\infty}^{\infty} \dfrac{n e^{i\pi^2\mu t}}{\mu} \chi_{(-\infty, -1]}(\mu) \cdot \widetilde{h}\left(\xi,\xi^2+n^2-\mu\right) \, d\mu\right\|^2_{H_t^{\sigma}\left([0,T]; \, L^2_{\xi n}\right)} \\
            &\eqsim \int_{-\infty}^{\infty} \sum_n \int_{-\infty}^{\infty} \left|\dfrac{n (1+|\mu|)^{\sigma}}{|\mu|} \chi_{(-\infty, -1]}(\mu) \cdot \widetilde{h}\left(\xi,\xi^2+n^2-\mu\right)\right|^2 \, d\mu \, d\xi \\
            &\le \int_{-\infty}^{\infty} \sum_n \int_{-\infty}^{\infty} \left|\dfrac{n}{(1+|\mu|)^{1-\sigma}} \chi_{(-\infty, -1]}(\mu) \cdot \widetilde{h}\left(\xi,\xi^2+n^2-\mu\right)\right|^2 \, d\mu \, d\xi \\
            &= \int_{-\infty}^{\infty} \sum_n \int_{-\infty}^{\infty} \left|\dfrac{n}{(1+\left|\lambda-\xi^2-n^2\right|)^{1-\sigma}} \chi_{[\xi^2+n^2+1, \infty)}(\lambda) \cdot \widetilde{h}(\xi,\lambda)\right|^2 \, d\lambda \, d\xi \\
            &\eqsim \int_{-\infty}^{\infty} \sum_n \int_{-\infty}^{\infty} \left|\dfrac{n}{(1+s^2-n^2)^{1-\sigma}} \chi_{[\sqrt{n^2+1}, \infty)}(s) \cdot \widetilde{h}\left(\xi,s^2+\xi^2\right)\right|^2 s \, ds \, d\xi \\
            &= \int_{-\infty}^{\infty} \sum_n \left(\int_{\sqrt{n^2+1}}^{n+1} + \int_{n+1}^{\infty} \right) \dfrac{n^2}{(1+s^2-n^2)^{2-2\sigma}} \cdot \left|\widetilde{h}\left(\xi,s^2+\xi^2\right)\right|^2 s \, ds \, d\xi\, .
        \end{align*}
        The first integral can be estimated by
        \begin{align*}
            \int_{-\infty}^{\infty}& \sum_n \int_{\sqrt{n^2+1}}^{n+1} \dfrac{n^2}{(1+s^2-n^2)^{2-2\sigma}} \cdot \left|\widetilde{h}\left(\xi,s^2+\xi^2\right)\right|^2 s \, ds \, d\xi \\
            &\le \int_{-\infty}^{\infty} \sum_n \int_{\sqrt{n^2+1}}^{n+1} n^2 \cdot \left|\widetilde{h}\left(\xi,s^2+\xi^2\right)\right|^2 s \, ds \, d\xi \\
            &\le \int_{-\infty}^{\infty} \sum_n \int_{\sqrt{n^2+1}}^{n+1} s^3 \cdot \left|\widetilde{h}\left(\xi,s^2+\xi^2\right)\right|^2 \, ds \, d\xi \le \int_{-\infty}^{\infty} \int_1^{\infty} s^3 \cdot \left|\widetilde{h}\left(\xi,s^2+\xi^2\right)\right|^2 \, ds \, d\xi \\
            &\lesssim \int_{-\infty}^{\infty} \int_{1+\xi^2}^{\infty} \left(\lambda-\xi^2\right) \cdot \left|\widetilde{h}(\xi,\lambda)\right|^2 \, d\lambda \, d\xi \lesssim \|h\|_{H_t^{\frac{1}{2}}([0,T]; \, L^2_x)}^2\, ,
        \end{align*}
        and the second integral satisfies
        \begin{align*}
            \int_{-\infty}^{\infty} &\sum_n \int_{n+1}^{\infty} \dfrac{n^2}{(1+s^2-n^2)^{2-2\sigma}} \cdot \left|\widetilde{h}\left(\xi,s^2+\xi^2\right)\right|^2 s \, ds \, d\xi \\
            &\lesssim \int_{-\infty}^{\infty} \sum_n \int_{n+1}^{\infty} \left(\dfrac{n}{s^2-n^2} \right)^{2-2\sigma} n^{2\sigma} \cdot \left|\widetilde{h}\left(\xi,s^2+\xi^2\right)\right|^2 s \, ds \, d\xi \\
            &\lesssim \int_{-\infty}^{\infty} \sum_{n\neq0} \int_{n+1}^{\infty} \dfrac{1}{(s-n)^{2-2\sigma}} \cdot s^{2\sigma+1} \left|\widetilde{h}\left(\xi,s^2+\xi^2\right)\right|^2 \, ds \, d\xi \\
            &= \int_{-\infty}^{\infty} \int_2^{\infty} s^{2\sigma+1} \left|\widetilde{h}\left(\xi,s^2+\xi^2\right)\right|^2 \left(\sum_{n=1}^{\left\lfloor s-1 \right\rfloor} \dfrac{1}{(s-n)^{2-2\sigma}} \right) \, ds \, d\xi \\
            &\le \int_{-\infty}^{\infty} \int_2^{\infty} s^{2\sigma+1} \left|\widetilde{h}\left(\xi,s^2+\xi^2\right)\right|^2 \left(\int_0^{s-1} \dfrac{d\eta}{(s-\eta)^{2-2\sigma}} \right) \, ds \, d\xi \\
            &= \int_{-\infty}^{\infty} \int_2^{\infty} s^{2\sigma+1} \left|\widetilde{h}\left(\xi,s^2+\xi^2\right)\right|^2 \left(\left.\dfrac{(s-\eta)^{2\sigma-1}}{1-2\sigma} \right|_0^{s-1} \right) \, ds \, d\xi \\
            &\lesssim \int_{-\infty}^{\infty} \int_2^{\infty} s^{2\sigma+1} \left|\widetilde{h}\left(\xi,s^2+\xi^2\right)\right|^2 s^{2\sigma-1} \, ds \, d\xi \\
            &= \int_{-\infty}^{\infty} \int_2^{\infty} s^{4\sigma} \left|\widetilde{h}\left(\xi,s^2+\xi^2\right)\right|^2 \, ds \, d\xi \eqsim \int_{-\infty}^{\infty} \int_2^{\infty} s^{4\sigma-1} \left|\widetilde{h}\left(\xi,s^2+\xi^2\right)\right|^2 \, ds^2 \, d\xi \\
            &= \int_{-\infty}^{\infty} \int_{4+\xi^2}^{\infty} \left(\lambda-\xi^2\right)^{2\sigma-\frac{1}{2}} \left|\widetilde{h}(\xi,\lambda)\right|^2 \, d\lambda \, d\xi \, .
        \end{align*}
        With $\sigma=\frac{3}{4} \in \left(\frac{1}{2},1\right)$, it is derived that
        \begin{equation}\label{3.28}
            \|I^+_{2,1}\|_{L^4(\R \times [0,1] \times [0,T])} \lesssim \|h\|_{H_t^{\frac{1}{2}}([0,T]; \, L^2_x)}\, .
        \end{equation}
        Similarly, we may also control the $L^4$-norm of $I^+_{2,2}$ in (\ref{3.26}) by the boundary data in a suitable norm. Note that the terms in $I^+_{2,2}$ equal to nonzero numbers only when $n\ge2$. Let $s=\sqrt{\lambda-\xi^2}$. Then, as we did for $I^+_{2,1}$, (\ref{3.5}) yields
        \begin{align*}
            \|I^+_{2,2}\|&_{L^4(\R \times [0,1] \times [0,T])}^2 \\
            &\lesssim \left\|\int_{\xi^2+\frac{n^2}{2}}^{\xi^2+n^2} \dfrac{n e^{i\pi^2\left(\xi^2+n^2-\lambda\right) t}}{\xi^2+n^2-\lambda} \widetilde{h}(\xi,\lambda) \left(1-\psi\left(\xi^2+n^2-\lambda\right)\right) \, d\lambda\right\|^2_{H_t^{\sigma}\left([0,T]; \, L^2_{\xi n}\right)} \\
            &\lesssim \int_{-\infty}^{\infty} \sum_n \int_{-\infty}^{\infty} \left|\dfrac{n}{(1+\left|\xi^2+n^2-\lambda\right|)^{1-\sigma}} \chi_{\left[\xi^2+\frac{n^2}{2}, \xi^2+n^2-1\right)}(\lambda) \cdot \widetilde{h}(\xi,\lambda)\right|^2 \, d\lambda \, d\xi \\
            &\eqsim \int_{-\infty}^{\infty} \sum_n \int_{-\infty}^{\infty} \left|\dfrac{n}{(1+n^2-s^2)^{1-\sigma}} \chi_{\left[\frac{n}{\sqrt{2}},\sqrt{n^2-1}\right)}(s) \cdot \widetilde{h}\left(\xi,s^2+\xi^2\right)\right|^2 s \, ds \, d\xi \\
            &= \int_{-\infty}^{\infty} \sum_n \left(\int_{n-1}^{\sqrt{n^2-1}} + \int_{\frac{n}{\sqrt{2}}}^{n-1} \right) \dfrac{n^2}{(1+n^2-s^2)^{2-2\sigma}} \cdot \left|\widetilde{h}\left(\xi,s^2+\xi^2\right)\right|^2 s \, ds \, d\xi \, .
        \end{align*}
        Note that the integral with respect to $s$ over the interval $\left[\frac{n}{\sqrt{2}},n-1\right]$ exists only for $n\ge4$. First, it is seen that any two elements in the sequence $\left\{\left[n-1,\sqrt{n^2-1}\right]\right\}_{n\in\Z}$ are disjoint. Thus,
        \begin{align*}
            \int_{-\infty}^{\infty} &\sum_{n\ge2} \int_{n-1}^{\sqrt{n^2-1}} \dfrac{n^2}{(1+n^2-s^2)^{2-2\sigma}} \cdot \left|\widetilde{h}\left(\xi,s^2+\xi^2\right)\right|^2 s \, ds \, d\xi \\
            &\le \int_{-\infty}^{\infty} \sum_{n\ge2} \int_{n-1}^{\sqrt{n^2-1}} n^2 \cdot \left|\widetilde{h}\left(\xi,s^2+\xi^2\right)\right|^2 s \, ds \, d\xi \\
            &\le \int_{-\infty}^{\infty} \sum_{n\ge2} \int_{n-1}^{\sqrt{n^2-1}} s(1+s)^2 \cdot \left|\widetilde{h}\left(\xi,s^2+\xi^2\right)\right|^2 \, ds \, d\xi \\
            &\le \int_{-\infty}^{\infty} \int_1^{\infty} s(1+s)^2 \cdot \left|\widetilde{h}\left(\xi,s^2+\xi^2\right)\right|^2 \, ds \, d\xi \\
            &\lesssim \int_{-\infty}^{\infty} \int_{1+\xi^2}^{\infty} \left(1+\lambda-\xi^2\right) \cdot \left|\widetilde{h}(\xi,\lambda)\right|^2 \, d\lambda \, d\xi \le \|h\|_{H_t^{\frac{1}{2}}([0,T]; \, L^2_x)}^2 \, .
        \end{align*}
        For $n\ge4$,
        \begin{align*}
            \int_{-\infty}^{\infty} &\sum_{n\ge4} \int_{\frac{n}{\sqrt{2}}}^{n-1} \dfrac{n^2}{(1+n^2-s^2)^{2-2\sigma}} \cdot \left|\widetilde{h}\left(\xi,s^2+\xi^2\right)\right|^2 s \, ds \, d\xi \\
            &\le \int_{-\infty}^{\infty} \sum_{n\ge4} \int_{\frac{n}{\sqrt{2}}}^{n-1} \left(\dfrac{n}{n^2-s^2} \right)^{2-2\sigma} n^{2\sigma} \cdot \left|\widetilde{h}\left(\xi,s^2+\xi^2\right)\right|^2 s \, ds \, d\xi \\
            &\lesssim \int_{-\infty}^{\infty} \sum_{n\ge4} \int_{\frac{n}{\sqrt{2}}}^{n-1} \dfrac{s^{2\sigma+1}}{(n-s)^{2-2\sigma}} \cdot \left|\widetilde{h}\left(\xi,s^2+\xi^2\right)\right|^2 \, ds \, d\xi \\
            &= \int_{-\infty}^{\infty} \int_{2\sqrt{2}}^{\infty} s^{2\sigma+1} \left|\widetilde{h}\left(\xi,s^2+\xi^2\right)\right|^2 \left(\sum_{n=\left\lceil s+1 \right\rceil}^{\left\lfloor \sqrt{2}s \right\rfloor} \dfrac{1}{(n-s)^{2-2\sigma}} \right) \, ds \, d\xi \\
            &\le \int_{-\infty}^{\infty} \int_{2\sqrt{2}}^{\infty} s^{2\sigma+1} \left|\widetilde{h}\left(\xi,s^2+\xi^2\right)\right|^2 \left(\int_{s+1}^{\sqrt{2}s} \dfrac{d\eta}{(\eta-s)^{2-2\sigma}} \right) \, ds \, d\xi \\
            &= \int_{-\infty}^{\infty} \int_{2\sqrt{2}}^{\infty} s^{2\sigma+1} \left|\widetilde{h}\left(\xi,s^2+\xi^2\right)\right|^2 \left(\left.\dfrac{(\eta-s)^{2\sigma-1}}{2\sigma-1} \right|_{s+1}^{\sqrt{2}s} \right) \, ds \, d\xi \\
            &\lesssim \int_{-\infty}^{\infty} \int_{2\sqrt{2}}^{\infty} s^{4\sigma} \left|\widetilde{h}\left(\xi,s^2+\xi^2\right)\right|^2 \, ds \, d\xi \\
            &\eqsim \int_{-\infty}^{\infty} \int_{2\sqrt{2}}^{\infty} (s^2)^{2\sigma-\frac{1}{2}} \left|\widetilde{h}\left(\xi,s^2+\xi^2\right)\right|^2 \, ds^2 \, d\xi \\
            &= \int_{-\infty}^{\infty} \int_{2\sqrt{2}+\xi^2}^{\infty} \left(\lambda-\xi^2\right)^{2\sigma-\frac{1}{2}} \left|\widetilde{h}(\xi,\lambda)\right|^2 \, d\lambda \, d\xi \, .
        \end{align*}
        Again, by letting $\sigma=\frac{3}{4}$, we obtain
        \begin{equation}\label{3.29}
            \|I^+_{2,2}\|_{L^4(\R \times [0,1] \times [0,T])} \lesssim \|h\|_{H_t^{\frac{1}{2}}([0,T]; \, L^2_x)} \, .
        \end{equation}
        Lastly, we study $I^+_{2,3}$ in (\ref{3.27}).
        Write
        \begin{align*}
            I^+_{2,3} &\eqsim \int_{-\infty}^{\infty} \int_0^{\infty} \widetilde{h}(\xi,\lambda) e^{-i\pi^2\lambda t + i\pi x \xi} \sum_{n=1}^{\infty} \left[ n \sin(n\pi y) \cdot \chi_{\left[0, \xi^2+\frac{n^2}{2}\right]}(\lambda) \dfrac{1-\psi\left(\xi^2+n^2-\lambda\right)}{\xi^2+n^2-\lambda} \right] \, d\lambda \, d\xi \\
            &= K_1 + K_2
        \end{align*}
        where
        \begin{align}
            K_1 &= \int_{-\infty}^{\infty} \int_0^{\infty} \widetilde{h}(\xi,\lambda) e^{-i\pi^2\lambda t + i\pi x \xi} \left[\sum_{n=1}^{\infty} n \sin(n\pi y) \cdot \chi_{\left[\xi^2, \xi^2+\frac{n^2}{2}\right]}(\lambda) \dfrac{1-\psi\left(\xi^2+n^2-\lambda\right)}{\xi^2+n^2-\lambda} \right] \, d\lambda \, d\xi \, , \label{3.30} \\
            K_2 &= \int_{-\infty}^{\infty} \int_0^{\infty} \widetilde{h}(\xi,\lambda) e^{-i\pi^2\lambda t + i\pi x \xi} \chi_{\left[0,\xi^2\right]}(\lambda) \left[\sum_{n=1}^{\infty} n \sin(n\pi y) \cdot \dfrac{1-\psi\left(\xi^2+n^2-\lambda\right)}{\xi^2+n^2-\lambda} \right] \, d\lambda \, d\xi \, . \label{3.31}
        \end{align}
        Choose $\sigma \in \left(\frac{3}{4},1\right)$. By \emph{Proposition 3.6} in \cite{Sun_NLSE_1d}, we have the inequality
        \begin{equation*}
            \left|\sum_{n=1}^{\infty} n \sin(n\pi y) \cdot \chi_{\left[0, \frac{n^2}{2}\right]}(\mu) \dfrac{1-\psi(n^2-\mu)}{n^2-\mu}\right| \lesssim \dfrac{|y|^{\sigma-1}}{(1+\sqrt{\mu})^{1-\sigma}} \, , \qquad 0 \le \sigma \le 1 \, .
        \end{equation*}
        Replace $\mu$ by $\lambda-\xi^2$ to obtain
        \begin{equation}\label{3.32}
            \left|\sum_{n=1}^{\infty} n \sin(n\pi y) \cdot \chi_{\left[\xi^2,\xi^2+\frac{n^2}{2}\right]}(\lambda) \dfrac{1-\psi\left(\xi^2+n^2-\lambda\right)}{\xi^2+n^2-\lambda}\right| \lesssim \dfrac{|y|^{\sigma-1}}{\left(1+\sqrt{\lambda-\xi^2}\right)^{1-\sigma}} \, , \qquad 0 \le \sigma \le 1 \, .
        \end{equation}
        Recall the Hausdorff-Young inequality $\left\|\mathcal{F}[f]\right\|_{L^r} \lesssim \|f\|_{L^{r'}}$ for each $r\in[2, \infty]$ and $f\in L^{r'}$ where $r'=\frac{r}{r-1}$. If $r=4$, then for (\ref{3.30}), by (\ref{3.32}), we have
        \begin{align*}
            & \|K_1\|_{L^4_{xyt}}^4 \lesssim \int_0^1\left\|\widetilde{h}(\xi,\lambda) \cdot \left[\sum_{n=1}^{\infty} n \sin(n\pi y) \cdot \chi_{\left[\xi^2, \xi^2+\frac{n^2}{2}\right]}(\lambda) \dfrac{1-\psi\left(\xi^2+n^2-\lambda\right)}{\xi^2+n^2-\lambda} \right] \right\|_{L^{\frac{4}{3}}_{\xi \lambda}}^4 \, dy \\
            &\lesssim \int_0^1 \left\{\int_{-\infty}^{\infty} \int_{\xi^2}^{\infty} \left|\widetilde{h}(\xi,\lambda)\right|^{\frac{4}{3}} \cdot \left|\sum_{n=1}^{\infty} n \sin(n\pi y) \cdot \chi_{\left[\xi^2,\xi^2+\frac{n^2}{2}\right]}(\lambda) \dfrac{1-\psi\left(\xi^2+n^2-\lambda\right)}{\xi^2+n^2-\lambda}\right|^{\frac{4}{3}} \, d\lambda \, d\xi\right\}^3 \, dy \\
            &{\lesssim} \int_0^1 \left\{\int_{-\infty}^{\infty} \int_{\xi^2}^{\infty} \left|\widetilde{h}(\xi,\lambda)\right|^{\frac{4}{3}} \cdot \left(\dfrac{|y|^{\sigma-1}}{\left(1+\sqrt{\lambda-\xi^2}\right)^{1-\sigma}}\right)^{\frac{4}{3}} \, d\lambda \, d\xi\right\}^3 \, dy \\
            &\le \int_0^1 \left[\int_{-\infty}^{\infty} \int_{\xi^2}^{\infty} \left|\widetilde{h}(\xi,\lambda)\right|^{2} (1+\lambda) \, d\lambda \, d\xi \right]^2 \cdot \left[\int_{-\infty}^{\infty} \int_{\xi^2}^{\infty} (1+\lambda)^{-2} \dfrac{|y|^{4\sigma-4}}{\left(1+\sqrt{\lambda-\xi^2}\right)^{4-4\sigma}} \, d\lambda \, d\xi \right] \, dy \\
            &\lesssim \left[\int_{-\infty}^{\infty} \int_{\xi^2}^{\infty} \left|\widetilde{h}(\xi,\lambda)\right|^{2} (1+\lambda) \, d\lambda \, d\xi \right]^2 \cdot \left[\int_{-\infty}^{\infty} \int_{\xi^2}^{\infty} \dfrac{(1+\lambda)^{-2}}{\left(1+\sqrt{\lambda-\xi^2}\right)^{4-4\sigma}} \, d\lambda \, d\xi \right] \\
            &\lesssim \|h\|_{H^{\frac{1}{2}}_t([0,T]; \, L^2_x)}^4 \cdot \left[\int_{-\infty}^{\infty} \int_{\xi^2}^{\infty} \dfrac{1}{(1+\lambda-\xi^2)^{2-2\sigma}(1+\lambda)^2} \, d\lambda \, d\xi \right] \\
            &\le \|h\|_{H^{\frac{1}{2}}_t([0,T]; \, L^2_x)}^4 \cdot \bigg [\int_{-\infty}^{\infty} \int_{\xi^2}^{2\xi^2} \dfrac{1}{(1+\lambda-\xi^2)^{2-2\sigma}(1+\lambda)^2} \, d\lambda \, d\xi \\
            & \qquad \qquad + \int_{-\infty}^{\infty} \int_{2\xi^2}^{\infty} \dfrac{1}{(1+\lambda-\xi^2)^{2-2\sigma}(1+\lambda)^2} \, d\lambda \, d\xi\bigg ] \\
            &\le \|h\|_{H^{\frac{1}{2}}_t([0,T]; \, L^2_x)}^4 \cdot \left[\int_{-\infty}^{\infty} \int_{\xi^2}^{2\xi^2} \dfrac{1}{(1+\xi^2)^2} \, d\lambda \, d\xi + \int_{-\infty}^{\infty} \int_{2\xi^2}^{\infty} \dfrac{1}{(1+\xi^2)^{2-2\sigma}(1+\lambda)^2} \, d\lambda \, d\xi\right] \\
            &\le \|h\|_{H^{\frac{1}{2}}_t([0,T]; \, L^2_x)}^4 \cdot \left[\left.\tan^{-1}(\xi)\right|_{-\infty}^{\infty} + \int_{-\infty}^{\infty} \dfrac{1}{(1+\xi^2)^{2-2\sigma}} \int_{2\xi^2}^{\infty} \frac{1}{(1+\lambda)^2} \, d\lambda \, d\xi\right] \\
            &\lesssim \|h\|_{H^{\frac{1}{2}}_t([0,T]; \, L^2_x)}^4 \cdot \left[1 + \int_{-\infty}^{\infty} \dfrac{1}{(1+\xi^2)^{3-2\sigma}} \, d\xi\right] \lesssim \|h\|_{H^{\frac{1}{2}}_t([0,T]; \, L^2_x)}^4 \, .
        \end{align*}
        To study $K_2$, we use the following identity,
        $$
        \sum^\infty_{n =1} \frac{n \sin nx }{n^2 + a^2} = \frac\pi 2 \frac{\sinh a (\pi - x ) }{ \sinh a \pi },\qquad 0 < x < 2\pi \, .
        $$
        Consider $0<\lambda\le\xi^2$ and $\mu = \xi^2-\lambda > 0$.
        \begin{align*}
            & \sum_{n=1}^{\infty} \dfrac{n \sin(n\pi y) \left(1-\psi\left(\xi^2+n^2-\lambda\right)\right)}{\xi^2+n^2-\lambda} = \sum_{n=1}^{\infty} \dfrac{n \sin(n\pi y) \left(1-\psi(n^2+\mu)\right)}{n^2 + \mu} \\
            & = \sum_{n=1}^\infty \sin(n\pi y) \cdot \dfrac{n}{n^2 + \mu} - \sin(\pi y) \dfrac{\psi(1+\mu)}{1 + \mu} =\frac\pi 2 \frac{\sinh \sqrt{\mu} (1 - y )\pi }{ \sinh \sqrt{\mu} \pi }- \sin(\pi y) \dfrac{\psi(1+\mu)}{1 + \mu}\, .
        \end{align*}
        Therefore, the absolute value satisfies
        \begin{align}
            & \left|\sum_{n=1}^{\infty} \dfrac{n \sin(n\pi y) \left(1-\psi\left(\xi^2+n^2-\lambda\right)\right)}{\xi^2+n^2-\lambda}\right| \lesssim e^{ - \pi y \sqrt{\xi ^2 -\lambda } }  + \dfrac{\psi(1+\xi ^2 -\lambda)}{1 + \xi ^2 -\lambda} \, . \label{3.33}
        \end{align}
        With (\ref{3.31}) and (\ref{3.33}), it is deduced that for each fixed $\sigma > 0$,
        \begin{align*}
            &\|K_2\|_{L^4_{xyt}}^4 = \int_0^1 \|K_2\|_{L^4_{xt}}^4 \, dy \\
            &\lesssim \int_0^1 \left(\int_{-\infty}^{\infty} \int_0^{\xi^2} \left|\sum_{n=1}^{\infty} \dfrac{n \sin(n\pi y) \cdot \left(1-\psi\left(\xi^2+n^2-\lambda\right)\right)}{\xi^2+n^2-\lambda} \widetilde{h}(\xi,\lambda) \right|^{\frac{4}{3}} \, d\lambda \, d\xi \right)^{\frac{3}{4} \cdot 4} \, dy \\
            &\lesssim \int_0^1 \left(\int_{-\infty}^{\infty} \int_0^{\xi^2} \left|
            \left ( e^{ - \pi y \sqrt{\xi ^2 -\lambda } }  + \dfrac{\psi(1+\xi ^2 -\lambda)}{1 + \xi ^2 -\lambda}\right )
            \widetilde{h}(\xi,\lambda) \right|^{\frac{4}{3}} \, d\lambda \, d\xi \right)^3 \, dy \\
            &\lesssim \int_0^1 \left(\int_{-\infty}^{\infty} \left ( \int_0^{|\xi|^{3/2}}
            + \int_{|\xi|^{3/2}}^{\xi^2} \right )  \left|
            \left ( e^{ - \pi y \sqrt{\xi ^2 -\lambda } }  + \dfrac{\psi(1+\xi ^2 -\lambda)}{1 + \xi ^2 -\lambda}\right )
            \widetilde{h}(\xi,\lambda) \right|^{\frac{4}{3}} \, d\lambda \, d\xi \right)^3 \, dy \\
            &\lesssim \left(\int_{-\infty}^{\infty} \int_0^{|\xi|^{3/2}} \left|\widetilde{h}(\xi,\lambda)\right|^2 \left(1+\xi^2\right)^{\sigma} \, d\lambda \, d\xi \right)^2 \cdot \int_{-\infty}^{\infty} \int_0^{|\xi|^{3/2}} \dfrac{1}{\left(1+\xi^2\right)^{2\sigma} \left(1+\xi^2-\lambda\right)^{1/2}} \, d\lambda \, d\xi \\
            &\qquad + \left(\int_{-\infty}^{\infty} \int^{\xi^2}_{|\xi|^{3/2}} \left|\widetilde{h}(\xi,\lambda)\right|^2 \left(1+\lambda^2\right)^{\sigma} \, d\lambda \, d\xi \right)^2 \cdot \int_{-\infty}^{\infty} \int^{\xi^2}_{|\xi|^{3/2}}  \dfrac{1}{\left(1+\lambda^2\right)^{2\sigma} \left(1+\xi^2-\lambda\right)^{1/2}} \, d\lambda \, d\xi \\
            &\lesssim \|h\|_{L^2_t([0,T]; \, H^{\sigma}_x)}^4 \cdot \int_{-\infty}^{\infty} \dfrac{1}{\left(1+\xi^2\right)^{2\sigma}} \int_0^{|\xi|^{3/2}} \dfrac{1}{\left(1+\xi^2-\lambda\right)^{1/2}} \, d\lambda \, d\xi \\
            & \qquad +  \|h\|_{H^\sigma_t([0,T]; \, L^2_x)}^4 \cdot \int_{-\infty}^{\infty} \dfrac{1}{\left(1+|\xi |^3\right)^{2\sigma}} \int_{|\xi|^{3/2}}^{\xi^2} \dfrac{1}{\left(1+\xi^2-\lambda\right)^{1/2}} \, d\lambda \, d\xi \\
            &\lesssim \|h\|_{L^2_t([0,T]; \, H^{\sigma}_x)}^4 \cdot \int_{-\infty}^{\infty} \dfrac{1+|\xi |^{1/2}}{\left(1+\xi^2\right)^{2\sigma}} \, d\xi + \|h\|_{H^\sigma_t([0,T]; \, L^{2}_x)}^4 \cdot \int_{-\infty}^{\infty} \dfrac{1+|\xi |}{\left(1+|\xi|^3\right)^{2\sigma}} \, d\xi \\
            &\lesssim \|h\|_{L^2_t([0,T]; \, H^{\sigma}_x)}^4 + \|h\|_{H^\sigma_t([0,T]; \, L^{2}_x)}^4
        \end{align*}
        if $\sigma > 3/8$. Hence,
        \begin{equation}\label{3.34}
            \|I^+_{2,3}\|_{L^4(\R \times [0,1] \times [0,T])} \lesssim \|h\|_{H^{\frac{1}{2}}_t([0,T]; \, L^2_x) \bigcap L^2_t([0,T]; \, H^{\frac12}_x)} \, \, .
        \end{equation}
        Adding (\ref{3.24}), (\ref{3.28}), (\ref{3.29}) and (\ref{3.34}), the estimate of $I^+_{2}$ in the $L^4 \bigcap L^{\infty}(L^2)$-norm  is obtained:
        \begin{equation}\label{3.35}
            \|I^+_2\|_{L^4_{xyt} \left(\R \times [0,1] \times [0,T]\right) \bigcap L^{\infty}_t \left([0,T]; \, L^2_{xy}(\R \times [0,1])\right)} \lesssim \|h\|_{H^{\frac{1}{2}}_t([0,T]; \, L^2_x) \bigcap L^2_t([0,T]; \, H^{\frac{1}{2}}_x)}\, \, .
        \end{equation}

        Now, we can study $I^+_3$ in (\ref{3.15}). By \emph{Proposition~\ref{prop_3.1}}
        \begin{align*}
            \|I^+_3\|&_{L^4(\R \times [0,1] \times [0,T]) \, \bigcap \, L^{\infty}_t([0,T]; \, L^2_{xy}(\R \times [0,1]))}^2 \\
            &\lesssim \int_{-\infty}^{\infty} \sum_n \left|n \int_0^{\infty} \dfrac{\widetilde{h}(\xi,\lambda) \cdot \left(1-\psi\left(\xi^2+n^2-\lambda\right)\right)}{\xi^2+n^2-\lambda} \, d\lambda \right|^2 \, d\xi \le L_1+L_2
        \end{align*}
        where
        \begin{align}
            L_1 &= \int_{-\infty}^{\infty} \sum_n \left|n \int_0^{\xi^2} \dfrac{\widetilde{h}(\xi,\lambda) \cdot \left(1-\psi\left(\xi^2+n^2-\lambda\right)\right)}{\xi^2+n^2-\lambda} \, d\lambda \right|^2 \, d\xi \, , \label{3.36} \\
            L_2 &= \int_{-\infty}^{\infty} \sum_n \left|n \int_{\xi^2}^{\infty} \dfrac{\widetilde{h}(\xi,\lambda) \cdot \left(1-\psi\left(\xi^2+n^2-\lambda\right)\right)}{\xi^2+n^2-\lambda} \, d\lambda \right|^2 \, d\xi \, . \label{3.37}
        \end{align}
        For (\ref{3.36}), we let $\mu=\xi^2-\lambda$ so that
        \begin{align*}
            L_1 &= \int_{-\infty}^{\infty} \sum_n \left|n \int_0^{\xi^2} \dfrac{\widetilde{h}(\xi,\lambda) \cdot \left(1-\psi\left(\xi^2+n^2-\lambda\right)\right)}{\xi^2+n^2-\lambda} \, d\lambda \right|^2 \, d\xi \\
            &\eqsim \int_{-\infty}^{\infty} \sum_n \left|n \int_0^{\xi^2} \dfrac{\widetilde{h}\left(\xi,\xi^2-\mu\right) \cdot \left(1-\psi\left(n^2+\mu\right)\right)}{n^2+\mu} \, d\mu \right|^2 \, d\xi \\
            &\le \int_{-\infty}^{\infty} \sum_n \left[\int_0^{\xi^2} \left|\widetilde{h}\left(\xi,\xi^2-\mu\right)\right|^2 \, d\mu \cdot \int_0^{\xi^2} \dfrac{n^2}{\left(n^2+\mu\right)^2} \, d\mu \right] \, d\xi \\
            &\le \int_{-\infty}^{\infty} \left[\int_0^{\xi^2} \left|\widetilde{h}\left(\xi,\xi^2-\mu\right)\right|^2 \, d\mu \cdot \int_0^{\xi^2} \sum_n \dfrac{n^2}{\left(n^2+\mu\right)^2} \, d\mu \right] \, d\xi \\
            &\le \int_{-\infty}^{\infty} \left[\int_0^{\xi^2} \left|\widetilde{h}\left(\xi,\xi^2-\mu\right)\right|^2 \, d\mu \cdot \int_0^{\xi^2} \sum_n \dfrac{1}{n^2+\mu} \, d\mu \right] \, d\xi \\
            &\lesssim \int_{-\infty}^{\infty} \left[\int_0^{\xi^2} \left|\widetilde{h}\left(\xi,\xi^2-\mu\right)\right|^2  \, d\mu \cdot \int_0^{\xi^2} \int_0^{\infty} \dfrac{1}{\eta^2+\mu} \, d\eta \, d\mu \right] \, d\xi \\
            &= \int_{-\infty}^{\infty} \left[\int_0^{\xi^2} \left|\widetilde{h}\left(\xi,\xi^2-\mu\right)\right|^2  \, d\mu \cdot \int_0^{\xi^2} \dfrac{1}{\sqrt{\mu}} \int_0^{\infty} \dfrac{1}{\eta^2+1} \, d\eta \, d\mu \right] \, d\xi \\
            &= \int_{-\infty}^{\infty} \int_0^{\xi^2} |\xi| \left|\widetilde{h}\left(\xi,\xi^2-\mu\right)\right|^2  \, d\mu \, d\xi \eqsim \int_{-\infty}^{\infty} \int_0^{\xi^2} |\xi| \left|\widetilde{h}\left(\xi,\lambda\right)\right|^2 \, d\lambda \, d\xi \le \|h\|_{L^2_t \left([0,T]; \, \dot{H}^{\frac{1}{2}}_x\right)}^2 \, .
        \end{align*}
        Let $\nu=\sqrt{\lambda-\xi^2}$ so that we can apply Lemma A-1 in \cite{Sun_NLSE_1d} to (\ref{3.37}) and obtain
        \begin{align*}
            L_2 &= 2 \int_{-\infty}^{\infty} \sum_n \left|\int_0^{\infty} \dfrac{n \nu \widetilde{h}\left(\xi,\xi^2+\nu^2\right) \cdot \left(1-\psi\left(n^2-\nu^2\right)\right)}{n^2-\nu^2} \, d\nu \right|^2 \, d\xi \\
            &\le 2 \int_{-\infty}^{\infty} \sum_n \left|\int_0^{\infty} \dfrac{\nu \widetilde{h}\left(\xi,\xi^2+\nu^2\right) \cdot \left(1-\psi\left(n^2-\nu^2\right)\right)}{n-\nu} \, d\nu \right|^2 \, d\xi \\
            &\lesssim \int_{-\infty}^{\infty} \int_0^{\infty} (1+\nu) \nu^2 \left|\widetilde{h}\left(\xi,\xi^2+\nu^2\right) \right|^2 \, d\nu \, d\xi \\
            &\lesssim \int_{-\infty}^{\infty} \int_{\xi^2}^{\infty} \left(1+\sqrt{\lambda}\right) \sqrt{\lambda} \left|\widetilde{h}(\xi,\lambda) \right|^2 \, d\lambda \, d\xi \lesssim \|h\|_{H^{\frac{1}{2}}_t \left([0,T]; \, L^2_x\right)}^2 \, .
        \end{align*}
        Hence,
        \begin{equation}\label{3.38}
            \|I^+_3\|_{L^4_{xyt} \left(\R \times [0,1] \times [0,T]\right) \bigcap L^{\infty}_t \left([0,T]; \, L^2_{xy}(\R \times [0,1])\right)} \lesssim \|h\|_{H^{\frac{1}{2}}_t \left([0,T]; \, L^2_x\right) \bigcap L^2_t \left([0,T]; \, H^{\frac{1}{2}}_x \right)} \, .
        \end{equation}

        The next part to work on is the estimate for $I^-$. First, we look at $I^-_1$ in (\ref{3.16}).
        \begin{align*}
            I^-_1 &= \int_{-\infty}^{\infty} \sum_n \left(e^{i\pi(x\xi+ny)} \cdot n \int_0^{\infty} \dfrac{e^{i\pi^2\lambda t}}{\xi^2+n^2+\lambda} \widetilde{h}(\xi,-\lambda) \, d\lambda\right) \, d\xi \\
            &\eqsim \int_{-\infty}^{\infty} \int_0^{\infty} e^{i\pi x\xi + i\pi^2\lambda t} \widetilde{h}(\xi,-\lambda) \sum_{n=1}^{\infty} \dfrac{n \sin(n\pi y)}{\xi^2+n^2+\lambda} \, d\lambda \, d\xi
        \end{align*}
        From (\ref{3.33}), if we replace $-\lambda$ by $\lambda$, then
        \begin{equation}\label{3.39} 
            \left|\sum_{n=1}^{\infty} \dfrac{n \sin(n\pi y)}{\xi^2+n^2+\lambda}\right| \lesssim e^{-\pi y\sqrt{\xi^2 + \lambda}} \lesssim \dfrac{|y|^{\sigma-1}}{\left(1 + \sqrt{\xi^2+\lambda}\right)^{1-\sigma}}
        \end{equation}
        for any $\lambda>0$, $y \in [0, 1]$ and $\sigma\in\left[1/2,1\right]$.
        Thus
        \begin{align*}
            \|I^-_1\|&_{L^{\infty}_t([0,T]; \, L^2_{xy}(\R \times [0,1]))}^2 = \sup_{t\in[0,T]} \left\|I^-_1(t)\right\|_{L^2_{xy}}^2 \\
            &\lesssim \sup_{t\in[0,T]} \int_0^1 \int_{-\infty}^{\infty} \left|\int_0^{\infty} e^{i\pi^2\lambda t} \widetilde{h}(\xi,-\lambda) \sum_{n=1}^{\infty} \dfrac{n \sin(n\pi y)}{\xi^2+n^2+\lambda} \, d\lambda \right|^2 \, d\xi \, dy \\
            &\le \int_0^1 \int_{-\infty}^{\infty} \left(\int_0^{\infty} \left|\widetilde{h}(\xi,-\lambda)\right| \cdot \left|\sum_{n=1}^{\infty} \dfrac{n \sin(n\pi y)}{\xi^2+n^2+\lambda}\right| \, d\lambda \right)^2 \, d\xi \, dy \\
            &\lesssim \int_0^1 \int_{-\infty}^{\infty} \left(\int_0^{\infty} \left|\widetilde{h}(\xi,-\lambda)\right| \cdot \left|\dfrac{|y|^{\sigma-1}}{\left(1 + \sqrt{\xi^2+\lambda}\right)^{1-\sigma}} \right| \, d\lambda \right)^2 \, d\xi \, dy \\
            &\eqsim \int_{-\infty}^{\infty} \left(\int_0^{\infty} \left|\widetilde{h}(\xi,-\lambda)\right| \cdot \left|\dfrac{1}{\left(1 + \sqrt{\xi^2+\lambda}\right)^{1-\sigma}} \right| \, d\lambda \right)^2 \, d\xi \\
            &\le \int_{-\infty}^{\infty} \left(\int_0^{\infty} (1+|\lambda|) \left|\widetilde{h}(\xi,-\lambda)\right|^2 \, d\lambda \right) \cdot \left(\int_0^{\infty} \dfrac{1}{(1+|\lambda|)\left(1+\xi^2+\lambda\right)^{1-\sigma}} \, d\lambda \right) \, d\xi \\
            &\le \int_{-\infty}^{\infty} \left(\int_0^{\infty} (1+|\lambda|) \left|\widetilde{h}(\xi,-\lambda)\right|^2 \, d\lambda \right) \cdot \left(\int_0^{\infty} \dfrac{1}{\left(1+\lambda\right)^{2-\sigma}} \, d\lambda \right) \, d\xi \\
            &\lesssim \int_{-\infty}^{\infty} \left(\int_0^{\infty} (1+|\lambda|) \left|\widetilde{h}(\xi,-\lambda)\right|^2 \, d\lambda \right) \, d\xi \le \|h\|^2_{H^{\frac{1}{2}}_t \left([0,T]; \, L^2_x\right)}\, .
        \end{align*}
        To estimate the $L^4$-norm, we use the same technique for $K_2$ with the assistance from (\ref{3.39}). We only choose $\sigma\in\left(\frac{3}{4},1\right)$.
        \begin{align*}
            \|I^-_1\|&_{L^4(\R \times [0,1] \times [0,T])}^4 = \int_0^1 \|I^-_1\|_{L^4_{xt}}^4 \, dy \\
            &\lesssim \int_0^1 \left(\int_{-\infty}^{\infty} \int_0^{\infty} \left|\widetilde{h}(\xi,-\lambda) \right|^{\frac{4}{3}} \cdot \left|\sum_{n=1}^{\infty} \dfrac{n \sin(n\pi y)}{\xi^2+n^2+\lambda} \right|^{\frac{4}{3}} \, d\lambda \, d\xi \right)^3 \, dy \\
            &\lesssim \int_0^1 \left(\int_{-\infty}^{\infty} \int_0^{\infty} \left|\widetilde{h}(\xi,-\lambda) \right|^{\frac{4}{3}} \cdot \left|\dfrac{|y|^{\sigma-1}}{\left(1+\sqrt{\xi^2+\lambda}\right)^{1-\sigma}} \right|^{\frac{4}{3}} \, d\lambda \, d\xi \right)^3 \, dy \\
            &\lesssim \left(\int_{-\infty}^{\infty} \int_0^{\infty} \left|\widetilde{h}(\xi,-\lambda) \right|^{\frac{4}{3}} \cdot \left|\dfrac{1}{\left(1+\sqrt{\xi^2+\lambda}\right)^{1-\sigma}} \right|^{\frac{4}{3}} \, d\lambda \, d\xi \right)^3 \\
            &\le \int_{-\infty}^{\infty} \left(\int_0^{\infty} (1+|\lambda|)^{\frac{1}{2}} \left|\widetilde{h}(\xi,-\lambda) \right|^2 \, d\lambda \right)^2 \cdot \left(\int_0^{\infty} \dfrac{1}{(1+|\lambda|) \left(1+\xi^2+\lambda\right)^{2-2\sigma}} \, d\lambda \right) \, d\xi \\
            &\le \int_{-\infty}^{\infty} \left(\int_0^{\infty} (1+|\lambda|)^{\frac{1}{2}} \left|\widetilde{h}(\xi,-\lambda) \right|^2 \, d\lambda \right)^2 \cdot \left(\int_0^{\infty} \dfrac{1}{\left(1+\lambda\right)^{3-2\sigma}} \, d\lambda \right) \, d\xi \lesssim \|h\|_{H^{\frac{1}{4}}_t\left([0,T]; \, L^2_x\right)}^4\, ,
        \end{align*}
        which gives
        \begin{equation}\label{3.40}
            \|I^-_1\|_{L^4_{xyt} \left(\R \times [0,1] \times [0,T]\right) \bigcap L^{\infty}_t \left([0,T]; \, L^2_{xy}(\R \times [0,1])\right)} \lesssim \|h\|_{H^{\frac{1}{2}}_t \left([0,T]; \, L^2_x\right)} \, .
        \end{equation}
        To study $I^-_2$ in (\ref{3.17}), by \emph{Proposition~\ref{prop_3.1}}, we have
        \begin{align*}
            \|I^-_2\|&_{L^4(\R \times [0,1] \times [0,T]) \, \bigcap \, L^{\infty}_t([0,T]; \, L^2_{xy}(\R \times [0,1]))}^2 \\
            &\le \left\|\int_{-\infty}^{\infty} \sum_n \left(e^{-i\pi^2\left(\xi^2+n^2\right)t+i\pi(x\xi+ny)} \int_0^{\infty} \dfrac{n \widetilde{h}(\xi,-\lambda)}{\xi^2+n^2+\lambda} \, d\lambda\right) \, d\xi\right\|^2_{L^4 \, \bigcap \, L^{\infty}_t \left([0,T]; \, L^2_{xy}\right)} \\
            &\lesssim \int_{-\infty}^{\infty} \sum_n \left|\int_0^{\infty} \dfrac{n \widetilde{h}(\xi,-\lambda)}{\xi^2+n^2+\lambda} \, d\lambda \right|^2 \, d\xi \\
            &\le \int_{-\infty}^{\infty} \sum_n \left(\int_0^{\infty} (1+|\lambda|)^{2\gamma} \left|\widetilde{h}(\xi,-\lambda)\right|^2 \, d\lambda \right) \cdot \left(\int_0^{\infty} \dfrac{n^2}{(1+|\lambda|)^{2\gamma} (\xi^2+n^2+\lambda)^2} \, d\lambda\right) \, d\xi \\
            &\le \int_{-\infty}^{\infty} \sum_n \left(\int_0^{\infty} (1+|\lambda|)^{2\gamma} \left|\widetilde{h}(\xi,-\lambda)\right|^2 \, d\lambda \right) \cdot \left(\int_0^{\infty} \dfrac{n^2}{(1+|\lambda|)^{2\gamma} (n^2+\lambda)^2} \, d\lambda\right) \, d\xi \\
            &\le \|h\|_{H^{\gamma}_t([0,T]; \, L^2_x)} \cdot \left(\sum_n \int_0^{\infty} \dfrac{n^2}{(1+|\lambda|)^{2\gamma} (n^2+\lambda)^2} \, d\lambda \right)\, .
        \end{align*}
        Let $\beta>0$ and $\gamma>\frac{\beta}{2}+\frac{1}{4}>\frac{1}{4}$ and consider the second factor,
        \begin{align*}
            \sum_n \int_0^{\infty} &\dfrac{n^2}{(1+|\lambda|)^{2\gamma} (n^2+\lambda)^2} \, d\lambda \eqsim \int_0^{\infty} \sum_{n=1}^{\infty} \dfrac{n^2}{(1+|\lambda|)^{2\gamma} (n^2+\lambda)^2} \, d\lambda \\
            &\le \int_0^{\infty} \sum_{n=1}^{\infty} \dfrac{n^2}{(n^2+\lambda)^{\frac{3}{2}+\beta}} \cdot \dfrac{1}{(1+|\lambda|)^{2\gamma} (n^2+\lambda)^{\frac{1}{2}-\beta}} \, d\lambda \\
            &\le \int_0^{\infty} \sum_{n=1}^{\infty} \dfrac{1}{n^{1+2\beta}} \cdot \dfrac{1}{(1+|\lambda|)^{2\gamma+\frac{1}{2}-\beta}} \, d\lambda \le \sum_{n=1}^{\infty} \dfrac{1}{n^{1+2\beta}} \cdot \int_0^{\infty} \dfrac{1}{(1+|\lambda|)^{2\gamma+\frac{1}{2}-\beta}} \, d\lambda \le C\, .
        \end{align*}
        Therefore,
        \begin{equation}\label{3.41}
            \|I^-_2\|_{L^4_{xyt} \left(\R \times [0,1] \times [0,T]\right) \bigcap L^{\infty}_t \left([0,T]; \, L^2_{xy}(\R \times [0,1])\right)} \lesssim \|h\|_{H^{\gamma}_t([0,T]; \, L^2_x)} \quad \forall \, \gamma > 1/4 \, \, .
        \end{equation}
        Finally, combine (\ref{3.18}), (\ref{3.35}), (\ref{3.38}), (\ref{3.40}) and (\ref{3.41}) to attain the expected estimate (\ref{3.10}),
        \begin{equation*}
            \left\|W_b [h]\right\|_{L^4_{xyt} \left(\R \times [0,1] \times [0,T]\right) \bigcap L^{\infty}_t \left([0,T]; \, L^2_{xy}(\R \times [0,1])\right)} \lesssim \|h\|_{H^{\frac{1}{2}}_t \left([0,T]; \, L^2_x(\R) \right) \bigcap L^2_t \left([0,T]; \, H^{\frac{1}{2}}_x(\R) \right)}\, .
        \end{equation*}

        Now, take the derivatives into consideration. Let $\alpha=(\alpha_1,\alpha_2)$ be a nonnegative two-dimensional multi-index with $|\alpha|=\alpha_1+\alpha_2=s$ and $s$ an integer. We let $u$ solve (\ref{2.1}) with $\varphi = f = 0 $ and claim that
        \begin{equation}\label{3.42} 
            \left\|\partial^{2s}_y u\right\|_X \lesssim \sum_{j=1,2} \|(1+ |\lambda | +|\xi|)^{\frac{1}{2}} (1+ |\lambda |+|\xi|^2)^s \widehat{h_j}\|_{L^2_{\xi\lambda}} \, ,
        \end{equation}
        where $X = L^r \bigcap L^{\infty}_t \left(L^2_{xy}\right)$ for $r\in[2,4]$.
        To begin the proof of this claim, we observe by (\ref{2.1}) with $\varphi = f = 0$ that $\partial_t u = i\left(\partial_x^2 + \partial_y^2\right) u$. Then, for $s=1$, $2$ in \eqref{3.42},
        \begin{equation*}
            \left\|\partial_y^2 u\right\|_X \le \left\|\left(\partial_x^2 + \partial_y^2\right) u\right\|_X + \left\|\partial_x^2 u\right\|_X = \left\|\partial_t u\right\|_{X} + \left\|\partial_x^2 u\right\|_X
        \end{equation*}
        and
        \begin{align*}
            \left\|\partial^4_y u\right\|_X &= \left\|\partial_y^2 \left(\partial_y^2 u\right)\right\|_X \le \left\|\partial_t \left(\partial_y^2 u\right)\right\|_{X} + \left\|\partial_x^2 \left(\partial_y^2 u\right)\right\|_X = \left\|\partial_y^2 \left(\partial_t u\right)\right\|_{X} + \left\|\partial_y^2 \left(\partial_x^2 u\right)\right\|_X \\
            &\le \left\|\left(\partial_t^2 u\right)\right\|_X + 2 \left\|\partial_t \partial_x^2 u\right\|_X + \left\|\left(\partial_x^4 u\right)\right\|_X\, .
        \end{align*}
        The pattern of the estimate as $s$ increases can be found for $m\in\N$ by induction as
        \begin{equation*}
            \left\|\partial_y^{2m} u\right\|_X \le \sum_{k=0}^m \binom{m}{k} \left\|\partial_t^k \partial_x^{2(m-k)} u\right\|_X\, .
        \end{equation*}
        For $k\ge1$, let $v=\partial_t^k\partial_x^{2(m-k)} u$. Then, $v$ solves (\ref{2.1}) with $\varphi = f = 0 $ and the boundary data as follows (here, note that the higher order compatibility conditions are needed, and see Remark~\ref{rmk_2.7}):
        \begin{equation*}
            \left\{
            \begin{array}{l}
                i v_t+v_{xx}+v_{yy} = 0\, , \\[.08in]
                v(x,y,0) = (i)^k\left[(\partial_x^2+\partial_y^2)^k u\right](x,y,0) = 0\, , \\[.08in]
                v(x,0,t) = \partial_t^k\partial_x^{2(m-k)} h_1(x,t)\, , \qquad v(x,1,t) = \partial_t^k\partial_x^{2(m-k)} h_2(x,t)\, .
            \end{array}
            \right.
        \end{equation*}
        Thus, by (\ref{3.10}),
        \begin{equation*}
            \left \|\partial_t^k\partial_x^{2(m-k)} u \right \|_X = \|v\|_X \lesssim \sum_{j=1,2} \left \|\partial_t^k\partial_x^{2(m-k)} h_j\right  \|_{H^{\frac{1}{2}}_t([0,T]; L^2_x) \bigcap L^2_t\left([0,T]; H^{\frac{1}{2}}\right)}\, \, \, .
        \end{equation*}
        On the other hand, with the formula (\ref{2.8}), it is easy to verify that
        \begin{equation*}
            \left \|\partial_x^{2m} v\right \|_X \lesssim \sum_{j=1,2} \left \|\partial_x^{2m} h_j\right \|_{H^{\frac{1}{2}}_t([0,T]; L^2_x) \bigcap L^2_t\left([0,T]; H^{\frac{1}{2}}\right)}\, .
        \end{equation*}
        Therefore,
        \begin{align*}
            \left \|\partial^{2m}_y u\right \|_X &\lesssim \sum_{j=1,2} \sum_{k=0}^m \binom{m}{k} \left \|\partial_t^k\partial_x^{2(m-k)} h_j\right \|_{H^{\frac{1}{2}}_t([0,T]; L^2_x) \bigcap L^2_t\left([0,T]; H^{\frac{1}{2}}\right)} \\
            &\eqsim \sum_{j=1,2} \left \|(1 + |\lambda|+|\xi|)^{\frac{1}{2}} (1+ |\lambda|+|\xi|^2)^m \widehat{h_j}\right \|_{L^2_{\xi\lambda}}\, .
        \end{align*}
        Thus, (\ref{3.42}) follows.
        Since $\alpha_1$, $\alpha_2\ge0$ with $\alpha_1+\alpha_2=s$,
        \begin{align*}
            \left\|\partial^{\alpha_1}_x \partial^{\alpha_2}_y u\right\|_X &\lesssim \sum_{j=1,2} \left\|(1+|\lambda|+|\xi|)^{\frac{1}{2}} (1+ |\lambda|+|\xi|^2)^{\frac{\alpha_2}{2}} |\xi|^{\alpha_1} \widehat{h_j}\right\|_{L^2_{\xi\lambda}} \\
            &\lesssim \sum_{j=1,2} \left\|(1+|\lambda|+|\xi|                         )^{\frac{1}{2}} (1+|\lambda|+|\xi|^2)^{\frac{s}{2}} \widehat{h_j}\right\|_{L^2_{\xi\lambda}} = \sum_{j=1,2} \left\|{h_j}\right\|_{\HH^s(0, T)} \, .
        \end{align*}
        Hence, (\ref{3.11}) for $r=4$ is derived when $s$ is an integer. If $s\geq 0$ is not an integer, (\ref{3.11}) for $r=4$  follows by interpolation.

        Finally, we attempt to reach a more general conclusion in terms of function spaces. In fact, the special case of the estimate for $r=4$ leads us to the point that if $T< \infty$, then
        \begin{equation*}\label{NLSE_IB_2d_stp_loc_lin_wb_l2w2_est} 
            \left\|W_b[h_1,h_2]\right\|_{L^2_t \left([0,T]; \, H^{s,2}_{xy}(\R \times [0,1]) \right)} \lesssim T^{\frac{1}{2}} \sum_{j=1,2} \|h_j\|_{\HH^s(0, T)}\, .
        \end{equation*}
        By interpolation from $L^2$ to $L^4$, we can obtain (\ref{3.11}) and therefore the whole proof is complete.
    \end{proof}

    Also, it can be shown that the condition $h_j \in H^{\frac{1}{2}}_t \left([0,T]; \, L^2_x(\R) \right) \bigcap L^2_t \left([0,T]; \, H^\frac12_x(\R) \right)$ for $s=0$ is sharp with respect to the regularity on $t$.
    \begin{prop}
        \begin{equation*}
            \left\|W_{b_1}[h]\right\|_{L^2_{xyt}} \lesssim \|h\|_{H^{\sigma}_t\left([0,T]; \, L^2_x\right) \bigcap L^2_t\left([0,T]; \, H^1_x\right)} \, \text{ only when } \, \sigma \ge 1/2 \, .
        \end{equation*}
    \end{prop}

    \begin{proof}
        We assume that $\left\|W_{b_1}[h]\right\|_{L^2_{xyt}} \lesssim \|h\|_{H^{\sigma}_t\left([0,T]; \, L^2_x\right) \bigcap L^2_t\left([0,T]; \, H^1_x\right)}$ where  $$h \in H^{\sigma}_t\left([0,T]; \, L^2_x\right) \bigcap L^2_t\left([0,T]; \, H^1_x\right)$$ with $0\le\sigma<1/2$. From (\ref{2.8}), we have the formula for $W_{b_1}[h]$,
        \begin{equation*}
            W_{b_1}[h](x,y,t) = \int_{-\infty}^{\infty} \sum_n \left[n \int_0^t e^{-i\pi^2\left(\xi^2+n^2\right)(t-\tau) + i\pi(\xi x + n y)} \widehat{h^x}(\xi, \tau) \, d\tau\right] \, d\xi\, .
        \end{equation*}
        Let $g=g(x,t)$ be such that $$\displaystyle \widehat{g^x}(\xi,t) = \widehat{h^x}(\xi,t) e^{i\pi^2\xi^2t} = \sum_{k\in \, \Z} e^{-i\pi^2kt} a_k(\xi) \, ,$$ assuming that $a_{n^2}(\xi) \equiv 0$ $\forall \, n\in\Z$ and $\supp \widehat{h^x} \subset [0,1]\times[0,\frac{2}{\pi}]$. Define
        \begin{equation*}
            a_k(\xi) = \left\{
            \begin{array}{ll}
                \widehat{g_1}(\xi)/|n|^{\beta} &, \, k=n^2+1\, , \\
                0 &, \, \text{otherwise}\, ,
            \end{array}
            \right.
        \end{equation*}
        so that $\displaystyle \widehat{h^x}(\xi,t) e^{i\pi^2\xi^2t} = \sum_{n\neq0} \widehat{g_1}(\xi)/|n|^{\beta} e^{-i\pi^2(n^2+1)t}$.

        Here, it is possible to choose $g_1 \in H^1(\R)$. Define the Fourier series of $f$ as
        \begin{equation*}
            f(t) = \sum_{k\in \, \Z} e^{-i \pi^2 k t} \alpha_k \, \text{, where } \, \alpha_k = \int_0^{\frac{2}{\pi}} e^{i \pi^2 k t} f(t) \, dt\, .
        \end{equation*}
        Hence,
        $$
        W_{b_1}[h](x,y,t) = \int_{-\infty}^{\infty} e^{-i\pi^2 \xi^2 t+i\pi x\xi} \left(\sum_n n \cdot e^{-i\pi^2 n^2 t+i\pi y n} \int_0^t e^{i\pi^2n^2\tau}\widehat{g^x}(\xi,\tau) \, d\tau \right) \, d\xi \, .
        $$
        If $$\mathcal{J} := \sum_n n \cdot e^{-i\pi^2 n^2 t+i\pi y n} \int_0^t e^{i\pi^2n^2\tau}\widehat{g^x}(\xi,\tau) \, d\tau \, ,$$
        it is deduced that
        \begin{align*}
            \mathcal{J} &= \sum_n n \cdot e^{-i\pi^2 n^2 t+i\pi y n} \int_0^t e^{i\pi^2n^2\tau} \sum_{k \neq n^2} e^{-i\pi^2kt} a_k(\xi) \, d\tau \\
            &= \sum_n n \cdot e^{-i\pi^2 n^2 t+i\pi y n} \left(\sum_{k \neq n^2} \dfrac{e^{i\pi^2 \left(n^2-k\right) t} - 1}{n^2-k} a_k(\xi) \right) \\
            &= \sum_n n \cdot e^{i\pi y n} \left[\sum_{k \neq n^2} e^{-i\pi^2 k t} \dfrac{a_k(\xi)}{n^2-k} + e^{-i\pi^2 n^2 t} \left(\sum_{k \neq n^2} \dfrac{a_k(\xi)}{k-n^2}\right) \right]\, .
        \end{align*}
        Let
        \begin{equation*}
            b_{nk}(\xi) = \left\{
            \begin{array}{ll}
                \dfrac{a_k(\xi)}{n^2-k} &, \, k \neq n^2 \,,\\
                \displaystyle \sum_{k \neq n^2} \dfrac{a_k(\xi)}{k-n^2} &, \, k=n^2\, .
            \end{array}
            \right.
        \end{equation*}
        Then $$\displaystyle \mathcal{J} = \sum_{n,k} n e^{i\pi y n} e^{-i\pi^2 k t} b_{nk}(\xi) \, $$ and $$W_{b_1}h(x,y,t) = \int_{-\infty}^{\infty} \sum_{n,k} e^{-i\pi^2 \xi^2 t+i\pi x\xi} n e^{i\pi y n} e^{-i\pi^2 k t} b_{nk}(\xi) \, d\xi \, .$$ Therefore,
        \begin{align*}
            \left\|W_{b_1}h\right\|&_{L^2_{xyt}}^2= \int_{-\infty}^{\infty} \sum_{n,k} n^2 |b_{nk}(\xi)|^2 \, d\xi \ge \int_{-\infty}^{\infty} \sum_{n} n^2 |b_{n,n^2+1}(\xi)|^2 \, d\xi = \int_{-\infty}^{\infty} \sum_{n} n^2 |a_{n^2+1}(\xi)|^2 \, d\xi \\
            &\gtrsim \sum_n \dfrac{n^2}{|n|^{2\beta}} \gtrsim \sum_n \dfrac{1}{|n|^{2\beta-2}}\, .
        \end{align*}
        For $0\le\sigma<1/2$, if $h\in H^{\sigma}_t\left([0,\frac{2}{\pi}]; L^2_x\right)$, the product rule for fractional derivatives (\emph{Proposition 3.3} in \cite{Weinstein_KdV_DP_SA}) implies
        \begin{align*}
            & \left\|D^{\sigma}_t\left(\widehat{h^x}(\xi,t) \right)\right\|_{L^2_{\xi t}}= \left\|D^{\sigma}_t\left(\widehat{g^x}(\xi,t) e^{- i\pi^2\xi^2t}\right)\right\|_{L^2_{\xi t}} \\
            &\le \left\| ( 1 + |\xi|^2) ^\sigma D^{\sigma}_t \widehat{g^x}(\xi,t) \right\|_{L^2_{\xi t}} \left\| ( 1 + |\xi|^2) ^{-\sigma} e^{i\pi^2\xi^2t}\right\|_{L^{\infty}_{\xi t}} + \left\| ( 1 + |\xi|^2) ^\sigma \widehat{g^x}(\xi,t) \right\|_{L^2_{\xi t}} \left\| e^{i\pi^2\xi^2t}\right\|_{L^{\infty}_{\xi t}} < \infty\, ,
        \end{align*}
        which requires that $$\left\|( 1 + |\xi|^2) ^\sigma \sum_{n\neq0} (n^2+1)^{\sigma} \frac{\widehat{g_1}(\xi)}{|n|^{\beta}} e^{-i\pi^2(n^2+1)t} \right\|_{L^2_{\xi t}} < \infty \, .$$
        Note that $g_1 \in H^1_x(\R)$ and $h \in L^2_t\left(\left(0,\frac{2}{\pi}\right); \, H^1_x(\R)\right)$. Owing to $\sigma<1/2$, it is possible to pick $\beta$ in $\left(2\sigma+\frac{1}{2}, \, \frac{3}{2}\right]$ such that $2\beta-2 <1$ and
        \begin{equation}\label{3.43} 
            \|W_{b_1}[h] \|^2_{L^2_{xyt}} \gtrsim \sum_n \dfrac{1}{|n|^{2\beta-2}} = \infty\, .
        \end{equation}
        By the assumption at the beginning of the proof, we have
        \begin{equation*}
            \|W_{b_1}[h]\|_{L^2_{xyt}} \lesssim \|h\|_{H^{\sigma}_t\left(\left(0,\frac{2}{\pi}\right); \, L^2_x\right)} + \|h\|_{L^2_t\left(\left(0,\frac{2}{\pi}\right); \, H^1_x\right)} < \infty\, ,
        \end{equation*}
        which gives a contradiction to (\ref{3.43}). Hence, $\sigma \ge 1/2$ is required. Or, the regularity requirement on $t$ for the boundary data in \emph{Proposition~\ref{prop_3.4}} is sharp.
    \end{proof}

\section{Local Well-posedness}\label{sec_4}

    Now, we are ready for the study of local well-posedness of (\ref{2.1}) with estimates established in the previous section. Recall the mild solution of (\ref{2.1}) given by (\ref{2.5}) and define an operator $\mathcal{A}[u](x,y,t)$:
    \begin{equation} \label{4.1} 
        u(x,y,t) = W_b[h_1,h_2] (x,y,t) + W_0(t)\phi (x,y) + \, i \left(\int_0^t W_0(t-\tau) f(\tau) \, d\tau \right) (x,y) = \mathcal{A}[u](x,y,t) \, ,
    \end{equation}
    where $f(u) = \lambda |u|^{p-2} u$ for $p \ge 3$ and $(x,y,t) \in \R \times [0,1] \times (0,T)$. We will show that there is a unique solution in $H^s(\R\times[0,1])$ with a maximal existence interval  and the solution continuously depends upon the initial and boundary data by proving that the operator $\mathcal{A}[u](x,y,t)$ is a contraction. Let $r\in[2,4]$ and define the following function spaces,
    \begin{equation*}
        \X^s_T := C_t \left([0,T]; \, H^s_{xy}(\R \times [0,1])\right) \bigcap L^r_t \left([0,T]; \, W^{s,r}_{xy}(\R \times [0,1])\right)
    \end{equation*}
    with $\displaystyle \|u\|_{\X^s_T} = \sup_{t\in[0,T]} \|u(t)\|_{H^s_{xy}(\R \times [0,1])} + \|u\|_{L^r_t\left([0,T]; \, W^{s,r}_{xy}(\R \times [0,1])\right)}$.
    Let
    \begin{equation*}
        \Y^s_T := C_t \left([0,T]; \, H^s_{xy}(\R \times [0,1])\right)
    \end{equation*}
    with $\displaystyle \|u\|_{\Y^s_T} = \sup_{t\in[0,T]} \|u(t)\|_{H^s_{xy}(\R \times [0,1])}$. For some $M>0$, define closed balls in $\X_T$ and $\Y_T$ of radius $M$ as
    \begin{equation*}
        B_M^{\X^s} := \left\{u: \, \|u\|_{\X^s_T} \le M\right\} \quad \text{and } \quad B_M^{\Y^s} := \left\{u: \, \|u\|_{\Y^s_T} \le M\right\}\, .
    \end{equation*}

    For the investigation on the existence of solution of (\ref{4.1}), we need the following two lemmas on differentiation of fractional order:
    \begin{lemma}[The chain rule for fractional derivatives] (see \emph{Proposition 3.1} in \cite{Weinstein_KdV_DP_SA}) 
        Let $0<s<1$ and $\alpha$ be a nonnegative multi-index with $|\alpha|=s$. Let $u: \R^2 \to \C$ and $f\in C^1(\C)$. Then
        \begin{equation}\label{4.2}
            \|D^{\alpha} f(u)\|_{L^{r_2}} \lesssim \|f'(u)\|_{L^{r_1}} \|D^{\alpha} u\|_{L^r}
        \end{equation}
        for $\frac{1}{r_2} = \frac{1}{r_1} + \frac{1}{r}$ with $1<r$, $r_1$, $r_2<\infty$. In particular, if $|f'(u)|$ is uniformly bounded, then
        \begin{equation}\label{4.3}
            \|D^{\alpha} f(u)\|_{L^r} \lesssim \|f'(u)\|_{L^{\infty}} \|D^{\alpha} u\|_{L^r}\, .
        \end{equation}
    \end{lemma}

    \begin{lemma}[The product rule (Leibnitz rule) for fractional derivatives] (see  \emph{Proposition 3.3} in \cite{Weinstein_KdV_DP_SA} or \emph{Lemma 2.6} in \cite{Kenig_KdV_Gen} for high-dimensional cases) 
        Let $0<s<1$, $u$, $v: \R^2 \to \C$ and $f\in C^1(\C)$. Then
        \begin{equation}\label{4.4}
            \|D^{\alpha} (uv)\|_{L^{r_2}} \lesssim \|v\|_{L^{r_1}} \|D^{\alpha} u\|_{L^r} + \|u\|_{L^{\tilde{r}_1}} \|D^{\alpha} v\|_{L^{\tilde{r}}}
        \end{equation}
        for $\frac{1}{r_2} = \frac{1}{r_1} + \frac{1}{r} = \frac{1}{\tilde{r}_1} + \frac{1}{\tilde{r}}$ with $1<r$, $r_1$, $\tilde{r}$, $\tilde{r}_1$, $r_2<\infty$.
    \end{lemma}
    \noindent The proof follows the idea in the appendix of \cite{Kato_E_NS}.

\begin{remark}\label{remark4.3}
Lemmas 4.1 and 4.2 are valid for functions in $\R^2$. To use these lemmas for functions in $\R\times [0, 1]$, we use an extension operator
$E$ such that $Eu \in W^{s, p } (\R^2 ) $ and $Eu \big | _{\R \times [0, 1]} = u $ for any $ u \in W^{s, p } (\R\times [0,1] )$ satisfying
$$
\| Eu \|_{ W^{s, p } (\R^2 )} \leq C \| u \|_{ W^{s, p } (\R\times [0,1] )}\, .
$$
The existence of such extension operator has been discussed in Chapter 5 of Adams' book \cite{Adams}. Therefore, by using the extension operator, Lemmas 4.1 and 4.2 are valid for functions in $\R\times [0, 1]$ except that the derivatives on the right hand sides of inequalities are replaced by $W^{|\alpha|, p }$-norms of $u$ and $v$.
\end{remark}

    First, we aim to prove the following theorem.
    \begin{thm}\label{thm_4.3} 
        Choose $r\in[2,4]$. Let $\mu>0$ such that $$\displaystyle \|\varphi\|_{H^s(\R \times [0,1])} + \sum_{j=1,2} \|h_j\|_{\HH^s(0, T)} \le \mu$$ and $\varphi$, $h$ satisfy certain compatibility conditions. Then
        \begin{itemize}
        \item[(a)]
            For $0\le s < 1/2$ and $3\le p \le 4$, or $1/2 \le s < 1$ and $3\le p \le \frac{3-2s}{1-s}$, or $s=1$ and $3\le p < \infty$, there is a $ T>0$ such that for $r\in[2,4]$ a unique solution $u \in \X^s_T$ of (\ref{4.1}) exists. Moreover,
            \begin{align}
                & \left\|\mathcal{A}[u]-\mathcal{A}[v]\right\|_{\X^s_T} \le (1/2) \, \|u-v\|_{\X^s_T}\, , \label{4.5} \\
                & \left\|\mathcal{A}[u]\right\|_{\X^s_T} \le M \label{4.6}
            \end{align}
            for any $u$ and $v \in B_M^{\X^s}$.
        \item[(b)]
            For $s >1 $ and $(p, s)$ satisfying \eqref{1.4.1}, there exists a $T>0$ such that  (\ref{4.1}) has a unique solution $u \in \Y^s_T $. Also, for $u$ and $v \in B_M^{\Y^s}$
            \begin{align}
                & \left\|\mathcal{A}[u]-\mathcal{A}[v]\right\|_{\Y^0_T} \le (1/2) \, \|u-v\|_{\Y^0_T} \, , \label{4.7} \\
                & \left\|\mathcal{A}[u]\right\|_{\Y^s_T} \le M \, . \label{4.8}
            \end{align}
            In addition, if we further assume that
            \begin{equation}
            \mbox{$p$ is  even or for $p$ not even, either $p\ge s+2$ with $s\in\Z$ or $p\ge[s]+3$ with $s\notin\Z$  ,}\label{4.8.1}
             \end{equation}
             then (\ref{4.7}) can be improved by
            \begin{equation}
                \left\|\mathcal{A}[u]-\mathcal{A}[v]\right\|_{\Y^s_T} \le (1/2) \, \|u-v\|_{\Y^s_T} \, , \label{4.9}
            \end{equation}
            for $u$ and $v \in B_M^{\Y^s}$.
        \end{itemize}
    \end{thm}

    \begin{proof}
        First, consider the case for $0 \le s < \frac{1}{2}$. Assume
        \begin{equation*}
            3 \le p \le r \le 4 \label{NLSE_IB_2d_stp_loc_WP_exiuni_llreg_r} \, ,
        \end{equation*}
        which implies that
        \begin{equation*}\label{NLSE_IB_2d_stp_loc_WP_exiuni_llreg_sobemb_xy}
            1 + \dfrac{s(p-2)}{2} \ge \dfrac{p}{r} \quad \text{or equivalently} \quad \dfrac{r(p-2)}{r-2} \le \dfrac{2r}{2-rs} \, .
        \end{equation*}
        In addition,
        \begin{equation*}\label{NLSE_IB_2d_stp_loc_WP_exiuni_llreg_sobemb_t}
            r \ge r'(p-1)= \frac{r(p-1)}{r-1} \, .
        \end{equation*}
        Then, by Sobolev embedding theorem, $W^{s,r} \hookrightarrow L^{\frac{r(p-2)}{r-2}}$ and $\|u(t)\|_{L^{\frac{r(p-2)}{r-2}}} \le \| u(t)\|_{W^{s, r}}$.

        If $q=r'=\frac{r}{r-1}$ is chosen, then $q \in \left(\frac{4}{3},2\right]$. Let $u \in B_M^{\X}$ and $\alpha$ be a multi-index such that $|\alpha|=s$ as usual. We know that $u: \R\times [0,1]\times \R \to \C$ and $f\in C^1(\C)$ and $|f'(u)| \lesssim |u|^{p-2} < \infty$ for $p\ge3$. For $0 \le s \le 1/2$ and $r\in(2,4]$, the chain rule (\ref{4.2}) and Remark \ref{remark4.3} suggest that
        \begin{align*}
            \left\|D^{\alpha} f(u)(t)\right\|&_{L^{r'}} \lesssim \left\|f'(u)(t)\right\|_{L^{\frac{r}{r-2}}} \| u(t)\|_{W^{s, r}} \\
            &\le \|u(t)\|_{L^{\frac{r(p-2)}{r-2}}}^{p-2} \| u(t)\|_{W^{s, r}} \le \| u(t)\|_{W^{s, r}}^{p-1} \, .
        \end{align*}
        Next, take the $L^{r'}$-norm of both sides of the inequality above with respect to time $t$ and  obtain
        \begin{align*} 
            \left\|D^{\alpha} f(u)\right\|&_{L^{r'}_t\left([0,T]; L^{r'}_{xy}(\R \times [0,1])\right)} \lesssim \| u\|_{L^{\frac{r(p-1)}{r-1}}_t\left([0,T]; W^{s, r}_{xy} ( \R\times [0,1]) \right)}^{p-1} \nonumber \\
            &\le T^{1-\frac{p}{r}} \| u\|_{L^r_t\left([0,T]; W^{s, r}_{xy}(\R \times [0,1])\right)}^{p-1}\quad .
        \end{align*}
        Equivalently,
        \begin{equation}\label{4.10}
            \left\|f(u)\right\|_{L^{r'}_t\left([0,T]; W^{s,r'}_{xy}(\R \times [0,1])\right)} \lesssim T^{1-\frac{p}{r}} \|u\|_{L^r_t\left([0,T]; W^{s,r}_{xy}(\R \times [0,1])\right)}^{p-1} \le T^{1-\frac{p}{r}} \|u\|_{\X^s_T}^{p-1} \, .
        \end{equation}
        (The investigation for $s=0$ can be found in \cite{Takaoka_NLSE_2d_Periodic}.)
        Assume $u, v \in \X^s_T$ and $w = u \cdot \theta + v \cdot (1-\theta)$ for $\theta\in[0,1]$. Then, we have
        \begin{equation*}
            |u|^{p-2}u - |v|^{p-2}v = \int_0^1 \left[\dfrac{p}{2} |w|^{p-2} + \left(\dfrac{p}{2}-1\right) |w|^{p-4} w^2\right] \cdot (u-v) \, d\theta \, .
        \end{equation*}
        Based upon the proof of Theorem 4.3 in \cite{Ran_NLSE_2d_uhp}, if $p>3$, (\ref{4.4}) yields
        \begin{align*}
            \Big \|D^{\alpha} \big [& |u|^{p-2}u(t) - |v|^{p-2}v(t)\big ] \Big \|_{L^{r'}} \\
            =& \left\|\int_0^1 D^{\alpha}\left[\frac{p}{2} |w|^{p-2} + \left(\frac{p}{2}-1\right) |w|^{p-4} w^2\right] \cdot (u-v) \, d\theta \right\|_{L^{r'}} \\
            \le& \sup_{\theta\in[0,1]} \left\|D^{\alpha}\left[\left(\frac{p}{2} |w|^{p-2} + \left(\frac{p}{2}-1\right) |w|^{p-4} w^2\right) \cdot (u-v)\right] \right\|_{L^{r'}} \\
            \lesssim& \sup_{\theta\in[0,1]} \left\||w(t)|^{p-3} \right\|_{L^{\frac{r(p-2)}{(r-2)(p-3)}}} \cdot \| w(t)\|_{W^{s, r}} \cdot \|u(t)-v(t)\|_{L^{\frac{r(p-2)}{r-2}}} \\
            & + \sup_{\theta\in[0,1]} \|w(t)\|_{L^{\frac{r(p-2)}{r-2}}}^{p-2} \cdot \|u(t)-v(t)\|_{W^{s, r}} \\
            =& \sup_{\theta\in[0,1]} \| w(t)\|_{W^{s, r}}^{p-2} \cdot \|u(t)-v(t)\|_{L^{\frac{p}{1-s}}} + \sup_{\theta\in[0,1]} \| w(t)\|_{W^{s, r}} \cdot \|u(t)-v(t)\|_{W^{s, r}} \\
            \lesssim& \left(\| u(t)\|_{W^{s,r}}^{p-2} + \| v(t)\|_{W^{s,r}}^{p-2}\right) \cdot \|(u(t)-v(t))\|_{W^{s,r}}\, .
        \end{align*}
        For $p=3$, (\ref{4.3}) shows that
        \begin{align*}
            & \left\|D^{\alpha} \left[|u|^{p-2}u(t) - |v|^{p-2}v(t)\right] \right\|_{L^{r'}} \\
            &\lesssim \sup_{\theta\in[0,1]} \left\|w(t) \right\|_{W^{s, r}} \cdot \|u(t)-v(t)\|_{L^{\frac{r(p-2)}{r-2}}} + \sup_{\theta\in[0,1]} \|w(t)\|_{L^{\frac{r(p-2)}{r-2}}} \cdot \|(u(t)-v(t))\|_{W^{s, r}} \\
            &\lesssim \left(\| u(t)\|_{W^{s,r}}+ \| v(t)\|_{W^{s,r}}\right) \cdot \|(u(t)-v(t))\|_{W^{s,r}}\, .
        \end{align*}
        Thus, Sobolev embedding theorem implies
        \begin{equation*}
            \left\|f(u)(t)-f(v)(t)\right\|_{W^{s,r'}} \lesssim \left(\|u(t)\|_{W^{s,r}_{xy}}^{p-2} + \|v(t)\|_{W^{s,r}_{xy}}^{p-2} \right) \cdot \left\|u(t)-v(t) \right\|_{W^{s,r}} \, .
        \end{equation*}
        In addition, $\frac{p}{r}\le1$ guarantees $r \ge \frac{r(p-2)}{r-2}$ and the following estimate for the difference of nonlinearity holds:
        \begin{align}\label{4.11}
            & \left\|f(u)-f(v)\right\|_{L^{r'}_t\left([0,T]; W^{s,r'}_{xy}(\R \times [0,1])\right)} \nonumber \\
            &\le \left(\|u\|_{L^{\frac{r(p-2)}{r-2}}_t\left([0,T]; W^{s,r}_{xy}(\R \times [0,1])\right)}^{p-2} + \|v\|_{L^{\frac{r(p-2)}{r-2}}_t\left([0,T]; W^{s,r}_{xy}(\R \times [0,1])\right)}^{p-2} \right) \|u-v\|_{L^r_t\left([0,T]; L^r_{xy}(\R \times [0,1])\right)} \nonumber \\
            &\le T^{1-\frac{p}{r}} \left(\|u\|_{L^r_t\left([0,T]; W^{s,r}_{xy}(\R \times [0,1])\right)}^{p-2} + \|v\|_{L^r_t\left([0,T]; W^{s,r}_{xy}(\R \times [0,1])\right)}^{p-2} \right) \|u-v\|_{L^r_t\left([0,T]; W^{s,r}_{xy}(\R \times [0,1])\right)} \, .
        \end{align}
        Note that $\mathcal{A}[u]$ is defined in (\ref{4.1}) and assume that $u, v \in B_M^{\X^s}$ for $0\le s < 1/2$ and $3\le p \le 4$ with some $r\in[p,4)$. Also
        \begin{equation*}
            \left\|\mathcal{A}(u)\right\|_{\X_T} = \left\|\mathcal{A}(u)\right\|_{L^{\infty}_t\left([0,T]; H^s_{xy}(\R \times [0,1])\right)} + \left\|\mathcal{A}(u)\right\|_{L^r_t\left([0,T]; W^{s,r}_{xy}(\R \times [0,1])\right)} \, .
        \end{equation*}
        If $3\le p <4$, we can choose $r\in(p,4)$ and let $\sigma \in (\frac{1}{2},\frac{r}{3r-4}]$, which implies $r' \in [\frac{4\sigma}{1+\sigma},2]$. Now, for $\frac{1}{2}<\sigma\le \frac{5r-4}{4(3r-4)}$, it is found that
        \begin{align}
            1+\sigma -\frac{4\sigma}{r'} &= 1-\left(3-\dfrac{4}{r}\right)\sigma \nonumber \\
            &\ge 1-\left(3-\dfrac{4}{r}\right) \cdot \dfrac{5r-4}{4(3r-4)} = \dfrac{1}{r} - \dfrac{1}{4} := \theta_r > 0\, . \label{4.12}
        \end{align}
        According to \emph{Proposition~\ref{prop_3.1}}, and (\ref{3.7}), (\ref{3.11}) in  \emph{Proposition~\ref{prop_3.3}} and \emph{\ref{prop_3.4}}, and (\ref{4.10}), it is deduced that
        \begin{align*}
            \left\|\mathcal{A}(u)\right\|&_{\X^s_T} \le \|W_b[h_1,h_2]\|_{\X^s_T} + \left\|W_0\phi \right\|_{\X_T} + \left\|\Phi_{0,f}\right\|_{\X^s_T} \\
            &\lesssim \sum_{j=1,2} \|h_j\|_{\HH^s(0, T)} + \left\|\phi \right\|_{H^s_{xy}} + T^{\theta_r} \left\|f \right\|_{L^{r'}_t \left([0,T]; H^{s,r'}_{xy} \right)} \\
            &\le \left(1+T^{\frac{1}{2}}+T^{\frac{1}{2}-\sigma}\right)\mu + T^{\theta_r} \left\|f \right\|_{L^{r'}_t \left([0,T]; H^{s,r'}_{xy} \right)} \\
            &\lesssim \left(1+T^{\frac{1}{2}}+T^{\frac{1}{2}-\sigma}\right)\mu + T^{\theta_r+1-\frac{p}{r}} \|u\|_{L^r_t\left([0,T]; W^{s,r}_{xy}(\R \times [0,1])\right)}^{p-1} \\
            &\le \left(1+T^{\frac{1}{2}}+T^{\frac{1}{2}-\sigma}\right)\mu + T^{\theta_r+1-\frac{p}{r}} \|u\|_{\X^s_T}^{p-1} \, ,
        \end{align*}
        where $\theta_r+1-\frac{p}{r} > \theta_r$ is fixed. Considering the Lipschitz norm, (\ref{4.11}) gives \begin{align*}
            \left\|\mathcal{A}(u)-\mathcal{A}(v)\right\|&_{\X^s_T} = \left\|\Phi_{0,f}(u)-\Phi_{0,f}(v) \right\|_{\X^s_T}
            \lesssim T^{\theta_r} \left\|f(u)-f(v) \right\|_{L^{r'}_t \left([0,T]; L^{r'}_{xy} \right)} \\
            \lesssim &T^{\theta_r+1-\frac{p}{r}} \left(\|u\|_{\X^s_T}^{p-2} + \|v\|_{\X^s_T}^{p-2} \right) \|u-v\|_{\X^s_T}
            \lesssim T^{\theta_r+1-\frac{p}{r}} M^{p-2} \|u-v\|_{\X^s_T}\, .
        \end{align*}
        Thus, with some constants $C_0$, $C_1$, we obtain
        \begin{equation}
            \left\|\mathcal{A}[u]\right\|_{\X^s_T} \le C_0 \left(\mu+ T^{\frac{1}{2}}\mu + T^{\frac{1}{2}-\sigma}\mu + T^{\theta_r+1-\frac{p}{r}} M^{p-1} \right)\, , \label{4.13}
        \end{equation}
        and
        \begin{equation}
            \left\|\mathcal{A}[u]-\mathcal{A}[v]\right\|_{\X^s_T} \le C_1 T^{\theta_r+1-\frac{p}{r}} M^{p-2} \|u-v\|_{\X^s_T} \, . \label{4.14}
        \end{equation}
        Choose $T$ small so that $T \le 1$ and $T^{\theta_r+1-\frac{p}{r}} M^{p-1} \le T^{\frac{1}{2}-\sigma}\mu$, which implies that we need $$T = \min\left\{1, \, \left(\frac{\mu}{M^{p-1}}\right)^{1/(\theta_r+1-\frac{p}{r}-(\frac{1}{2}-\sigma))}, \, \left(\frac{1}{2C_1M^{p-2}}\right)^{1/(\theta_r+1-\frac{p}{r})} \right\}.$$  Observe that the right hand side of (\ref{4.13}) yields
        \begin{align}
            C_0 & \left(\mu+ T^{\frac{1}{2}}\mu + T^{\frac{1}{2}-\sigma}\mu + T^{\theta_r+1-\frac{p}{r}} M^{p-1} \right) \le 4C_0 \mu T^{\frac{1}{2}-\sigma} \nonumber \\
            & = 4C_0\mu \max
            \left\{1, \, \left(\frac{\mu}{M^{p-1}}\right)^{(\frac{1}{2}-\sigma)/(\theta_r+1-\frac{p}{r}-(\frac{1}{2}-\sigma))}, \, \left(\frac{1}{2C_1M^{p-2}}\right)^{(\frac{1}{2}-\sigma)/(\theta_r+1-\frac{p}{r})} \right\} \label{4.15}
        \end{align}
        from which we can show that if $\sigma > \frac12$ is sufficiently close to $\frac12$,
       a function $M_0 (\mu )$ can be found such that the right side of (\ref{4.15}) is less than or equal to $M$ for any $M \geq M_0 (\mu )$. With the choice of $T$ and $M$, we further obtain that $C_1 T^{\theta_r+1-\frac{p}{r}} M^{p-2}\le\frac{1}{2}$ in (\ref{4.14}).

       Now, consider the critical case (based upon the technique applied here) when $p=4$. Let $r=4$. By (\ref{3.8}), we claim that
        \begin{equation*}
            \left\|\mathcal{A}(u)\right\|_{\X^s_T} \lesssim C_T(\mu + \|u\|_{\X^s_T}^{p-1})\, , \quad
            \left\|\mathcal{A}(u)-\mathcal{A}(v)\right\|_{\X^s_T} \lesssim C_T M^{p-2} \|u-v\|_{\X^s_T}
        \end{equation*}
        where $C_T>0$ depends upon $T$ only. For $0<T<\infty$, let $M=2C_T\mu$. If $u$, $v\in B_{2C_T\mu}^{\X^s}$ with $\mu$ small enough so that $$
        \left\|\mathcal{A}(u)\right\|_{\X^s_T} \le C_T(\mu + M^{p-1}) \le 2C_T\mu
        $$
        and
        $$
        \left\|\mathcal{A}(u)-\mathcal{A}(v)\right\|_{\X^s_T} \lesssim \widetilde{C_T} \mu^{p-2} \|u-v\|_{\X^s_T}
        $$
        where $\widetilde{C_T} \mu^{p-2} \le \frac{1}{2}$, we can obtain (\ref{4.5}) and (\ref{4.6}).
        Note that this argument needs $\mu$ small. For $\mu$ not small, the technique in Sections 4.5 and 4.7 of \cite{Cazenave_NLSE} can be applied to prove the same results.

        For $\frac{1}{2} \le s <1$, by Sobolev embedding theorem, if $\frac{r(p-2)}{r-2} = \frac{2}{1-s}$, $H^s \hookrightarrow L^{\frac{r(p-2)}{r-2}}$, that is, $\|u(t)\|_{L^{\frac{r(p-2)}{r-2}}} \lesssim \|u(t)\|_{H^s}$, which holds if and only if
        \begin{equation*}
            \dfrac{4}{2-(p-2)(1-s)} = r < 4 \quad \text{or equivalently} \quad 3 \le p < \dfrac{3-2s}{1-s} \, .
        \end{equation*}
        We repeat the above argument to deduce that
        \begin{align*}
            \left\|D^{\alpha} f(u)(t)\right\|&_{L^{r'}} \lesssim \|u(t)\|_{L^{\frac{r(p-2)}{r-2}}}^{p-2} \| u(t)\|_{W^{s, r}} \le \|u(t)\|_{H^s}^{p-2} \| u(t)\|_{W^{s, r}} \, .
        \end{align*}
        Taking the norm with respect to $t$, it is obtained that
        \begin{equation*} 
            \left\|D^{\alpha} f(u)\right\|_{L^{r'}_t\left([0,T]; L^{r'}_{xy}(\R \times [0,1])\right)} \lesssim T^{1-\frac{2}{r}} \|u(t)\|_{L^{\infty}\left([0,T]; H^s\right)}^{p-2} \| u\|_{L^r_t\left([0,T]; W^{s, r}_{xy}(\R \times [0,1])\right)} \, ,
        \end{equation*}
        and
        \begin{equation*}\label{NLSE_IB_2d_stp_loc_WP_exiuni_lreg_M_est_slr} 
            \left\|f(u)\right\|_{L^{r'}_t\left([0,T]; W^{s,r'}_{xy}(\R \times [0,1])\right)} \lesssim T^{1-\frac{2}{r}} \|u(t)\|_{L^{\infty}\left([0,T]; H^s\right)}^{p-2} \| u\|_{L^r_t\left([0,T]; W^{s, r}_{xy}(\R \times [0,1])\right)} \, ,
        \end{equation*}
        which give
        \begin{equation*}\label{NLSE_IB_2d_stp_loc_WP_exiuni_lreg_M_est_lrwr} 
            \left\|f(u)\right\|_{L^{r'}_t\left([0,T]; W^{s,r'}_{xy}(\R \times [0,1])\right)} \lesssim T^{1-\frac{2}{r}} \|u\|_{\X^s_T}^{p-1} \, .
        \end{equation*}
        Similarly, for $p\ge3$,
        \begin{align*}
            & \left\|f(u)-f(v)\right\|_{L^{r'}_t\left([0,T]; W^{s,r'}_{xy}(\R \times [0,1])\right)} \\
            &\le T^{1-\frac{2}{r}} \left(\|u\|_{L^r_t\left([0,T]; W^{s,r}_{xy}(\R \times [0,1])\right)}^{p-2} + \|v\|_{L^r_t\left([0,T]; W^{s,r}_{xy}(\R \times [0,1])\right)}^{p-2} \right) \|u-v\|_{L^{\infty}_t\left([0,T]; H^s_{xy}(\R \times [0,1])\right)} \\
            & \qquad + T^{1-\frac{2}{r}} \left(\|u\|_{L^{\infty}_t\left([0,T]; H^s_{xy}(\R \times [0,1])\right)}^{p-2} + \|v\|_{L^{\infty}_t\left([0,T]; H^s_{xy}(\R \times [0,1])\right)}^{p-2} \right) \|u-v\|_{L^r_t\left([0,T]; W^{s,r}_{xy}(\R \times [0,1])\right)}\, .
        \end{align*}
        Thus, we can reach the similar conclusions as
        \begin{align*}
            & \left\|\mathcal{A}(u)\right\|_{\X^s_T} \lesssim \left(1+T^{\frac{1}{2}}+T^{\frac{1}{2}-\sigma}\right)\mu + T^{\theta_r+\frac{(p-2)(1-s)}{2}} \|u\|_{\X^s_T}^{p-1} \, ,\\
            & \left\|\mathcal{A}(u)-\mathcal{A}(v)\right\|_{\X^s_T} \lesssim T^{\theta_r+\frac{(p-2)(1-s)}{2}} M^{p-2} \|u-v\|_{\X^s_T} \, ,
        \end{align*}
        which lead to
        \begin{equation*}
            \left\|\mathcal{A}[u]\right\|_{\X^s_T} \le C_0 \left(\mu+ T^{\frac{1}{2}}\mu + T^{\frac{1}{2}-\sigma}\mu + T^{\theta_r+\frac{(p-2)(1-s)}{2}} M^{p-1} \right)\, , \label{NLSE_IB_2d_stp_loc_WP_exiuni_lreg_M_est}
        \end{equation*}
        and
        \begin{equation*}
            \left\|\mathcal{A}[u]-\mathcal{A}[v]\right\|_{\X^s_T} \le C_1 T^{\theta_r+\frac{(p-2)(1-s)}{2}} M^{p-2} \|u-v\|_{\X^s_T} \, . \label{NLSE_IB_2d_stp_loc_WP_exiuni_lreg_dist_est}
        \end{equation*}
        Hence, if letting
        $$
        T = \min\left\{1, \, \left(\dfrac{\mu}{M^{p-1}}\right)^{1/(\theta_r+\frac{(p-2)(1-s)}{2}-(\frac{1}{2}-\sigma))}, \, \left(\dfrac{1}{2C_1M^{p-2}}\right)^{1/(\theta_r+\frac{(p-2)(1-s)}{2})} \right\} \,  ,
        $$
        an argument similar to (\ref{4.15}) for $M$ can be used to obtain (\ref{4.5}) and (\ref{4.6}).

        If $p = \frac{3-2s}{1-s}$, then $r=4$, which is a critical situation again. For sufficiently small norms of the initial and boundary data  with some fixed $T>0$, (\ref{4.5}) and (\ref{4.6}) can also be obtained. Again, for large initial and boundary data, the argument in Sections 4.5 and 4.7 of \cite{Cazenave_NLSE} can be carried out.

        By contraction mapping theorem, we can conclude that the problem (\ref{2.1}) (or (\ref{2.5})) has a unique solution $u$ in $\X^s_T$ for $0\le s <\frac{1}{2}$ with $3\le p\le4$ or $\frac{1}{2}\le s <1$ with $3\le p \le \frac{3-2s}{1-s}$.

        Next, consider $s=1$. Let $r$ be arbitrarily chosen in the open interval $(2,4)$. It is known that $H^1 \hookrightarrow L^r$ for any $r<\infty$. The chain rule on the first derivative of $f(u)$ with respect to $(x,y)$ gives
        \begin{align*}
            \left\|\nabla f(u)(t)\right\|&_{L^{r'}} \lesssim \|u(t)\|_{L^{\frac{r(p-2)}{r-2}}}^{p-2} \| u(t)\|_{W^{1, r}} \le \|u(t)\|_{H^1}^{p-2} \| u(t)\|_{W^{1, r}} \, .
        \end{align*}
        Then, taking the norm in $t$, we have
        \begin{equation*} 
            \left\|\nabla f(u)\right\|_{L^{r'}_t\left([0,T]; L^{r'}_{xy}(\R \times [0,1])\right)} \lesssim T^{1-\frac{2}{r}} \|u(t)\|_{L^{\infty}\left([0,T]; H^1\right)}^{p-2} \| u\|_{L^r_t\left([0,T]; W^{1, r}_{xy}(\R \times [0,1])\right)} \, ,
        \end{equation*}
        and
        \begin{equation*}\label{NLSE_IB_2d_stp_loc_WP_exiuni_mreg_M_est_slr} 
            \left\|f(u)\right\|_{L^{r'}_t\left([0,T]; W^{1,r'}_{xy}(\R \times [0,1])\right)} \lesssim T^{1-\frac{2}{r}} \|u(t)\|_{L^{\infty}\left([0,T]; H^1\right)}^{p-2} \| u\|_{L^r_t\left([0,T]; W^{1, r}_{xy}(\R \times [0,1])\right)} \, ,
        \end{equation*}
        which imply
        \begin{equation*}\label{NLSE_IB_2d_stp_loc_WP_exiuni_mreg_M_est_lrwr} 
            \left\|f(u)\right\|_{L^{r'}_t\left([0,T]; W^{1,r'}_{xy}(\R \times [0,1])\right)} \lesssim T^{1-\frac{2}{r}} \|u\|_{\X^1_T}^{p-1} \, .
        \end{equation*}
        For the Lipschitz estimate with $p\ge3$,
        \begin{align*}
            & \left\|f(u)-f(v)\right\|_{L^{r'}_t\left([0,T]; W^{1,r'}_{xy}(\R \times [0,1])\right)} \\
            &\le T^{1-\frac{2}{r}} \left(\|u\|_{L^r_t\left([0,T]; W^{1,r}_{xy}(\R \times [0,1])\right)}^{p-2} + \|v\|_{L^r_t\left([0,T]; W^{1,r}_{xy}(\R \times [0,1])\right)}^{p-2} \right) \|u-v\|_{L^{\infty}_t\left([0,T]; H^1_{xy}(\R \times [0,1])\right)} \\
            & \qquad + T^{1-\frac{2}{r}} \left(\|u\|_{L^{\infty}_t\left([0,T]; H^1_{xy}(\R \times [0,1])\right)}^{p-2} + \|v\|_{L^{\infty}_t\left([0,T]; H^1_{xy}(\R \times [0,1])\right)}^{p-2} \right) \|u-v\|_{L^r_t\left([0,T]; W^{1,r}_{xy}(\R \times [0,1])\right)}\, .
        \end{align*}
        Thus, for any $r\in(2,4)$ and some constants $C_0$, $C_1$,
        \begin{align*}
            & \left\|\mathcal{A}(u)\right\|_{\X^1_T} \le C_0 \left(\mu +T^{\frac{1}{2}}\mu + T^{\frac{1}{2}-\sigma}\mu + T^{\theta_r+1-\frac{2}{r}} M^{p-1} \right)\, , \\
            & \left\|\mathcal{A}(u)-\mathcal{A}(v)\right\|_{\X^1_T} \le C_1 T^{\theta_r+1-\frac{2}{r}} M^{p-2} \|u-v\|_{\X^1_T} \, .
        \end{align*}
        Repeating the same argument as that for the case $0\le s <1$, if
        $$
         T = \min\left\{1, \, \left(\dfrac{\mu}{M^{p-1}}\right)^{1/(\theta_r+1-\frac{2}{r}-(\frac{1}{2}-\sigma))}, \, \left(\dfrac{1}{2C_1M^{p-2}}\right)^{1/(\theta_r+1-\frac{2}{r})} \right\}
        $$
        and $M$ is chosen similarly to (\ref{4.15}), then (\ref{4.5}) and (\ref{4.6}) are valid again. Thus, the contraction mapping theorem yields a unique solution in $\X^s_T$.
        Since $\theta_r+1-\frac{2}{r} > 0$, $\forall \, r\in(2,4)$, we do not need to impose any restrictions on $\mu$.

        After finishing the discussion for lower regularity, we now study the existence and uniqueness of the solution with $s>1$. We first derive a set of similar estimates.
        Assuming that $p$ and $s$ satisfy \eqref{1.4.1}, for $u$, $v \in B_M^{\Y^s}$ with $M>0$ given, it can be shown that
        \begin{align}
            &\|f(u)\|_{L^2((0,T); \, H^s)}  \lesssim T \|u\|_{L^{\infty}((0,T); \, H^s)}^{p-1} \, , \label{4.16} \\
            &\|f(u)-f(v)\|_{L^2((0,T); \, L^2)}  \lesssim T M^{p-2} \|u-v\|_{L^{\infty}((0,T); \, L^2)} \, . \label{4.17}
        \end{align}

        First, assume that  $p$ and $s$ satisfy \eqref{4.8.1}, which is stronger than \eqref{1.4.1}. The Lipschitz continuity holds for $H^s$ norm, i.e.,
        \begin{equation}
            \|f(u)-f(v)\|_{L^1((0,T); \, H^s)} \lesssim T M^{p-2} \|u-v\|_{L^{\infty}((0,T); \, H^s)} \, . \label{4.18}
        \end{equation}
        By (\ref{4.16}) and (\ref{4.18}),
        \begin{align*}
            & \left\|\mathcal{A}(u)\right\|_{\Y^s_T} \le \|W_b[h_1,h_2]\|_{\Y^s_T} + \left\|W_0\phi \right\|_{\X_T} + \left\|\Phi_{0,f}\right\|_{\Y^s_T} \\
            & \quad \lesssim  \left(1+T^{\frac{1}{2}}+T^{\frac{1}{2}-\sigma}\right)\mu + T^{\theta_r} \left\|f \right\|_{L^2_t \left([0,T]; H^s_{xy} \right)} \lesssim \left(1+T^{\frac{1}{2}}+T^{\frac{1}{2}-\sigma}\right)\mu + T^{\theta_r+1} \|u\|_{L^{\infty}((0,T); \, H^s)}^{p-1} \\
            &\qquad \le C_0 \left(1+T^{\frac{1}{2}}+T^{\frac{1}{2}-\sigma}\right)\mu + T^{\theta_r+1} M^{p-1} \, ,
        \end{align*}
        as well as
        \begin{align*}
            \left\|\mathcal{A}(u)-\mathcal{A}(v)\right\|&_{\Y^s_T} = \left\|\Phi_{0,f}(u)-\Phi_{0,f}(v) \right\|_{\Y^s_T}
            \lesssim T^{\theta_r} \left\|f(u)-f(v) \right\|_{L^2_t \left([0,T]; H^s_{xy} \right)} \\
            \lesssim & C_1T^{\theta_r+1} M^{p-2} \|u-v\|_{\Y^s_T} \, ,
        \end{align*}
        where $C_0, C_1$ are constants.
        Now, we choose $\sigma > \frac{1}{2}$ but close enough to $\frac{1}{2}$ so that $\theta_r>0$ again. Let $M$ and $T$ be such that  $$T = \min\left\{1, \, \left(\dfrac{\mu}{M^{p-1}}\right)^{1/(\theta_r+\sigma))}, \, \left(\dfrac{1}{2C_1M^{p-2}}\right)^{1/(\theta_r+\frac{1}{2})} \right\}$$ and $M$ is found similarly in (\ref{4.15}), which yield (\ref{4.8}) and (\ref{4.9}).
        By the contraction mapping principle, we can find a fixed point $u \in \Y^s_T$.

        For other $p$ and $s$ satisfying \eqref{1.4.1}, we can argue as follows. Since $s+1\le p< s+2$ for $s\in\Z$ or $[s]+2 \le p<[s]+3$ for $s\notin\Z$ with $p$ not even, (\ref{4.17}) gives
        \begin{equation}
            \left\|\mathcal{A}[u]-\mathcal{A}[v]\right\|_{\mathcal{Y}^0_T} \le C_1 T^{\theta_r+1} M^{p-2} \|u-v\|_{\mathcal{Y}^0_T}\, . \label{4.19}
        \end{equation}
        Hence, (\ref{4.7}) holds for the same choices of $M$ and $T$.
With (\ref{4.19}), a fixed point $u \in \Y^0_T$ can also be found by finding a sequence $\{u_n\}_{n\ge1} \subset B_M^{\Y^s} \subset \Y^s_T$ converging to $u$ in $\Y^0_T$; i.e., choose $u_1 \in B_M^{\Y^s}$. Then let $u_2 = \mathcal{A}[u_1]$, which satisfies $u_2 \in B_M^{\Y^s}$ because of (\ref{4.8}). We keep defining $u_{n+1} = \mathcal{A}[u_n]$ and again $u_{n+1} \in B_M^{\Y^s}$; moreover $\|u_{n+1}-u_n\|_{\Y^0_T} = \left\|\mathcal{A}[u_n]-\mathcal{A}[u_{n-1}]\right\|_{\Y^0_T} \le \frac{1}{2} \|u_n-u_{n-1}\|_{\Y^0_T}$ by (\ref{4.7}). Thus, a Cauchy sequence $\{u_n\}_{n\ge1} \subset B_M^{\Y^s} \subset \Y^s_T$ with $u_n \to u$ in $\Y^0_T$ is created. We may see that the problem has a unique solution in $\Y^s_T$. Since $\Y^s_T$ is also a reflexive Banach space, we conclude that $u \in \Y^s_T$ is the unique solution of (\ref{2.5}) for $s>1$. Note that the Lipschitz continuity for $\mathcal{A}$ in $\Y^s_T$ does not hold for this case. The proof is finished.
    \end{proof}

    Next,  we consider the maximal existence interval $[0,T_{\max})$ of the solution found in \emph{Theorem~\ref{thm_4.3}}.
    \begin{thm}\label{thm_4.4}
        Let $p\ge3$  and $s \geq 0$ satisfying \eqref{1.4.1}. Assume that a unique solution $u$ to (\ref{2.1}) exists in $\X^s_T$ if $0 \le s < \frac{1}{2}$ with $3\le p \le 4$ or $\frac{1}{2} \le s < 1$ with $3\le p < \frac{3-2s}{1-s}$ or $s=1$ with $3\le p < \infty$, or in $\Y^s_T$ if $s>1$, for $t \in [0,T] $ with $\varphi \in H^s (\R \times [0,1])$ and $h_j \in \HH^s_{loc} (0, \infty)$ for $T > 0$, $j=1$, $2$. Let $T_{\max}= \sup T$ and suppose $T_{\max}<\infty$. Also, define $u^{\ast}$ on $[0,T_{\max})$ as the solution of (\ref{2.1}) in $\X^s_{T_{\max}}$ if $0 \le s < \frac{1}{2}$ with $3\le p \le 4$ or $\frac{1}{2} \le s < 1$ with $3\le p < \frac{3-2s}{1-s}$ or $s=1$ with $3\le p < \infty$, or $\Y^s_{T_{\max}}$ if $s>1$, with $u^{\ast}(t) = u(t)$ on $[0,T]$ whose existence and uniqueness have been proved in Theorem~\ref{thm_4.3}. Then, $\displaystyle \lim_{t\uparrow T_{\max}} \|u^*(t)\|_{H^s(\R \times [0,1])} = \infty$.
    \end{thm}

    \begin{proof}
        The proof can be obtained using classical extension procedure as an analogue of the proof for the maximal existence interval in \cite{Ran_NLSE_2d_uhp} except that the discussion is performed on the domain $\R\times[0,1]$.
    \end{proof}

    The continuous dependence property of solutions on the initial and boundary data can be derived as follows.
    \begin{thm} 
        Let $p$ and $s$ satisfy the conditions described in Theorem \ref{thm_4.4}.
        Assume $\{\varphi_n\}$ be a sequence of functions in $H^s(\R \times [0,1])$ and $\varphi \in H^s(\R \times [0,1])$ so that $\varphi_n \to \varphi$ as $n \to \infty$ in $H^s(\R \times [0,1])$. Let $ h_1, h_2 $  and $h_{1,n}, h_{2, n}, n = 1, 2, \cdots$  satisfy
        $$
            h_1, h_2, h_{1,n}, h_{2,n} \in \HH^s(0, T) \quad \mbox{ and } \quad h_{1,n} \to h_1, \, h_{2,n} \to h_2
        $$
        as $n \to \infty$ in $\HH^s(0, T)$ for some $T>0$. Let $u_n$ be the solutions to (\ref{2.1}) with $u_n(x,y,0)=\varphi_n(x,y)$ and  $u_n(x, 0, t) =h_{1,n}(x,t), u_n(x, 1, t) =h_{2,n}(x,t)$ and $u$ be the solution with $u(x,y,0)=\varphi(x,y)$ and $u(x, 0, t) =h_{1}(x,t), u(x, 1, t) =h_{2}(x,t)$, respectively. Then $u_n \to u$ as $n \to \infty$ in $X_T$ with $\left\|u_n\right\|_{X_T}\le M$ where $X_T=\X^s_T$ if $0 \le s < \frac{1}{2}$ with $3\le p \le 4$ or $\frac{1}{2} \le s < 1$ with $3\le p < \frac{3-2s}{1-s}$ or $s=1$ with $3\le p < \infty$, or $\Y^s_T$ if $s>1$, respectively.
    \end{thm}

    \begin{proof}
        First, consider $0 \le s \le 1$ with assumptions on $p$ and $r$ given in \emph{Theorems~\ref{thm_4.3}} and \emph{\ref{thm_4.4}} which guarantee the existence of a common interval $[0, T_c]$ for $u_n$ and $u$  because of the choice of $T_{\max}$ only dependent upon the initial and boundary conditions. Furthermore, from the proof of (\ref{4.5}), for $\theta_r$ defined by (\ref{4.12}) and $\sigma>\frac{1}{2}$, we can obtain
        \begin{align*}
            & \left\|u-u_n\right\|_{\X^s_T} = \left\|\mathcal{A}[u]-\mathcal{A}[u_n]\right\|_{\X^s_T} \le C \bigg ( \left(T^{\frac{1}{2}} + T^{\frac{1}{2}-\sigma} \right) \left\|\varphi-\varphi_n\right\|_{H^s(\R \times [0,1])} \\
            & \qquad + \sum_{j=1,2}\left\|h_j-h_{j,n}\right\|_{\HH^s(0, T) } + T^{\theta_r} M^{p-2} \|u-u_n\|_{\X^s_T}\bigg ) \, .
        \end{align*}
        Let $T$ satisfy $C T^{\theta_r} M^{p-2} \leq 1/2$. Then
        \begin{equation*}
            \left \|u-u_n\right\|_{\X^s_T} \le 2 \widetilde{C} \bigg ( \left\|\varphi-\varphi_n\right\|_{H^s(\R \times [0,1])} + \sum_{j=1,2}\left\|h_j-h_{j,n}\right\|_{\HH^s(0, T) }  \bigg ) \, .
        \end{equation*}
        Since $T$ only depends upon the uniform bounds for $u, u_n, \varphi, \varphi_n , h, h_n$ in their respective norms with $t \in [0, T_c]$, the above inequality holds for $T, 2T, \dots$ until reaching $T_c$. The continuous dependence is proved for $0 \leq s \le 1$.

        For $s > 1$, first notice that by (\ref{4.19}), $$\left\|\mathcal{A}[u]-\mathcal{A}[v]\right\|_{\Y^0_T} \le  C_1 T^{\theta_r+1} M^{p-2} \, \|u-v\|_{\Y^0_T} .$$ Hence
        \begin{align*}
            & \left\|u-u_n\right\|_{\Y^0_T} = \left\|\mathcal{A}[u]-\mathcal{A}[u_n]\right\|_{\Y^0_T} \le C \bigg ( \left\|\varphi-\varphi_n\right\|_{L^2(\R \times [0,1])} \\
            & \qquad + \sum_{j=1,2}\left\|h_j-h_{j,n}\right\|_{\HH^s(0, T) } + T^{\theta_r+1} M^{p-2} \|u-u_n\|_{\Y^0_T} \bigg ) \, .
        \end{align*}
        Thus, choosing $T\le1$ so that $C T^{\theta_r+1} M^{p-2} \leq 1/3$, we have
        \begin{equation*}
            \left\|u-u_n\right\|_{\Y^0_T} \le 2\widetilde{C} \bigg ( \left\|\varphi-\varphi_n\right\|_{L^2(\R \times [0,1])} + \sum_{j=1,2}\left\|h_j-h_{j,n}\right\|_{\HH^s(0, T) } \bigg ) \, \to 0
        \end{equation*}
        as $n\to 0$. To show the continuous dependence in $\Y^s_T$, note that it can be verified on the strip domain $\R\times [0,1]$ that
        \begin{align*}
            \|&f(u)-f(v)\|_{L^1((0,T); \, H^s(\R\times [0,1]))} \nonumber \\
            & \lesssim T M^{p-2} \|u-v\|_{L^{\infty}((0,T); \, H^s(\R\times [0,1]))} + T^{\theta_r+1} \varepsilon\{\|u-v\|_{L^{\infty}((0,T); \, L^2(\R\times [0,1]))}\} \, , 
        \end{align*}
        where $\varepsilon: \R^+ \to \R^+$ is a continuous function so that $\varepsilon(t)\to0^+$ as $t\to0^+$. Thus, it is straightforward to derive
        \begin{align*}
            &\left\|u-u_n\right\|_{\Y^s_T} = \left\|\mathcal{A}[u]-\mathcal{A}[u_n]\right\|_{\Y^s_T}  \le C \bigg (\left(T^{\frac{1}{2}} + T^{\frac{1}{2}-\sigma} \right) \left\|\varphi-\varphi_n\right\|_{H^s(\R \times [0,1])} \\
            & \quad + \sum_{j=1,2}\left\|h_j-h_{j,n}\right\|_{\HH^s(0, T) )}  + T^{\theta_r+1} M^{p-2} \|u-u_n\|_{\Y^s_T} + T^{\theta_r+1} \varepsilon\{\|u-u_n\|_{\Y^0_T}\} \bigg ) \, .
        \end{align*}
        Rewrite this as
        \begin{align*}
            \left\|u-u_n\right\|_{\Y^s_T} \le & 2C \bigg ( \left\|\varphi-\varphi_n\right\|_{L^2(\R \times [0,1])} + \sum_{j=1,2}\left\|h_j-h_{j,n}\right\|_{\HH^s(0, T) } + \varepsilon\{\|u-u_n\|_{\Y^0_T}\} \bigg ) \, .
        \end{align*}
        Since the right hand side of the inequality approaches zero as $n\to0$, so does the left hand side. Hence, the proof of the continuous dependence is completed.
    \end{proof}
This completes the proof of the statements \emph{(i)-(iii)} in \emph{Theorem~\ref{thm_main}}.

    Now, we discuss the possibility of removing the auxiliary space from the well-posedness for $0\le s \le 1$. In the proof of the local well-posedness, the regularity property and conditional well-posedness of (\ref{2.1}) are discussed for $0\leq s\leq 1$. By \emph{(i) and (ii)} of \emph{Theorem~\ref{thm_main}}, the argument in Section 4 of \cite{Sun_CUC_Evo} and Sections 5.1-5.5 in \cite{Cazenave_NLSE} can provide the proof of the following persistence of regularity result, i.e., if $0\le s_1 < s$, let $u$ in $\X^{s_1}_{T_{\max}}$  be the unique solution of (\ref{2.1}) with the maximal existence interval $[0,T_{\max})$ under the assumption that $0 \le s_1 < 1/2$ with $3\le p \le 4$ or $1/2 \le s_1 < 1$ with $3\le p < \frac{3-2s}{1-s_1}$ or $s_1=1$ with $3\le p < \infty$, and $[0,T_{\max})$ be given in \emph{Theorem~\ref{thm_4.3}} with $\varphi \in H^{s_1} (\R \times [0,1])$ and $h_j \in \HH^{s_1}(0, T) $, $j=1$, $2$, with $T \geq T_{\max}$. If $\varphi \in H^s (\R \times [0,1])$ and $h_j \in \HH^{s}(0, T) $ with $s >1$, then $u$ is also in $ \Y^s_{T_{\max}}$.

    \begin{prop} \label{prop_4.6}
        For $0\le s\leq 1$, the IBVP (\ref{2.1}) has the property of  persistence of regularity.
    \end{prop}
    Then, by applying the unconditional well-posedness theorem in \cite{Sun_CUC_Evo} for (\ref{2.1}), the following theorem can be obtained, which gives \emph{(iii)} of \emph{Theorem~\ref{thm_main}}.
    \begin{thm}
        For $0\le s \le 1$, the problem (\ref{2.1}) is unconditionally well-posed.
    \end{thm}

    \begin{proof}
        The claim of this theorem is a result directly from \emph{Theorem 2.6} in \cite{Sun_CUC_Evo} and \emph{Proposition~\ref{prop_4.6}}.
    \end{proof}

\section{Global Well-posedness}\label{sec_5}

    In this section, we investigate the global existence of solution for (\ref{2.1})  with $T\in(0,\infty]$. We first prove the following identities.
    \begin{lemma}
        If the solution of (\ref{2.1}) exists for any $t>0$ and is sufficiently smooth, then for arbitrary smooth function $\eta=\eta(y)$
        \begin{align}
            & \, (|u|^2)_t = -2 \text{Im} \left[(u_x \overline{u})_x + (u_y \overline{u})_y\right] \, , \label{5.1} \\[.08in]
            & \left(|u_x|^2+|u_y|^2-\dfrac{2\lambda}{p}|u|^p\right)_t = 2 \text{Re} \left[(\overline{u_t} u_x)_x + (\overline{u_t} u_y)_y\right] \, , \label{5.2} \\
            & \left(|u_y|^2-|u_x|^2+\dfrac{2\lambda}{p}|u|^p\right)_y = -2 \text{Re}(\overline{u}_x u_y)_x - i(u \overline{u_y})_t + i(u \overline{u_t})_y \, , \label{5.3} \\
            & -\eta \left(|u_y|^2-|u_x|^2+\dfrac{2\lambda}{p}|u|^p\right)_y = 2\text{Re}(\eta u_y \overline{u}_x)_x + i(\eta u \overline{u_y})_t - i (\eta \overline{u_t} u)_y \nonumber \\
            & \qquad\qquad\qquad + \eta_y \left(u \overline{u}_x\right)_x - \eta_y |u_x|^2 + \eta_y \left(u \overline{u_y}\right)_y - \eta_y |u_y|^2 + \lambda \eta_y |u|^p \, , \label{5.4} 
        \end{align}
        and
        \begin{align}\label{5.5} 
            &\left[\left(y-\dfrac{1}{2}\right) \left(|u_y|^2-|u_x|^2 \right)\right]_y \nonumber \\
            &\qquad = 2|u_y|^2 - 2\text{Re}\left[\left(y-\dfrac{1}{2}\right) u_y \overline{u}_x\right]_x - i\left[\left(y-\dfrac{1}{2}\right) u \overline{u}_y\right]_t + i \left[\left(y-\dfrac{1}{2}\right) \overline{u}_t u\right]_y \nonumber \\
            & \qquad\quad - \, \left(u \overline{u}_x\right)_x - \left(u \overline{u}_y\right)_y - \lambda \left(1-\dfrac{2}{p}\right) |u|^p - \dfrac{2\lambda}{p} \left[\left(y-\dfrac{1}{2}\right) |u|^p\right]_y\, .
        \end{align}
    \end{lemma}

    \begin{proof}
        The proofs of (\ref{5.1}) to (\ref{5.3}) can be found in \cite{Ran_NLSE_2d_uhp}. For (\ref{5.4}), we multiply (\ref{5.3}) by $\eta$ to obtain
        \begin{align*}
            -\eta &\left ( |u_y|^2-|u_x|^2 \right ) _y = 2\text{Re}(\eta u_y \overline{u}_x)_x + i(\eta u \overline{u_y})_t - i \eta (\overline{u_t} u)_y + \eta \left(\dfrac{2\lambda}{p}|u|^p\right)_y \\
            &= 2\text{Re}(\eta u_y \overline{u}_x)_x + i(\eta u \overline{u_y})_t - i (\eta \overline{u_t} u)_y + i \eta_y \overline{u_t} u + \eta \left(\dfrac{2\lambda}{p}|u|^p\right)_y \\
            &= 2\text{Re}(\eta u_y \overline{u}_x)_x + i(\eta u \overline{u_y})_t - i (\eta \overline{u_t} u)_y + \eta \left(\dfrac{2\lambda}{p}|u|^p\right)_y + \eta_y u \left(\overline{u}_{xx} + \overline{u}_{yy} + \lambda \overline{u} |u|^{p-2} \right) \\
            &= 2\text{Re}(\eta u_y \overline{u}_x)_x + i(\eta u \overline{u_y})_t - i (\eta \overline{u_t} u)_y + \eta \left(\dfrac{2\lambda}{p}|u|^p\right)_y + \eta_y u \overline{u}_{xx} + \eta_y u \overline{u}_{yy} + \lambda \eta_y |u|^p \\
            &= 2\text{Re}(\eta u_y \overline{u}_x)_x + i(\eta u \overline{u_y})_t - i (\eta \overline{u_t} u)_y + \eta \left(\dfrac{2\lambda}{p}|u|^p\right)_y \\
            & \qquad + \eta_y \left(u \overline{u}_x\right)_x - \eta_y |u_x|^2 + \eta_y \left(u \overline{u_y}\right)_y - \eta_y |u_y|^2 + \lambda \eta_y |u|^p\, .
        \end{align*}
        To prove (\ref{5.5}), we replace $\eta$ in (\ref{5.4}) by $y-\frac{1}{2}$.
        \begin{align*}
            -\bigg[\left(y-\dfrac{1}{2}\right)&\left(|u_y|^2-|u_x|^2\right)\bigg]_y + \left(|u_y|^2-|u_x|^2 \right) = -\left(y-\dfrac{1}{2}\right) \left(|u_y|^2-|u_x|^2\right)_y \\
            &= 2\text{Re}\left[\left(y-\dfrac{1}{2}\right) u_y \overline{u}_x\right]_x + i\left[\left(y-\dfrac{1}{2}\right) u \overline{u_y}\right]_t - i \left[\left(y-\dfrac{1}{2}\right) \overline{u_t} u\right]_y \\
            & + \, \left(u \overline{u}_x\right)_x - |u_x|^2 + \left(u \overline{u_y}\right)_y - |u_y|^2 + \lambda |u|^p + \left(y-\dfrac{1}{2}\right) \left(\dfrac{2\lambda}{p}|u|^p\right)_y \\
            &= 2\text{Re}\left[\left(y-\dfrac{1}{2}\right) u_y \overline{u}_x\right]_x + i\left[\left(y-\dfrac{1}{2}\right) u \overline{u_y}\right]_t - i \left[\left(y-\dfrac{1}{2}\right) \overline{u_t} u\right]_y \\
            & + \, \left(u \overline{u}_x\right)_x - |u_x|^2 + \left(u \overline{u_y}\right)_y - |u_y|^2 + \lambda |u|^p + \left(\left(y-\dfrac{1}{2}\right) \dfrac{2\lambda}{p}|u|^p\right)_y - \dfrac{2\lambda}{p}|u|^p\, .
        \end{align*}
    \end{proof}

    Next, the following \emph{a-priori} estimate of the solution to (\ref{2.1}) in $H^1(\R \times [0,1])$ is derived.
    \begin{prop}\label{prop_5.2} 
        Suppose that either $p\ge 3$ and $\lambda<0$ or $p=3$ and $\lambda>0$. For any given $T>0$ and a solution $u$ of (\ref{2.1}) in $C_t \left([0,T]; \, H^1_{xy}(\R \times [0,1])\right)$, there is a $\psi: \, \R^+ \to \R^+$ as a nondecreasing function of the norms of $\varphi \in H^1(\R \times [0,1])$ and $h_1, h_2 \in H^1(\R \times [0,T])$ such that
        \begin{equation*}\label{NLSE_IB_2d_stp_glb_lh1h1_est}
            \sup_{t\in[0,T]} \|u(t)\|_{H^1(\R \times [0,1])} \le \psi\left (\|\varphi\|_{H^1(\R \times [0,1])}+\|h_1\|_{H^1(\R \times [0,T])} + \|h_2\|_{H^1(\R \times [0,T])}\right )\, .
        \end{equation*}
    \end{prop}

    \begin{proof}
        We first integrate (\ref{5.5}) with respect to $x$, $y$ and $t$ to obtain
        \begin{align*}
             \dfrac{1}{2} \int_{-\infty}^{\infty} &\int_0^t \left\{\left(|u_y(x,1,\tau)|^2 + |u_y(x,0,\tau)|^2 \right) - \left(|u_x(x,1,\tau)|^2 + |u_x(x,0,\tau)|^2 \right)\right\} \, d\tau \, dx \\
            & = 2 \int_{-\infty}^{\infty} \int_0^1 \int_0^t |u_y|^2 \, d\tau \, dy \, dx - i \int_{-\infty}^{\infty} \int_0^1 \left(y-\dfrac{1}{2}\right) u(t) \overline{u_y}(t) \, dy \, dx \\
            & \quad + i \int_{-\infty}^{\infty} \int_0^1 \left(y-\dfrac{1}{2}\right) \varphi \overline{\varphi_y} \, dy \, dx + \dfrac{i}{2} \int_{-\infty}^{\infty} \int_0^t \left(\left(\overline{h_2}\right)_t h_2 + \left(\overline{h_1}\right)_t h_1\right) \, d\tau \, dx \\
            & \quad - \int_{-\infty}^{\infty} \int_0^t  h_2 \overline{u_y}(x,1,\tau) \, d\tau \, dx + \int_{-\infty}^{\infty} \int_0^t  h_1 \overline{u_y}(x,0,\tau) \, d\tau \, dx - \dfrac{\lambda}{p} \int_{-\infty}^{\infty} \int_0^t |h_2|^p \, d\tau \, dx \\
            & \quad - \dfrac{\lambda}{p} \int_{-\infty}^{\infty} \int_0^t |h_1|^p \, d\tau \, dx - \lambda \left(1-\dfrac{2}{p}\right) \int_{-\infty}^{\infty} \int_0^1 \int_0^t |u|^p \, d\tau \, dy \, dx \, ,
        \end{align*}
        which yields
        \begin{align*}
             \int_{-\infty}^{\infty} &\int_0^t \left(|u_y(x,1,\tau)|^2 + |u_y(x,0,\tau)|^2 \right) \, d\tau \, dx \\
            & = 4 \|u_y\|_{L^2_{xyT}}^2 - 2i \int_{-\infty}^{\infty} \int_0^1 \left(y-\dfrac{1}{2}\right) u(t) \overline{u_y}(t) \, dy \, dx + 2i \int_{-\infty}^{\infty} \int_0^1 \left(y-\dfrac{1}{2}\right) \varphi \overline{\varphi_y} \, dy \, dx \\
            & \quad + i \int_{-\infty}^{\infty} \int_0^t \left(\left(\overline{h_2}\right)_t h_2 + \left(\overline{h_1}\right)_t h_1\right) \, d\tau \, dx - 2 \int_{-\infty}^{\infty} \int_0^t  h_2 \overline{u_y}(x,1,\tau) \, d\tau \, dx \\
            & \quad + 2 \int_{-\infty}^{\infty} \int_0^t  h_1 \overline{u_y}(x,0,\tau) \, d\tau \, dx - \dfrac{2\lambda}{p} \|h_2\|_{L^p_{xT}}^p - \dfrac{2\lambda}{p} \|h_1\|_{L^p_{xT}}^p - 2 \lambda \left(1-\dfrac{2}{p}\right) \|u\|_{L^p_{xyT}}^p \\
            & \quad + \|(h_1)_x\|_{L^2_{xT}}^2 + \|(h_2)_x\|_{L^2_{xT}}^2\, ,
        \end{align*}
        where $L^p_{xT} $ or $L^p_{xyT} $ stands for the norm with $t$-integral from $0$ to $T$.
        Consider the first case $\lambda<0$ and recall the inequality $ab \le (1/q) {(\vartheta a)^q} + (1/q') {(b /{\vartheta} )^{q'}}$ for $\vartheta$, $a$, $b\ge0$,
        \begin{align*}
            & \int_{-\infty}^{\infty} \int_0^t \left(|u_y(x,1,\tau)|^2 + |u_y(x,0,\tau)|^2 \right) \, d\tau \, dx \\
            &\le 4 \|u_y\|_{L^2_{xyT}}^2 + \|u(t)\|_{L^2_{xy}} \|u_y(t)\|_{L^2_{xy}} + \|\varphi\|_{L^2_{xy}} \|\varphi_y\|_{L^2_{xy}} + \|(h_1)_x\|_{L^2_{xT}}^2 + \|(h_2)_x\|_{L^2_{xT}}^2 \\
            & \quad + \left\|(h_1)_t\right\|_{L^2_{xT}} \|h_1\|_{L^2_{xT}} + \left\|(h_2)_t\right\|_{L^2_{xT}} \|h_2\|_{L^2_{xT}} - \dfrac{2\lambda}{p} \|h_2\|_{L^p_{xT}}^p - \dfrac{2\lambda}{p} \|h_1\|_{L^p_{xT}}^p - 2 \lambda \left(1-\dfrac{2}{p}\right) \|u\|_{L^p_{xyT}}^p \\
            & \qquad + 2 \|h_1\|_{L^2_{xT}} \left(\int_{-\infty}^{\infty} \int_0^t |u_y(x,0,\tau)|^2 \, d\tau \, dx \right)^{\frac{1}{2}} + 2 \|h_2\|_{L^2_{xT}} \left(\int_{-\infty}^{\infty} \int_0^t |u_y(x,1,\tau)|^2 \, d\tau \, dx \right)^{\frac{1}{2}} \\
            &\le 4 \|u_y\|_{L^2_{xyT}}^2 + \|u(t)\|_{L^2_{xy}} \|u_y(t)\|_{L^2_{xy}} + \|\varphi\|_{L^2_{xy}} \|\varphi_y\|_{L^2_{xy}} + \|(h_1)_x\|_{L^2_{xT}}^2 + \|(h_2)_x\|_{L^2_{xT}}^2 \\
            & \quad + \left\|(h_1)_t\right\|_{L^2_{xT}} \|h_1\|_{L^2_{xT}} + \left\|(h_2)_t\right\|_{L^2_{xT}} \|h_2\|_{L^2_{xT}} - \dfrac{2\lambda}{p} \|h_2\|_{L^p_{xT}}^p - \dfrac{2\lambda}{p} \|h_1\|_{L^p_{xT}}^p - 2 \lambda \left(1-\dfrac{2}{p}\right) \|u\|_{L^p_{xyT}}^p \\
            & \qquad + 2 \|h_1\|_{L^2_{xT}}^2 + 2 \|h_2\|_{L^2_{xT}}^2 + \dfrac{1}{2} \int_{-\infty}^{\infty} \int_0^t \left(|u_y(x,1,\tau)|^2 + |u_y(x,0,\tau)|^2 \right) \, d\tau \, dx\, ,
        \end{align*}
        which provides the $L^2$-norm of $u_y$ with respect to $x$, $t$ on the boundary $y=0$ and $y=1$:
        \begin{align}\label{5.6} 
            \int_{-\infty}^{\infty} & \int_0^t \left(|u_y(x,1,\tau)|^2 + |u_y(x,0,\tau)|^2 \right) \, d\tau \, dx \nonumber \\
            &\le 8 \|u_y\|_{L^2_{xyT}}^2 + 2 \|u(t)\|_{L^2_{xy}} \|u_y(t)\|_{L^2_{xy}} - 4 \lambda \left(1-\dfrac{2}{p}\right) \|u\|_{L^p_{xyT}}^p \nonumber \\
            & \quad + 2 \left\|(h_1)_t\right\|_{L^2_{xT}} \|h_1\|_{L^2_{xT}} + 2 \left\|(h_2)_t\right\|_{L^2_{xT}} \|h_2\|_{L^2_{xT}} - \dfrac{4\lambda}{p} \|h_2\|_{L^p_{xT}}^p - \dfrac{4\lambda}{p} \|h_1\|_{L^p_{xT}}^p \nonumber \\
            & \qquad + 4 \|h_1\|_{L^2_{xT}}^2 + 4 \|h_2\|_{L^2_{xT}}^2 + 2 \|(h_1)_x\|_{L^2_{xT}}^2 + 2 \|(h_2)_x\|_{L^2_{xT}}^2 + 2 \|\varphi\|_{L^2_{xy}} \|\varphi_y\|_{L^2_{xy}} \, .
        \end{align}
        For each $t\in [0,T]$, (\ref{5.1}) yields
        \begin{align*}
            \| & u(t)\|_{L^2_{xy}}^2 = \int_{-\infty}^{\infty} \int_0^1 |u(x,y,t)|^2 \, dy \, dx \\
            &= \int_{-\infty}^{\infty} \int_0^1 |u(x,y,0)|^2 \, dy \, dx + \int_{-\infty}^{\infty} \int_0^1 \int_0^t \left(|u(x,y,\tau)|^2\right)_t \, d\tau \, dy \, dx \\
            &= \|\varphi\|_{L^2_{xy}}^2 - 2 \text{Im} \int_{-\infty}^{\infty} \int_0^1 \int_0^t \left[(u_x(x,y,\tau) \overline{u}(x,y,\tau))_x + (u_y(x,y,\tau) \overline{u}(x,y,\tau))_y\right] \, d\tau \, dy \, dx \\
            &= \|\varphi\|_{L^2_{xy}}^2 + 2 \text{Im} \int_{-\infty}^{\infty} \int_0^t  u_y(x,0,\tau) \overline{u}(x,0,\tau) \, d\tau \, dx - 2 \text{Im} \int_{-\infty}^{\infty} \int_0^t u_y(x,1,\tau) \overline{u}(x,1,\tau)  \, d\tau \, dx \\
            &\le \int_{-\infty}^{\infty} \int_0^t \left(|u_y(x,1,\tau)|^2 + |u_y(x,0,\tau)|^2 \right) \, d\tau \, dx + \|h_1\|_{L^2_{xT}}^2 + \|h_2\|_{L^2_{xT}}^2 + \|\varphi\|_{L^2_{xy}}^2\, .
        \end{align*}
        Use (\ref{5.6}) to replace $\int_{-\infty}^{\infty} \int_0^t \left(|u_y(x,1,\tau)|^2 + |u_y(x,0,\tau)|^2 \right) d\tau dx$ in the inequality above,
        \begin{align*}
            \|u(t)\|_{L^2_{xy}}^2 &\le 8 \|u_y\|_{L^2_{xyT}}^2 + 2 \|u(t)\|_{L^2_{xy}} \|u_y(t)\|_{L^2_{xy}} - 4 \lambda \left(1-\dfrac{2}{p}\right) \|u\|_{L^p_{xyT}}^p \\
            & \quad + \, 2 \left\|(h_1)_t\right\|_{L^2_{xT}} \|h_1\|_{L^2_{xT}} + 2 \left\|(h_2)_t\right\|_{L^2_{xT}} \|h_2\|_{L^2_{xT}} - \dfrac{4\lambda}{p} \|h_2\|_{L^p_{xT}}^p - \dfrac{4\lambda}{p} \|h_1\|_{L^p_{xT}}^p \\
            & \qquad + \, 4 \|h_1\|_{L^2_{xT}}^2 + 4 \|h_2\|_{L^2_{xT}}^2 + 2 \|(h_1)_x\|_{L^2_{xT}}^2 + 2 \|(h_2)_x\|_{L^2_{xT}}^2 + 2 \|\varphi\|_{L^2_{xy}} \|\varphi_y\|_{L^2_{xy}} \\
            & \qquad \quad + \, \|h_1\|_{L^2_{xT}}^2 + \|h_2\|_{L^2_{xT}}^2 + \|\varphi\|_{L^2_{xy}}^2 \\
            \le & 8 \|u_y\|_{L^2_{xyT}}^2 + \left(\dfrac{1}{2} \|u(t)\|_{L^2_{xy}}^2 + 2 \|u_y(t)\|^2_{L^2_{xy}} \right) - 4 \lambda \left(1-\dfrac{2}{p}\right) \|u\|_{L^p_{xyT}}^p \\
            & \ + \, 2 \left\|(h_1)_t\right\|_{L^2_{xT}} \|h_1\|_{L^2_{xT}} + 2 \left\|(h_2)_t\right\|_{L^2_{xT}} \|h_2\|_{L^2_{xT}} - \dfrac{4\lambda}{p} \|h_2\|_{L^p_{xT}}^p - \dfrac{4\lambda}{p} \|h_1\|_{L^p_{xT}}^p \\
            & \quad + \, 5 \|h_1\|_{L^2_{xT}}^2 + 5 \|h_2\|_{L^2_{xT}}^2 + 2 \|(h_1)_x\|_{L^2_{xT}}^2 + 2 \|(h_2)_x\|_{L^2_{xT}}^2 + 2 \|\varphi\|_{L^2_{xy}} \|\varphi_y\|_{L^2_{xy}} + \|\varphi\|_{L^2_{xy}}^2\, .
        \end{align*}
        After combining the similar terms, we obtain
        \begin{align}\label{5.7} 
            \|u(t)\|_{L^2_{xy}}^2 \le & 16 \|u_y\|_{L^2_{xyT}}^2 + 4 \|u_y(t)\|^2_{L^2_{xy}} - 8 \lambda \left(1-\dfrac{2}{p}\right) \|u\|_{L^p_{xyT}}^p \nonumber \\
            & \quad + \, 4 \left\|(h_1)_t\right\|_{L^2_{xT}} \|h_1\|_{L^2_{xT}} + 4 \left\|(h_2)_t\right\|_{L^2_{xT}} \|h_2\|_{L^2_{xT}} - \dfrac{8\lambda}{p} \|h_2\|_{L^p_{xT}}^p - \dfrac{8\lambda}{p} \|h_1\|_{L^p_{xT}}^p \nonumber \\
            &\qquad  + \, 10 \|h_1\|_{L^2_{xT}}^2 + 10 \|h_2\|_{L^2_{xT}}^2 + 4 \|(h_1)_x\|_{L^2_{xT}}^2 + 2 \|(h_2)_x\|_{L^2_{xT}}^2 \nonumber \\
            & \qquad \quad + \, 4 \|\varphi\|_{L^2_{xy}} \|\varphi_y\|_{L^2_{xy}} + 2 \|\varphi\|_{L^2_{xy}}^2\, .
        \end{align}
        Applying the strategy used in \cite{Ran_NLSE_2d_uhp}, we integrate the identity (\ref{5.2}) with respect to $x$, $y$ and $t$ and use (\ref{5.6}) to deduce
        \begin{align*}
            \|u_x(t)& \|_{L^2_{xy}}^2 + \|u_y(t)\|^2_{L^2_{xy}} - \dfrac{2\lambda}{p} \|u(t)\|^p_{L^p_{xy}} \\
            &= \|\varphi_x\|_{L^2_{xy}}^2 + \|\varphi_y\|_{L^2_{xy}}^2 - \dfrac{2\lambda}{p} \|\varphi\|_{L^p_{xy}}^p \\
            &\quad  - 2 \text{Re} \int_{-\infty}^{\infty} \int_0^t \left(\overline{h_1}\right)_t u_y(x,0,\tau) \, d\tau \, dx + 2 \text{Re} \int_{-\infty}^{\infty} \int_0^t \left(\overline{h_2}\right)_t u_y(x,1,\tau) \, d\tau \, dx \\
            &\le \dfrac{1}{128} \int_{-\infty}^{\infty} \int_0^t \left(|u_y(x,0,\tau)|^2 + |u_y(x,1,\tau)|^2 \right) \, d\tau \, dx + 128 \left\|(h_1)_t\right\|^2_{L^2_{xT}} + 128 \left\|(h_2)_t\right\|^2_{L^2_{xT}} \\
            & \quad + \|\varphi_x\|_{L^2_{xy}}^2 + \|\varphi_y\|_{L^2_{xy}}^2 - \dfrac{2\lambda}{p} \|\varphi\|_{L^p_{xy}}^p \\
            &\le \dfrac{1}{16} \|u_y\|_{L^2_{xyT}}^2 + \dfrac{1}{64} \|u(t)\|_{L^2_{xy}} \|u_y(t)\|_{L^2_{xy}} - \dfrac{\lambda}{32} \left(1-\dfrac{2}{p}\right) \|u\|_{L^p_{xyT}}^p \\
            & \quad + c \left(\left\|(h_1)_t\right\|_{L^2_{xT}} \|h_1\|_{L^2_{xT}} + \left\|(h_2)_t\right\|_{L^2_{xT}} \|h_2\|_{L^2_{xT}} - \dfrac{\lambda}{4p} \|h_2\|_{L^p_{xT}}^p - \dfrac{\lambda}{4p} \|h_1\|_{L^p_{xT}}^p \right. \nonumber \\
            & \qquad + \|h_1\|_{L^2_{xT}}^2 + \|h_2\|_{L^2_{xT}}^2 + \|(h_1)_x\|_{L^2_{xT}}^2 + \|(h_2)_x\|_{L^2_{xT}}^2 + \|\varphi\|_{L^2_{xy}} \|\varphi_y\|_{L^2_{xy}} \\
            & \quad\qquad  \left. + \left\|(h_1)_t\right\|^2_{L^2_{xT}} + \left\|(h_2)_t\right\|^2_{L^2_{xT}} + \|\varphi_x\|_{L^2_{xy}}^2 + \|\varphi_y\|_{L^2_{xy}}^2 - \dfrac{2\lambda}{p} \|\varphi\|_{L^p_{xy}}^p \right)
        \end{align*}
        for some $c>0$.
        Denote $C=C(\|\varphi\|_{H^1_{xy}})$ and $C(t)=C(\|h_j\|_{H^1_{xT}})$ as functions depending upon $\|\varphi\|_{H^1_{xy}}$ and $\|h_j\|_{H^1_{xT}}$ ($j=1$, $2$), respectively; in particular, $C=0$ if $\|\varphi\|_{H^1_{xy}}=0$ and $C(t)=0$ if $\|h_j\|_{H^1_{xT}}=0$.
        Then, (\ref{5.7}) yields
        \begin{align*}
            \|u_x(t) &\|_{L^2_{xy}}^2 + \|u_y(t)\|^2_{L^2_{xy}} - \dfrac{2\lambda}{p} \|u(t)\|^p_{L^p_{xy}} \\
            &\le \dfrac{1}{16} \|u_y\|_{L^2_{xyT}}^2 + \dfrac{1}{64} \|u(t)\|_{L^2_{xy}} \|u_y(t)\|_{L^2_{xy}} - \dfrac{\lambda}{32} \left(1-\dfrac{2}{p}\right) \|u\|_{L^p_{xyT}}^p + C(t) + C \\
            &\le \dfrac{1}{16} \|u_y\|_{L^2_{xyT}}^2 - \dfrac{\lambda}{32} \left(1-\dfrac{2}{p}\right) \|u\|_{L^p_{xyT}}^p + C(t) + C \\
            & \quad + \dfrac{1}{64} \|u_y(t)\|_{L^2_{xy}} \left\{16 \|u_y\|_{L^2_{xyT}}^2 + 4 \|u_y(t)\|^2_{L^2_{xy}} - 8 \lambda \left(1-\dfrac{2}{p}\right) \|u\|_{L^p_{xyT}}^p + C(t) + C \right\}^{\frac{1}{2}} \\
            &\le \dfrac{1}{16} \|u_y\|_{L^2_{xyT}}^2 - \dfrac{\lambda}{32} \left(1-\dfrac{2}{p}\right) \|u\|_{L^p_{xyT}}^p + C(t) + C \\
            & \quad + \dfrac{1}{64} \|u_y(t)\|_{L^2_{xy}} \left\{4 \|u_y\|_{L^2_{xyT}} + 2 \|u_y(t)\|_{L^2_{xy}} + \sqrt{8 |\lambda| \left(1-\dfrac{2}{p}\right)} \|u\|_{L^p_{xyT}}^{\frac{p}{2}} + C(t) + C \right\} \\
            &= \dfrac{1}{16} \|u_y\|_{L^2_{xyT}}^2 - \dfrac{\lambda}{32} \left(1-\dfrac{2}{p}\right) \|u\|_{L^p_{xyT}}^p + \dfrac{1}{16} \|u_y(t)\|_{L^2_{xy}} \|u_y\|_{L^2_{xyT}} + C(t) + C \\
            & \quad + \dfrac{1}{32} \|u_y(t)\|^2_{L^2_{xy}} + \sqrt{\dfrac{|\lambda|}{8^3} \left(1-\dfrac{2}{p}\right)} \|u_y(t)\|_{L^2_{xy}} \|u\|_{L^p_{xyT}}^{\frac{p}{2}} + \|u_y(t)\|_{L^2_{xy}} (C(t) + C) \\
            &\le \dfrac{1}{16} \|u_y\|_{L^2_{xyT}}^2 - \dfrac{\lambda}{32} \left(1-\dfrac{2}{p}\right) \|u\|_{L^p_{xyT}}^p + \dfrac{1}{32} \|u_y(t)\|_{L^2_{xy}}^2 + \dfrac{1}{32} \|u_y\|_{L^2_{xyT}}^2 + C(t) + C \\
            & \quad + \dfrac{1}{32} \|u_y(t)\|^2_{L^2_{xy}} + \dfrac{1}{64} \|u_y(t)\|_{L^2_{xy}}^2 + \dfrac{|\lambda|}{32} \left(1-\dfrac{2}{p}\right) \|u\|_{L^p_{xyT}}^p + \dfrac{1}{32} \|u_y(t)\|_{L^2_{xy}}^2 \\
            &\le \dfrac{1}{8} \|u_y(t)\|_{L^2_{xy}}^2 + \dfrac{1}{8} \|u_y\|_{L^2_{xyT}}^2 - \dfrac{\lambda}{16} \left(1-\dfrac{2}{p}\right) \|u\|_{L^p_{xyT}}^p + C(t) + C
        \end{align*}
        or
        \begin{equation*}
            \|u_x(t)\|_{L^2_{xy}}^2 + \|u_y(t)\|^2_{L^2_{xy}} - \dfrac{8\lambda}{3p} \|u(t)\|^p_{L^p_{xy}} \le \dfrac{1}{6} \|u_y\|_{L^2_{xyT}}^2 - \dfrac{\lambda}{12} \left(1-\dfrac{2}{p}\right) \|u\|_{L^p_{xyT}}^p + C(t) + C \, .
        \end{equation*}
        Moreover, since $\lambda<0$, we have
        \begin{align*}
            & \|u_x(t)\|_{L^2_{xy}}^2 + \|u_y(t)\|^2_{L^2_{xy}} - \dfrac{8\lambda}{3p} \|u(t)\|^p_{L^p_{xy}} \le \dfrac{1}{6} \left(\|u_x\|_{L^2_{xyT}}^2 + \|u_y\|_{L^2_{xyT}}^2 - \dfrac{8\lambda}{3p} \|u\|_{L^p_{xyT}}^p \right) + C(t) + C \\
            &= \dfrac{1}{6} \int_0^t \left(\|u_x(\tau)\|_{L^2_{xy}}^2 + \|u_y(\tau)\|_{L^2_{xy}}^2 - \dfrac{8\lambda}{3p} \|u(\tau)\|_{L^p_{xy}}^p \right) \, d\tau + C(t) + C\, .
        \end{align*}
        By the Gronwall's inequality, we obtain
        \begin{align*}
            & \|u_x(t)\|_{L^2_{xy}}^2 + \|u_y(t)\|^2_{L^2_{xy}} - \dfrac{8\lambda}{3p} \|u(t)\|^p_{L^p_{xy}} \\
            &\qquad \le \left(C(t) + C\right) \cdot \exp\left(\int_0^t \frac{1}{6} \, d\tau \right) = \left(C(t) + C\right) \exp\left(\frac{t}{6} \right) := \psi_1
        \end{align*}
        where $\psi_1$ is an increasing function of $\|\varphi\|_{H^1_{xy}}$ and $\|h_j\|_{H^1_{xT}}, j= 1,2,$ and $\psi_1=0$ if and only if $\|\varphi\|_{H^1_{xy}},$  $\|h_j\|_{H^1_{xT}}, j =1,2 $ are zero. Thus
        \begin{equation*}
            \|u_x(t)\|_{L^2_{xy}}^2 + \|u_y(t)\|^2_{L^2_{xy}} \le \psi_1\, .
        \end{equation*}
        Note that for all the terms with $\|\cdot\|_{L^r}$, one can bound them with $H^1$-norms according to the Sobolev embedding theorem for a domain of dimension $2$. Thus, it is clear that $\|u(t)\|_{H^1_{xy}}$ is uniformly bounded for any given $T$ if $\lambda<0$.

        For $\lambda>0$, analogous to the previous argument, (\ref{5.7}) implies
        \begin{align*}
            \| &  u_x(t)\|_{L^2_{xy}}^2 + \|u_y(t)\|^2_{L^2_{xy}} - \dfrac{2\lambda}{p} \|u(t)\|^p_{L^p_{xy}} \\
            &\le \dfrac{1}{16} \|u_y\|_{L^2_{xyT}}^2 + \dfrac{1}{64} \|u(t)\|_{L^2_{xy}} \|u_y(t)\|_{L^2_{xy}} - \dfrac{\lambda}{32} \left(1-\dfrac{2}{p}\right) \|u\|_{L^p_{xyT}}^p + C(t) + C \\
            &\le \dfrac{1}{16} \|u_y\|_{L^2_{xyT}}^2 + \dfrac{1}{64} \|u_y(t)\|_{L^2_{xy}} \left\{16 \|u_y\|_{L^2_{xyT}}^2 + 4 \|u_y(t)\|^2_{L^2_{xy}} + C(t) + C \right\}^{\frac{1}{2}} + C(t) + C \\
            &\le \dfrac{1}{16} \|u_y\|_{L^2_{xyT}}^2 + \dfrac{1}{64} \|u_y(t)\|_{L^2_{xy}} \left\{4 \|u_y\|_{L^2_{xyT}} + 2 \|u_y(t)\|_{L^2_{xy}} + C(t) + C \right\} + C(t) + C \\
            &= \dfrac{1}{16} \|u_y\|_{L^2_{xyT}}^2 + \dfrac{1}{16} \|u_y(t)\|_{L^2_{xy}} \|u_y\|_{L^2_{xyT}} + \dfrac{1}{32} \|u_y(t)\|^2_{L^2_{xy}} + \|u_y(t)\|_{L^2_{xy}} (C(t) + C) + C(t) + C \\
            &\le \dfrac{1}{16} \|u_y\|_{L^2_{xyT}}^2 + \dfrac{1}{32} \|u_y(t)\|_{L^2_{xy}}^2 + \dfrac{1}{32} \|u_y\|_{L^2_{xyT}}^2 + \dfrac{1}{32} \|u_y(t)\|^2_{L^2_{xy}} + \dfrac{1}{32} \|u_y(t)\|_{L^2_{xy}}^2 + C(t) + C \\
            &\le \dfrac{1}{8} \|u_y(t)\|_{L^2_{xy}}^2 + \dfrac{1}{8} \|u_y\|_{L^2_{xyT}}^2 + C(t) + C\, .
        \end{align*}
        By Gagliardo-Nirenberg inequality and H\"older's inequality,
        \begin{align*}
             \|u_x&(t)\|_{L^2_{xy}}^2 + \|u_y(t)\|^2_{L^2_{xy}} \\
            &\le \dfrac{1}{8} \|u_y(t)\|_{L^2_{xy}}^2 + \dfrac{1}{8} \|u_y\|_{L^2_{xyT}}^2 + \dfrac{2\lambda}{p} \|u(t)\|^p_{L^p_{xy}} + C(t) + C \\
            &\le \dfrac{1}{8} \|u_y(t)\|_{L^2_{xy}}^2 + \dfrac{1}{8} \|u_y\|_{L^2_{xyT}}^2 + \dfrac{2\lambda}{p} \left(\|u_x(t)\|_{L^2_{xy}} + \|u_y(t)\|_{L^2_{xy}}\right)^{p-2} \cdot \|u(t)\|_{L^2_{xy}}^2 + C(t) + C\, .
        \end{align*}
        Also, (\ref{5.1}) gives
        \begin{align*}
             \|u&(t)\|_{L^2_{xy}}^2 = \int_{-\infty}^{\infty} \int_0^1 |u(x,y,t)|^2 \, dy \, dx \\
            &= \int_{-\infty}^{\infty} \int_0^1 |u(x,y,0)|^2 \, dy \, dx + \int_{-\infty}^{\infty} \int_0^1 \int_0^t \left(|u(x,y,\tau)|^2\right)_t \, d\tau \, dy \, dx \\
            &= \|\varphi\|_{L^2_{xy}}^2 - 2 \text{Im} \int_{-\infty}^{\infty} \int_0^1 \int_0^t \left[(u_x(x,y,\tau) \overline{u}(x,y,\tau))_x + (u_y(x,y,\tau) \overline{u}(x,y,\tau))_y\right] \, d\tau \, dy \, dx \\
            &= \|\varphi\|_{L^2_{xy}}^2 + 2 \text{Im} \int_{-\infty}^{\infty} \int_0^t  u_y(x,0,\tau) \overline{u}(x,0,\tau) \, d\tau \, dx - 2 \text{Im} \int_{-\infty}^{\infty} \int_0^t u_y(x,1,\tau) \overline{u}(x,1,\tau)  \, d\tau \, dx \\
            &\le \left(\int_{-\infty}^{\infty} \int_0^t \left(|u_y(x,1,\tau)|^2 + |u_y(x,0,\tau)|^2 \right) \, d\tau \, dx\right)^{\frac{1}{2}} \cdot K(t) + C\, ,
        \end{align*}
        where $K(t) = c \max \{\|h_1\|_{L^2_{xT}}, \|h_2\|_{L^2_{xT}}\}$ with a possible constant $c>0$. Then, plug (\ref{5.6}) into the inequality above to obtain
        \begin{align*}
             \|u(t)\|_{L^2_{xy}}^2 & \le \|u_y\|_{L^2_{xyT}} \cdot K(t) + \|u(t)\|_{L^2_{xy}}^{\frac{1}{2}} \|u_y(t)\|_{L^2_{xy}}^{\frac{1}{2}} \cdot K(t) + C(t) + C \\
            &\le \|u_y\|_{L^2_{xyT}} \cdot K(t) + \dfrac{1}{4} \|u(t)\|_{L^2_{xy}}^2 + \dfrac{3}{4} \|u_y(t)\|_{L^2_{xy}}^{\frac{2}{3}} \cdot K(t)^{\frac{4}{3}} + C(t) + C,
        \end{align*}
        i.e.,
        \begin{equation*}\label{NLSE_IB_2d_stp_glb_ll2_K_est}
            \|u(t)\|_{L^2_{xy}}^2 \le \|u_y\|_{L^2_{xyT}} \cdot K(t) + \|u_y(t)\|_{L^2_{xy}}^{\frac{2}{3}} \cdot K(t)^{\frac{4}{3}} + C(t) + C \, .
        \end{equation*}
        We substitute the revised estimate on the $L^{\infty}(L^2$)-norm of $u$ into the inequality for derivatives of $u$,
        \begin{align*}
             \| & u_x(t)\|_{L^2_{xy}}^2 + \|u_y(t)\|^2_{L^2_{xy}} \\
            &\le \dfrac{1}{8} \|u_y(t)\|_{L^2_{xy}}^2 + \dfrac{1}{8} \|u_y\|_{L^2_{xyT}}^2 + \left(\|u_x(t)\|_{L^2_{xy}} + \|u_y(t)\|_{L^2_{xy}}\right)^{p-2} \cdot \|u(t)\|_{L^2_{xy}}^2 + C(t) + C \\
            &\le \dfrac{1}{8} \|u_y(t)\|_{L^2_{xy}}^2 + \dfrac{1}{8} \|u_y\|_{L^2_{xyT}}^2 + C(t) + C \\
            & \quad + \dfrac{2\lambda}{p} \left(\|u_x(t)\|_{L^2_{xy}}^{p-2} + \|u_y(t)\|_{L^2_{xy}}^{p-2}\right) \cdot \left(\|u_y\|_{L^2_{xyT}} \cdot K(t) + \|u_y(t)\|_{L^2_{xy}}^{\frac{2}{3}} \cdot K(t)^{\frac{4}{3}} + C(t) + C \right) \\
            &= \dfrac{1}{8} \|u_y(t)\|_{L^2_{xy}}^2 + \dfrac{1}{8} \|u_y\|_{L^2_{xyT}}^2 + C(t) + C + \|u_x(t)\|_{L^2_{xy}}^{p-2} \|u_y(t)\|_{L^2_{xy}}^{\frac{2}{3}} \cdot K(t)^{\frac{4}{3}} \\
            & \quad + \|u_y(t)\|_{L^2_{xy}}^{\frac{3p-4}{3}} \cdot K(t)^{\frac{4}{3}} + \left(\|u_x(t)\|_{L^2_{xy}}^{p-2} + \|u_y(t)\|_{L^2_{xy}}^{p-2}\right) \|u_y\|_{L^2_{xyT}} \cdot K(t) \\
            & \qquad + \left(\|u_x(t)\|_{L^2_{xy}}^{p-2} + \|u_y(t)\|_{L^2_{xy}}^{p-2}\right) (C(t) + C)\, .
        \end{align*}
        It turns out that the uniform boundedness can be derived only when $p\le3$ in this case. If $p<3$, the above inequality gives
        \begin{equation*}
            \|u_x(t)\|_{L^2_{xy}}^2 + \|u_y(t)\|^2_{L^2_{xy}} \le c \int_0^t \left(\|u_x(\tau)\|_{L^2_{xy}}^2 + \|u_y(\tau)\|_{L^2_{xy}}^2 \right) \, d\tau + C(t) + C\, .
        \end{equation*}
        By the Gronwall's inequality
        \begin{equation*}
            \|u_x(t)\|_{L^2_{xy}}^2 + \|u_y(t)\|^2_{L^2_{xy}} \le (C(t) + C) e^{ct} \, .
        \end{equation*}
        If $p=3$, then
        \begin{equation*}
            (1-K(t)) \left(\|u_x(t)\|_{L^2_{xy}}^2 + \|u_y(t)\|^2_{L^2_{xy}}\right) \le c \int_0^t \left(\|u_x(\tau)\|_{L^2_{xy}}^2 + \|u_y(\tau)\|_{L^2_{xy}}^2 \right) \, d\tau + C(t) + C\, .
        \end{equation*}
        Since $K(t) = c \max \{\|h_1\|_{L^2_{xT}(\R \times [0,T])}, \|h_2\|_{L^2_{xT}(\R \times [0,T])} \} < \infty$. We can partition $[0,T]$ into a finite number of subintervals $(t_{j-1},t_j)$ for $j=1$, $\cdots$, $m$ and $m \in \N$ with $\sup_j |t_j-t_{j-1}| \sim \delta$ so that, on each interval, $\|h_1\|_{L^2_{xT}(\R \times [t_{j-1},t_j]}$ and $\|h_2\|_{L^2_{xT}(\R \times [t_{j-1},t_j]} < (1/{2c})$. Then, starting from $[0,\delta]$, we move over one subinterval and use $u(t_j)$ as the initial value for a new IBVP on $(t_j,t_{j+1})$. At the end, $\sup_{t\in[0,T]} \|u(t)\|_{H^1(\R \times [0,1]))}$ is uniformly bounded by a function depending upon the initial and boundary data only. In particular, we let $\psi_2 = (C + C(t)) \left((1/2) - K(t)\right)^{-1}$.
        Hence, if we let $\psi=\psi_1$ when $p\ge3$ for $\lambda<0$ and $\psi=\psi_2$ when $p=3$ as $\lambda>0$, $\sup_{t\in[0,T]} \|u(t)\|_{H^1(\R \times [0,1]))}$ is uniformly bounded by $\psi$. The proof is finished.
    \end{proof}

     Thus, \emph{Theorem~\ref{thm_4.3}} and \emph{Propositions~\ref{thm_4.4}} and \emph{\ref{prop_5.2}} imply the following theorem.
    \begin{thm}
        Assume that either $p\ge 3$ and $\lambda<0$ or $p=3$ and $\lambda>0$. Then, (\ref{2.1}) is globally well-posed in $H^1(\R\times[0,1])$ if $\varphi\in H^1(\R \times [0,1])$ and $h_j\in H^1_{t\text{-loc}} \left(\R; \, L^2_x(\R)\right) \cap L^2_{t\text{-loc}} \left(\R; \, H^2_x(\R)\right)$ for $j=1$, $2$.
    \end{thm}

    \bigskip
    \noindent{\bf Acknowledgements.} The authors were partly supported by National Science Foundation under grant No. DMS-1210979. We sincerely thank an anonymous referee for
a careful reading of the manuscript and many helpful comments, corrections and suggestions.


\end{document}